\documentclass[11pt]{amsart}

\usepackage{amssymb,latexsym}
\usepackage{amsthm}
\usepackage{amsmath}
\usepackage{amsfonts}
\usepackage{mathrsfs}
\usepackage{graphicx}
\usepackage[all]{xy}
\usepackage{chngcntr}
\usepackage{fancyhdr}
\usepackage[top=1in,bottom=1in,left=1.25in,right=1.25in]{geometry}
\usepackage[usenames,dvipsnames,pdftex]{xcolor}
\usepackage{tikz,ifthen}
\usepackage{color}
\usepackage{pgfplots}
\usepackage{tikz}
\usetikzlibrary{matrix,arrows,decorations.pathmorphing}

\newtheorem{thm}{Theorem}[section]
\newtheorem{lem}[thm]{Lemma}
\newtheorem{prop}[thm]{Proposition}

\newtheorem{rem}[thm]{Remark}
\newtheorem{defn}[thm]{Definition}

\def\Q{\mathbb{Q}}
\def\F{\mathbb{F}}
\def\R{\mathbb{R}}
\def\Z{\mathbb{Z}}
\def\A{\mathbb{A}}
\def\C{\mathbb{C}}
\def\G{\mathbb{G}}

\def\U{\mathcal{U}}

\def\Fl{\mathcal{F}\ell}
\def\Fll{\mathcal{F}\ell^{\times}}

\def\Io{\mathcal{I}wa}

\def\Ol{\mathscr{O}}
\def\Op{\mathscr{O}_{\C_p}}

\def\W{\mathcal{W}}
\def\X{\mathcal{X}}

\def\V{\mathcal{V}}
\def\Vb{\mathscr{V}}
\def\Vc{\mathbf{V}}
\def\tt{\textbf{TI}}
\def\ka{\kappa}
\def\z{\mathfrak{z}}
\def\Y{\mathcal{Y}}
\def\s{\mathfrak{s}}
\def\Igb{\mathcal{I}\mathcal{G}\mathcal{B}}
\def\Fg{\mathscr{F}}
\def\iw{\mathcal{I}\mathcal{W}}
\def\waip{\underline{\omega}_{w,v}^{\ka_{\U},AIP}}

\usepackage[colorlinks, citecolor=blue]{hyperref}

\begin{document}

\title{Perfectoid  unitary Shimura varieties and $p$-adic Eichler-Shimura map I}
\author{Ruishen Zhao}
\date{}
 
\address{Yau Mathematical Science Center\\
	Tsinghua University\\
	  Jingzhai, Haidian District\\
	Beijing 100084, China}
\email{zrs13@tsinghua.org.cn}	
\renewcommand\thefootnote{}

\renewcommand{\thefootnote}{\arabic{footnote}}

\begin{abstract}
We investigate $p$-adic automorphic forms on unitary groups through the geometry of infinite-level unitary Shimura varieties and the Hodge-Tate period map. We first develop a perfectoid construction of overconvergent automorphic forms.  Building on this, we  establish a canonical overconvergent Eichler-Shimura map linking overconvergent cohomology to these $p$-adic automorphic forms. This map induces a comparison between the corresponding coherent sheaves on the eigenvariety, with applications to the study of its geometry and to $p$-adic $L$-functions.
 
\end{abstract}

\maketitle
\setcounter{tocdepth}{2}

\tableofcontents

\section{Introduction}

This series papers  aims to develop a $p$-adic Eichler-Shimura comparison theory for unitary Shimura varieties. The present paper focuses on compact unitary Shimura varieties of signature $(1,n)$ (often referred to as being of Kottwitz-Harris-Taylor type, see \cite{ht2001}). Our  main result is  the construction of a canonical comparison map over the eigenvariety. While this serves as a prototypical case, subsequent paper will further explore higher Coleman theory and address the  general signature $(a, b)$ together with the non-compact setting.

 We begin by briefly recalling some background.

 Shimura varieties are indispensable objects in Langlands program and related topics (e.g. see \cite{ht2001} and \cite{scholze2015torsion}), as their cohomology provides a rich source of Galois representations and automorphic forms. A basic problem is to  compare its etale cohomology and coherent cohomology. Eichler and Shimura established such a comparison for modular curves over $\C$. Later in \cite{faltingschai} Faltings generalized this to locally symmetric spaces via the dual BGG spectral sequence, a far-reaching extension of  complex Hodge theory. Moreover, in his seminal work  \cite{faltings87}, Faltings studied Hodge-Tate structure of modular forms and  established a $p$-adic analogue of Eichler-Shimura decomposition for modular curves over $\Q_p$, which is  essential  in $p$-adic Hodge theory.

On the other hand, since the pioneering contributions of Serre, Katz, Hida, Mazur, and Coleman, the theory of $p$-adic automorphic forms and eigenvarieties has been established and rapidly developed, focusing on the $p$-adic interpolation of classical automorphic forms.   It is now well-understood that both the etale and coherent cohomology of Shimura varieties admit such $p$-adic deformations, as explored in \cite{hansen2017universal} and \cite{AIP2015} respectively. A natural and further topic is to $p$-adically interpolate the Eichler-Shimura comparison maps. We refer to such a theme of $p$-adic family comparison as \textbf{$p$-adic Eichler-Shimura theory}.

Our paper establishes such a result for unitary Shimura varieties.   Let $E/\Q$ be an imaginary quadratic extension. Let  $G/\Q$ be the unitary group of signature of $(1,n)$ defining the Shimura variety $Sh_{G}$ as studied by Harris-Taylor \cite{ht2001}. Fix a prime $p \geq 5$ which splits in $E$ and an isomorphism $G_{\Q_p}\cong GL_{n+1}\times \mathbb{G}_m$.  As the similitude factor $\mathbb{G}_{m}$ plays no role in our analysis, we often view $G_{\Q_p}$ as $GL_{n+1}$ for simplicity. Let $P_{\mu}$ denote the parabolic subgroup of $G(\Q_p)$ associated to the Hodge cocharacter $\mu$ (see section 2.1 in \cite{caraianischolze2017}) and its Levi subgroup $M_{\mu}$ is isomorphic to $GL_1 \times GL_n$. Let $M$ denote the $GL_n$-factor inside $M_{\mu}$. We consider "\textit{reduced}" weight, i.e. characters for $T_{M}$ (maximal split torus of $M$). This restriction is mild as central weight twists for $G(\Q_p)$ won't influence the geometry. Further take  compatible pinning down data $(T,B)$ for $G(\Q_p)$ and $(T_M,B_M)$ for $M(\Q_p)$. Let $Iw$ denote  the corresponding Iwahori subgroup inside $G(\Z_p)$ and $Iw^{+}$ denote the strict Iwahori subgroup (preimage of  $T(\F_p)$ under modulo $p$). Let $X_{Iw^+}$ denote the Shimura variety with strict Iwahori level subgroup at $p$.  For a weight $(R_{\U},\ka_{\U})$ (see section \ref{weight space}), where $R_{\U}$ is a suitable topological ring (e.g. affinoid algebra) and $\ka_{\U}:T_{M}(\Z_p)\rightarrow R_{\U}^{\times}$ is a continuous map, suppose $\ka_{\U}$ is $r$-analytic ($r \in \Q_{>0}$). let $\mathfrak{OC}^{r}_{\ka_{\U},\C_p}$ denote the \textbf{overconvergent cohomology} group (at strict Iwahori level), which is $p$-adic interpolation of classical etale cohomology $H^{n}(X_{Iw^+},\mathbf{V}_{\ka})$. Further take $w \in \Q_{>0}$ and let $M^{\ka_{\U}+n+1}_{Iw^+,w}$ denote the space of $w$-perfectoid automorphic forms (same as a suitable\textbf{ overconvergent automorphic forms}) at level $Iw^+$ with weight $\ka_{\U}+n+1$ (shift $\ka_{\U}$ by $n+1$ due to Kodaira-Spencer isomorphism, see proposition \ref{proposition proetale etale}), which is $p$-adic interpolation of classical automorphic forms (coherent cohomology $H^0(X_{Iw^+},\underline{\omega}^{\ka})$).

Our main result can be summarized as follow (theorem \ref{overconvergent es}, \ref{thm p adic and classical es} and \ref{es over eigenvariety}):

\begin{thm}

There is a Hecke and $Gal_{\Q_p}$-equivariant \textbf{overconvergent Eichler-Shimura map} of weight $\ka_{\U}$ \[ES_{\ka_{\U}}: \mathfrak{OC}^{r}_{\ka_{\U},\C_p} \rightarrow  M^{\ka_{\U}+n+1}_{Iw^+,w}(-n),\] which factors through the classical Eichler-Shimura map at classical weights. This map is  \textbf{functorial} in the  weight $\ka_{\U}$ and glues to a comparison map \[\mathcal{ES}:\mathscr{OC}^{\dag} \rightarrow \mathscr{F}^{\dag}_{Iw^+}(-n)\] between corresponding coherent sheaves over the $n$-dimensional eigenvariety $\mathcal{E}$.

\end{thm}

\begin{rem}

We remark on other results for $p$-adic Eichler-Shimura theory in the literatures:

(1) In \cite{AIS15} F. Andreatta, A. Iovita and G. Stevens established such $p$-adic family comparison theory for modular curves and recently Andreatta-Iovita upgraded this result into "overconvergent de Rham" (via period ring $\mathbb{B}_{dR}$) setting in \cite{Andreatta-Iovita20}.

(2) In \cite{chj}, P. Chojecki, D. Hansen and C. Johansson developed  such a comparison map for  Shimura curves. Later in \cite{Juan20curve} J. Camargo complemented their result by further relating overconvergent modular symbols with higher Coleman theory.  Moreover, in \cite{drw} H. Diao, G. Rosso and J. Wu generalized Chojecki-Hansen-Johansson's method into $GSp_{2g}$ case  and in \cite{diao2025overconvergent} they further studied higher Coleman theory (by \cite{bp2021higher})  in the setting of $GSp_4$.  The method by Chojecki-Hansen-Johansson is novel. They used infinite level Shimura varieties and pro-etale topology, which provides a new perspective. Their comparison map is more explicit and they can deal with more general weights than \cite{AIS15}, which enables them to get the comparison map over the full eigenvariety.  Our paper is inspired by this method. Compared to $GL_2$ case in \cite{chj}, the higher rank unitary group setting introduces  additional complexity. There are also subtle difference-particularly in the representation theory-between Siegel case in \cite{drw} and our unitary case (see section \ref{classical es} for a brief discussion). In a subsequent  paper of this series I will also further explore higher Coleman theory (as in \cite{diao2025overconvergent}) in our setting.

(3) More recently, Pan Lue's remarkable work \cite{pan2022} introduced geometric Sen theory to establish a different  $p$-adic Eichler-Shimura theory for modular curves. Building on his significant methods,  Boxer, Calegari, Gee and Pilloni subsequently  obtained general results in this direction \cite{bcgp2025}. Nevertheless, their work \textbf{differs} from ours in several key aspects.  First, they use (locally analytic part) completed cohomology whereas   we use overconvergent cohomology via analytic distribution.  Second, their method is \textbf{Lie-theoretic}, relying on Lie algebras and differential operators;  in contrast, our approach is purely \textbf{group-theoretic}. Consequently, their notion of weight is defined on the Lie algebra, while ours is always a character of the group.   An important consequence is that their results require the weight to be $p$-adically non-Liouville (see Definition 2.3.25 of \cite{bcgp2025}), a restriction we do not need. Third,    to obtain $p$-adic family decomposition (e.g split the relevant  spectral sequence),   they impose an ordinarity (or small slope) assumption, whereas in the next paper of this series we plan to work around "nice" points of the eigenvariety, as in \cite{diao2025overconvergent}. Finally, the motivations and applications also differ: their work is aimed at modularity results, while ours is directed toward the study of eigenvarieties and $p$-adic $L$-functions, which aligns more closely with group-theoretic methods. On the other hand, it would be interesting to further explore the connections between these two approaches to $p$-adic Eichler-Shimura theory.

\end{rem}

The proof of the above theorem is divided into two parts, following the strategy in \cite{chj} and \cite{drw} but adapted to the case of unitary Shimura varieties. In the first part we provide  a \textbf{perfectoid construction} of overconvergent automorphic forms, which is of independent interest. In the second part we investigate overconvergent cohomology and overconvergent automorphic forms through \textbf{pro-etale topology}, and explicitly establish a canonical overconvergent Eichler-Shimura map relating them.

Now we introduce the first part. The main ingredient is the infinite level Shimura varieties together with the  Hodge-Tate period map $\pi_{HT}$ introduced in the breakthrough \cite{scholze2015torsion} of Peter Scholze.

 The perfectoid construction   also provides a $p$-adic analogue of  complex analytic automorphic forms. See the nice introduction (in particular table 1.1) of \cite{chj} for more details.  As a warm up, let's recall the \textit{analytic} definition of modular forms of weight $k$:
 
$\bullet$ $f$ is a holomorphic function on the upper half-plane $\mathfrak{h}$ of moderate growth, satisfying the transformation law,  
$f(\frac{az+b}{cz+d})=(cz+d)^k f(z)$ for all $\gamma \in \Gamma$ (arithmetic subgroup for $GL_{2}(\Z)$, e.g. $\Gamma_1(N)$). 

 Equivalently $f$ is also realized in   in $H^0(X_{\Gamma}(\C), \omega^{k})$ ($X_{\Gamma}(\C)$ is the compactified modular curve with level $\Gamma$ and $ \omega^{k}$ is the weight $k$ modular sheaf). To see this equivalence,  consider the complex uniformization $\mathfrak{h} \rightarrow Y_{\Gamma}(\C)$ (the modular curve with level $\Gamma$), the line bundle $\omega^{k}$ can be suitably \textbf{trivialized} over $\mathfrak{h}$ which involves\textbf{ twisted} $\Gamma$-action via automorphic factor $J(z,\gamma)=cz+d$, then $H^0(X_{\Gamma}(\C), \omega^{k})$ will exactly correspond to analytic functions $f$ on $\mathfrak{h}$ satisfying transformation law (ignore the issue of cusps).
 
To establish $p$-adic analogue, we first consider the following diagram

\[\xymatrix{
\X_{\infty} \ar[r]^{\pi_{HT}} \ar[d]_{\pi_{K_p}} & \Fl \\
 \X_{K_p}
},\] where $\X_{K_p}$ is the adic space of unitary Shimura variety with level $K_p$ at $p$ and $\X_{\infty}$ is the infinite level Shimura variety, which is a perfectoid space (see \cite{scholze2015torsion}). The Hodge-Tate period map $\pi_{HT}$ relates the geometry of flag variety to the infinite level Shimura while the natural map $\X_{\infty}\rightarrow \X_{K_p}$ connects the infinite level with finite level.

The \textit{intuition} is that $\pi_{K_p}$ can be seen as a kind of $p$-adic uniformization (analogue of  complex uniformization). However, the $p$-adic picture is more complicated than complex setting. Automorphic bundle can't be trivialized even after pulling back to $\X_{\infty}$. This suggests to  working with suitable subspace of $\X_{\infty}$  where such a trivialization can be achieved--a perspective that fits well  with the theory of overconvergent automorphic forms, which live in a certain neighborhood of the ordinary locus rather than the entire Shimura variety $\X_{Iw}$.  

For $w \in \Q_{>0}$, we will introduce $w$-ordinary locus over $\Fl$, obtain $\X_{\infty,w}$ through pullback via $\pi_{HT}$ and $\X_{Iw,w}$ will be the image under $\pi_{Iw}$. These are the desired subspaces and $\X_{Iw,w}$ will be the definition domain for perfectoid automorphic forms.

More precisely, let $N_{\mu}^{opp}$ denote the unipotent part of the opposite parabolic subgroup $P_{\mu}^{opp}$ and the flag variety $\Fl$ is isomorphic to $P_{\mu}\backslash  G_{\Q_p}$. We further fix a suitable inclusion  \[\A^n \cong N_{\mu}^{opp} \hookrightarrow P_{\mu}\backslash  G_{\Q_p}=\Fl,\] and denote its image (an open subspace) inside $\Fl$ by $\Fl^{\times}$ and use this inclusion to set up coordinates $\mathbf{z}=(\mathbf{z}_1,\cdots,\mathbf{z}_n)$ (view as \textbf{column} vector) over $\Fl^{\times}$. Use the same notation for their corresponding adic space version and we define\[\Fll_{w}:=\{x \in \Fll: \max\limits_{i}\inf\limits_{a \in p \Z_p}{|\textbf{z}_i(x)-a|}\leq p^{-w}\}.\] The locus shrinks as $w$ increases. Through $\pi_{HT}$ and $\pi_{Iw}$ we further obtain $\X_{\infty,w}$ and $\X_{Iw,w}$.

The flag variety $\Fl$ is the projective space $\mathbb{P}^n$ parametrizing lines in $n+1$-dimensional vector space. There is  an exact sequence over the flag variety  \[0\rightarrow \mathscr{L} \rightarrow\mathscr{V} \rightarrow \mathscr{W}^D \rightarrow 0,\] where $\mathscr{L}$ is the tautological line bundle and $\mathscr{V}$ is the $n+1$-dimensional trivial vector bundle. Let $\{e_0,\cdots,e_n\}$ denote the standard basis for $\Q_p^{n+1}$, then $(e_1,\cdots,e_n)$ further produces sections $\mathbf{s}=(\mathbf{s}_1,\cdots,\mathbf{s}_n)$ (view as \textbf{row} vector), for  a matrix $g \in Iw \subset GL_{n+1}(\Z_p)$, suppose $g=\begin{pmatrix} g_a & g_b \\ g_c & g_d \end{pmatrix}$, where $g_a$ is $1 \times 1$ matrix, $g_b$ is $1 \times n$ matrix, $g_c$ is $n \times 1$-matrix and $g_d$ is $n \times n$ matrix. Then under our coordinates, over $\Fl_{w}^{\times}$ we have the following transformation rule (lemma \ref{transform law}) \[g^{*}(\mathbf{s})=\mathbf{s}(\mathbf{z}g_b+g_d).\] It suggests to introduce the analogue of automorphic factor \[J(\mathbf{z},g):=\mathbf{z} g_b+g_d.\] Moreover through pullback to $\X_{\infty,w}$, let $\z$ denote the coordinates and $\s$ denote the corresponding sections for $\underline{\omega}_{\infty}$, we have \[g^{*}(\s)=\s J(\z,g).\]

On the other hand, the $n$-dimensional vector bundle $\underline{\omega}_{\infty}$ is  isomorphic to the pullback via $\pi_{Iw}$ of  the $n$-dimensional Hodge bundle over $\X_{Iw}$ (a special example of theorem 2.1.3 in \cite{caraianischolze2017}).  It suggests to define perfectoid automorphic forms over $\X_{Iw,w}$ as certain analytic functions on $\X_{\infty,w}$ satisfying suitable transform law, which provides a $p$-adic analogue of complex analytic definition of automorphic forms. 

We will denote the sheaf of $w$-overconvergent perfectoid forms with weight $\ka_{\U}$ and level $Iw$ by $\underline{\omega}_{w}^{\ka_{\U}}$. It is  a sheaf of uniform $\C_p$-Banach algebras over $\X_{Iw,w}$. More precisely, let $Iw_M$ denote the Iwahori subgroup for $M$, and for weight $\ka_{\U}=(\ka_{\U,1},\cdots,\ka_{\U,n})$ set $\ka_{\U}^{\vee}= (-\ka_{\U,n},\cdots,-\ka_{\U,1})$. For any affinoid $\V \subset \X_{Iw,w}$ and let $\V_{\infty}$ denote the preimage inside $\X_{\infty,w}$, the section $\underline{\omega}_{w}^{\ka_{\U}}(\V)$ are those functions $f$ in $C_{\ka_{\U}}^{w-an}(Iw_M, \Ol_{\X_{\infty,w}}(\V_{\infty})\widehat{\otimes} R_{\U})$ satisfying \textbf{transformation law} \[g^{*}(f)=\rho_{\ka_{\U}}^{-1}(J(\z,g))(f), \forall g \in Iw.\] Here $C_{\ka_{\U}}^{w-an}(Iw_M, \Ol_{\X_{\infty,w}}(\V_{\infty})\widehat{\otimes} R_{\U})$ is certain $w$-analytic induction  (see section \ref{weight space}) on $Iw_M$, i.e. the right translation via $B_M(\Z_p)$ acts through the character $\ka_{\U}^{\vee}$ and the resulting \textbf{left} representation of $Iw_{M}$ via left translation is denoted by $\rho_{\ka_{\U}}$. In other words,   $C_{\ka_{\U}}^{w-an}(Iw_M, \Ol_{\X_{\infty,w}}(\V_{\infty})\widehat{\otimes} R_{\U})$ consists of functions on $\V_{\infty}$  valued in $C_{\ka_{\U}}^{w-an}(Iw_M,  \C_p\widehat{\otimes} R_{\U})$.

We introduce the \textbf{twisted} $Iw$-action on  $C_{\ka_{\U}}^{w-an}(Iw_M, \Ol_{\X_{\infty,w}}(\V_{\infty})\widehat{\otimes} R_{\U})$ by \[\alpha . f:=\rho_{\ka_{\U}}(J(\z,\alpha))\alpha^{*}(f).\] Equivalently we can view overconvergent automorphic forms  as those functions which are invariant under the twisted $Iw$-action.

The next task is to show that perfectoid automorphic forms are the same as overconvergent automorphic forms in the literature (e.g. \cite{shen2016}).  Previously, Andreatta-Iovita-Pilloni did such construction in their seminal work \cite{AIP2015} for Siegel Shimura varieties. Later  in \cite{shen2016} Xu Shen generalized their construction to  compact unitary Shimura varieties with signature $(1,n)\times (0,n+1)\times...\times(0,n+1)$. We will compare our construction with his result. 

Both constructions are heavy in notations, here I only sketch ideas. See section \ref{section3} for more details. It is recommended to read section \ref{section classical form} first, which provides a toy example for such a comparison via  describing  classical automorphic forms through  $\X_{\infty,w}$.  The \textit{intuition} is:

\textit{Interpret classical theory of canonical subgroup in term of the new framework via} $\pi_{HT}$.

Roughly speaking, in \cite{shen2016} Xu Shen used Hodge height to measure the "\textit{distance}" to ordinarity. For $v \in \Q \cap [0,1]$ he introduced $\X_{Iw}(v)$ (in his notation is $X(v)$ in section 3.2 of \cite{shen2016}), a neighborhood of the multiplicative ordinary locus inside $\X_{Iw}$. Suppose $v < \frac{1}{2p^{m-1}}$ ($m \in \Z_{>1}$) and take $w \in \Q \cap (0,m-v\frac{p^m}{p-1}]$, he constructed a natural map  \[\pi^{AIP}:\iw^+_w \rightarrow \X_{Iw}(v).\]  Then the sheaf of $w$-analytic $v$-overconvergent automorphic forms of weight $\ka_{\U}$ with Iwahori level  is \[\pi_{*}^{AIP}\Ol_{\iw^+_w}[\ka_{\U}^{\vee}],\] which means the subsheaf where $B_{M}(\Z_p)$ acts through the character $\ka_{\U}^{\vee}$.

With the help of (pseudo-)canonical subgroup (i.e. see section 2.3 and section 3.7 of \cite{drw}), we can show that two kinds of locus $\{\X_{Iw,w}\}$ and $\{\X_{Iw}(v)\}$ are equivalent system of neighborhoods of multiplicative ordinary locus. The remaining task to compare two kinds of sheaves. Unlike the classical setting, when $n>1$, the map $\pi_w$ is \textbf{not} a group torsor. The Borel subgroup $B_{M}(\Z_p)$ naturally acts on it but this action is neither transitive nor free. To make the comparison more comprehensive, I further introduced  a kind of "\textit{generalized Igusa torsor}": \[\pi^{AIP,B}: \Igb_w \rightarrow \iw^+_w \rightarrow \X_{Iw}(v).\] It is a group torsor over $\X_{Iw}(v)$ and can  replace the role of $\iw^+_w$ in the construction.

By pulling back to infinite level, we obtain the following Cartesian  diagram 
   
\[
\xymatrix{
\Igb_{\infty,w} \ar[r] \ar[d]_{\pi_{\infty}^{AIP,B}} & \Igb_w \ar[d]_{\pi^{AIP,B}} \\
\X_{\infty}(v) \ar[r] & \X_{Iw}(v)
},
\]
where $\X_{\infty}(v)$ is the subspace of $\X_{\infty}$ via pullback of $\X_{Iw}(v)$. Then the previous constructed section $\s=(\s_1,\cdots,\s_n)$ will further \textbf{trivialize} the group torsor $\Igb_{\infty,w}$. Under this trivialization, the overconvergent automorphic forms  will correspond to $Iw$-invariant of certain space of $w$-analytic induction on $Iw_M$. The $Iw$-action is twisted and the transformation rule \[g^*(\s)=\s J(\z,g)\] shows that this twisted action is exactly the \textbf{same} twisted action of $Iw$ in the construction of perfectoid automorphic forms. Therefore two constructions are equivalent. The proof also demonstrates that our perfectoid construction realizes overconvergent automorphic forms as the $p$-adic analogue of  complex analytic automorphic forms.

One immediate advantage of perfectoid construction is that we can think of overconvergent automorphic forms as $Iw$-invariant of certain "\textit{big}" sheaf. It suggests that we can also try to realize overconvergent cohomology as $Iw$-invariants of some "\textit{big}" sheaf and explore comparison maps between these "\textit{big}" sheaves. This is exactly the strategy of second part. In practice, there is a subtle fact we will work with \textbf{strict Iwahori}   level to establish such comparison map. The reason is that the automorphic factor $J(\z,g)$ suggests to use \textbf{opposite} Borel subgroup of $G(\Q_p)$ in defining analytic distribution and overconvergent cohomology.

To carry out this method, one essential ingredient is the  theory of \textbf{pro-etale topology }established in \cite{scholzepro}. Following the general construction in \cite{hansen2017universal} and \cite{jn2019eigenvariety}, we introduce analytic induction and the etale sheaf $\mathscr{D}_{\ka_{\U}}^{r}$ computing overconvergent cohomology. Then through the completed pullback we obtain the desired "\textit{big}" sheaf $\mathscr{OD}_{\ka_{\U}}^{r}$ over pro-etale site of $\X_{Iw^+,w}$ (analogous open adic space inside strict Iwahori level Shimura variety $\X_{Iw^+}$). Similarly we obtain the pro-etale sheaf $\widehat{\underline{\omega}}_{w}^{\ka_{\U}}$ via completed pullback of $\underline{\omega}_{w}^{\ka_{\U}}$.

Inspired by highest weight theory of classical representation theory (see section \ref{classical es} for more details) we explicitly construct a map \[\mathscr{OD}_{\ka_{\U}}^{r} \rightarrow \widehat{\underline{\omega}}_{w}^{\ka_{\U}},\] which is equivariant under the action of $Iw^+$. Taking $Iw^{+}$-invariants and further apply some standard computations in $p$-adic Hodge theory relating pro-etale cohomology and etale cohomology (i.e. \cite{scholzepro} and generalized projection formula in \cite{drw}), we get the desired $p$-adic Eichler-Shimura map (theorem \ref{overconvergent es}):  \[ES_{\ka_{\U}}: \mathfrak{OC}^{r}_{\ka_{\U},\C_p} \rightarrow  M^{\ka_{\U}+n+1}_{Iw^+,w}(-n).\] It is compatible with classical Eichler-Shimura map at classical weights (theorem \ref{thm p adic and classical es}), thus can be as $p$-adic interpolation of such classical comparison maps. Moreover it is functorial in weights and  can be glued into a comparison map over the eigenvariety $\mathcal{E}$ (theorem \ref{es over eigenvariety}).

The methods developed in this paper can be further generalized to other unitary Shimura varieties. Moreover, having related overconvergent cohomology to overconvergent automorphic forms, we can further study its interaction with higher Coleman theory (as in \cite{diao2025overconvergent} for $GSp_4$). These directions will be explored in subsequent papers of this series.

Our development of a $p$-adic Eichler-Shimura theory for unitary groups is primarily motivated by the study of the geometry of eigenvarieties and the arithmetic  of $p$-adic $L$-functions. In the case of $GSp_4$ case, for instance, such a theory has enabled Diao, Rosso, and Wu to establish results on the ramification locus of the weight map for the eigenvariety (see Corollary 1.2.5 in \cite{diao2025overconvergent}), and has allowed Loeffler and his collaborators to make progress on the Bloch-Kato conjectures (see his ICM 2022 talk \cite{loeffler2022icm}). We aim to explore analogous arithmetic applications for unitary groups in the future.

The paper is organized as follows.

$\bullet$ Section \ref{section2} introduces the basic notations and construct perfectoid automorphic forms via infinite level Shimura varieties and the $\pi_{HT}$ period map.

$\bullet$ Section \ref{section3}  compares this construction with previous constructions of overconvergent automorphic forms and establishes their equivalence.

$\bullet$ Section \ref{section4} introduces the overconvergent cohomology groups. 

$\bullet$ Section \ref{section5}  construct  suitable pro-etale sheaves associated with perfectoid automorphic forms and overconvergent cohomology, and derives the desired overconvergent Eichler-Shimura map.

$\bullet$ Section \ref{section6} glues this map into a comparison map between the corresponding coherent sheaves on the eigenvariety. 
 
$\bullet$ Section \ref{section7} discusses further directions, including generalizations to other unitary Shimura varieties and potential arithmetic applications.

\textbf{Acknowledgements.} I'm particularly grateful to Liang Xiao and Hansheng Diao for their encouragement and discussion. I also thank Wenhan Dai, Hao Fu, Zicheng Qian, Peihang Wu, Ruiqi Bai, Fei chen, Yi Zhu and Yichao Tian  for their help. I thank Kai Xu for his support. I  thank DeepSeek, Gemini and cats for their encouragement.  During June 2025 I gave a talk introducing this work on the conference about arithmetic geometry and representation theory in Harbin Institute of Technology, I'm really appreciated to HIT (in particular Jiandi Zou) for such an invitation and their nice host.

Finally, many ideas were conceived  when I was a postdoc in Morningside Center of Mathematics. The extraordinary environment at MCM helped me overcome lots of difficulties. Therefore I'm intensely grateful for MCM.

\textbf{Notations}

For the convenience of readers, we mainly use similar conventions as in \cite{chj}, \cite{drw}, \cite{shen2016} and \cite{bp2021higher}.

$\bullet$ $p$ is an odd prime with $p>3$ (this is for simplicity as many ideas still hold for small prime).

$\bullet$ We fix an algebraic closure $\overline{\Q_p}$ of $\Q_p$ and an algebraic isomorphism $C_p \cong \C$, where $C_p$ is the $p$-adic completion of $ \overline{\Q_p}$.We normalize the absolute valuation (norm) on $C_p$ so that $|p|=p^{-1}$. For any $w \in \Q>0$, we denote by $p^{w}$ an element in $C_p$ with valuation $p^{-w}$. All constructions and arguments in this paper are independent on these choices and conventions.  We will write $Gal_{\Q_p}$ for the Galois group $Gal( \overline{\Q_p}|\Q_p)$.

$\bullet$ When we use the standard free module $R^{n}$ over a ring $R$, we  work with the standard basis $\{e_0,...,e_{n-1}\}$, where $e_i$ is the \textbf{column vector} with $1$ at $i+1$-th entry and $0$ at other entry. Depending on the context we  also use the re-labeling $\{x_1,\cdots,x_n\}$. Unless otherwise stated, all vectors are assumed to be \textbf{column} vectors and the matrix usually acts on them via left multiplication. The transpose of the matrix $g$ is denoted by $g^{t}$ and the composition of transpose and inverse is $g^{-t}$ for short.

$\bullet$ Let $M_{a,b}(R)$ denote the set of $a \times b$ matrices over a (possibly non-unital) ring $R$. When $a=b$ we simplify the notation to $M_a(R)$.

$\bullet$ Let $n$ be a positive integer. In this paper, we will usually decompose a $(n+1)\times (n+1)$ into blocks in the following way:  we will write \[g= \begin{pmatrix}
g_{a} & g_{b} \\
g_{c} & g_{d}
\end{pmatrix}\] where $g_a$ is a $1\times 1$ matrix, $g_b$ is a $1 \times n$ matrix, $g_c$ is a $n \times 1$ matrix and $g_d$ is a $n\times n$ matrix.

$\bullet$ Unless otherwise stated, symbols in Gothic font (e.g. $\mathfrak{X}$) will represent formal schemes; symbols in calligraphic font (e.g. $\mathcal{X}$) will represent adic spaces; and symbols in script font (e.g. $\mathscr{F}$) will represent other geometric objects like sheaves.

\section{Perfectoid automorphic forms}
\label{section2}

In this section we construct the perfectoid automorphic forms through the infinite level unitary Shimura varieties and the Hodge-Tate period map $\pi_{HT}$.

\subsection{Unitary Shimura varieties and \texorpdfstring{$\pi_{HT}$}{\pi_{HT}}}
\label{unitary shimura}

In this section we introduce the unitary Shimura varieties.  These were studied by Harris-Taylor for proving the local Langlands of $GL_N$ (see \cite{ht2001}). For more details and background knowledge, see section I.7, III.1, III.4 of \cite{ht2001} and section 2 of \cite{shen2016} for a quick summary.

We will mainly follow the notation in Harris-Taylor (\cite{ht2001}) and Xu Shen (\cite{shen2016}), though as we work with imaginary quadratic field instead of general CM field, which will simplify many notations. In fact  with more efforts in notations, our method can be generalized to that setting.

Let $E/\Q$ denote an imaginary quadratic field. Take a prime $p\geq5$ that splits in $E$. The condition $p\geq 5$ is harmless and could be removed with additional technical work (for canonical subgroups).  Let the complex conjugation of $Gal(E/\Q)$ be $c$. Suppose $p$ splits as $\varpi \varpi^c$. Let $\mathbf{B}/E$ denote a central division algebra of dimension $(n+1)^2$ ($n\geq 1$) over $E$ such that

$\bullet$ the opposite algebra $\mathbf{B}^{op}$ is isomorphic to $\mathbf{B} \otimes_{E,c}E$;

$\bullet$ $\mathbf{B}$ splits at $\varpi$;

$\bullet$ For any place $v$ of $E$, if it is not split over $\Q$, $\mathbf{B}_v$ is split; if it is split over $\Q$, either $\mathbf{B}_v$ is split or $\mathbf{B}_v$ is a division algebra;

$\bullet$ if $n+1$ is even then $\frac{n+3}{2}$ is congruent modulo $2$ to the number of places of $\Q$ above which $\mathbf{B}$ is ramified.

Due to \cite{ht2001} I.7 (see p.51), we can choose an involution of second kind $*$ on $\mathbf{B}$. What's more, we can choose some alternating pairing $\langle , \rangle$ on $\mathbf{V} \times \mathbf{V} \longrightarrow \Q$ for the left $\mathbf{B}\otimes_{E} \mathbf{B}^{op}$ module $\mathbf{V}=\mathbf{B}$, which corresponds to another involution of second kind $\sharp$ on $B$. The resulting reductive group $G/\Q$ is defined by \[G(R)=\{(g,\lambda)\in(\mathbf{B}^{op}\otimes_{\Q}R )^{\times} \times R^{\times}| g g^{\sharp}=\lambda Id\},\] here $R$ is any $\Q$-algebra. Let $G_1$ denote the kernel of the similitude map $G\longrightarrow \G_m$, $(g,\lambda)\mapsto \lambda$. Due to lemma I.7.1 of \cite{ht2001}, we can further require this alternating pairing on $\mathbf{V } \times \mathbf{V} \longrightarrow \Q$ satisfy

$\bullet$ If $l$ is a rational prime which is not split in $E$, then $G$ is quasi-split at $l$,

$\bullet$ The unitary group $G_{\R}$ has signature $(1,n)$.

By our conditions at $p$, we have the following isomorphisms over $\Q_p$ : \[G_{\Q_p}\cong (\mathbf{B}^{op}_{\varpi})^{\times}\times \G_m \cong GL_{n+1}\times \G_m.\] Then $G_{1,\Q_p}\cong GL_{n+1}$. As the similitude factor $\G_{m}$ plays no role in later computations, we can ignore it and think $G_{\Q_p}$ just as $GL_{n+1}$ by abuse of notations.


Fix a maximal order $O_{\mathbf{B}_{\varpi}}$ inside $\mathbf{B}_{\varpi}$ and an isomorphism \[O_{\mathbf{B}_{\varpi}}\cong GL_{n+1}(O_{E_\varpi})=GL_{n+1}(\Z_p).\] Let $O_{\mathbf{B}}\subset \mathbf{B}$ be the unique maximal $\Z_{(p)}$-order such that $(O_{\mathbf{B}})^*=O_\mathbf{B}$ and $O_{\mathbf{B},\varpi}=O_{\mathbf{B}_\varpi}$.

Now we turn to describe the unitary Shimura variety associated to the Shimura datum for $G$. Let $K \subset G(\A_f)$ be a small enough (neat) open compact subgroup. Then we have a projective smooth moduli variety $Sh_{K}$ over $E$, which parametrizes abelian varieties with additional PEL (polarization, endomorphism, level structure) type  structure. More precisely, for any connected locally noetherian $E$-scheme $S$, the set $Sh_{K}(S)$ is identified with the set of isomorphism classes $\{(A,\lambda,\iota,\eta)\}/\simeq$, where

$\bullet$ $A/S$ is an abelian scheme of dimension $(n+1)^2$;

$\bullet$ $\lambda: A \longrightarrow A^{\vee}$ is a polarization;

$\bullet$ $\iota: \mathbf{B}\longrightarrow End(A)\otimes \Q$ defines an action with  relations $\lambda \circ \iota(b)=\iota(b^*)^{\vee}\circ \lambda$ for any $b \in \mathbf{B}$ and the pair $(A, \iota)$ is \textit{compatible}  in the sense of Lemma III.1.2 in \cite{ht2001} (or equivalently this action satisfies the Kottwitz condition about $Lie(A)$);

$\bullet$ $\eta$ is a level structure $\eta: \mathbf{V}\otimes \A_f \longrightarrow\mathbf{V}_{f}(A)$ $(mod$ $K)$.

This moduli variety is  a disjoint union of $|ker^{1}(\Q,G)|$ copies of the PEL unitary Shimura variety $Sh_{K}(G,X)$ associated to the corresponding Shimura datum. By abuse of notations, we ignore such difference and  think such moduli space as our Shimura variety.

\begin{rem}

Indeed, as we work with $E/\Q$ an imaginary quadratic field instead of general CM fields, the center of $G$ is just $Res_{E/\Q}\G_m$, which is cohomological trivial (the obstruction $ker^{1}(\Q,G)$ vanishes), we lose nothing.

\end{rem}

 What's more, this paper mainly concerns with level subgroup at $p$ and studies $p$-adic geometry. Thus from now on, we will always suppose $K=K_p \times K^{p}$ and fix the prime to $p$-part $K^p$ (tame level). Similarly we can decompose the level structure $\eta$ as $\eta^{p} \times \eta_p$. Therefore the above data can decomposed into  $(A, \lambda, \iota, \eta^{p})$ and $\eta_p$ (information at $p$). Again in this paper we will focus on  $\eta_p$, which is purely about $A[p^{\infty}]$. Notice that we have the canonical decomposition \[A[p^{\infty}]= A[\varpi^{\infty}] \oplus A[\varpi^{c,\infty}], \] and $A[\varpi^{\infty}]$ is dual to $A[\varpi^{c,\infty}]$. Let $\varepsilon \in M_{n+1}(\Z_p)$ denote the idempotent which  is $1$ at entry $(1,1)$ and $0$ at other entries, then we get a $1$-dimensional $p$-divisible groups (and height is $n+1$) $H=\varepsilon A[\varpi^{\infty}]$ and a decomposition \[\varepsilon A[p^{\infty}]=H \oplus H^{D}.\] We can translate the information $\eta_p$ into level structure information about $H$.

For later use, we further decompose the tame level $K^{p}=K^{S_0}\times K_{S_0}$, where $S_0$ is a finite set of \textit{bad primes} (outside $p$):

$\bullet$ $S_0$ contains $\{2,3\}$.

$\bullet$ For each prime $l \notin S_0 \cup \{p\}$, the field extension $E/\Q$ is unramified at $l$, $G(\Q_l)$ is an unramified $l$-adic group, the level subgroup $K_{l}$ is a hyperspecial subgroup for $G(\Q_l)$.

For later applications, we introduce more notations to explain the moduli interpretations more clearly.

Let $V$ denote the standard free module $\Z_p^{n+1}$ (\textbf{not} the previous module $\mathbf{V}$), with the standard basis $\{e_0,...,e_n\}$, where $e_i$ is the\textbf{ column} vector with $1$ at $i$-th entry and $0$ at other entries. We equip $V_{\Q_p}$ with the standard action (left multiplication) by $G_{\Q_p}=GL_{n+1}(\Q_p)$ (ignore the similitude factor). The standard lattice $V$ determine a hyperspecial subgroup $K_p(0)$, which is isomorphic to $GL_{n+1}(\Z_p))$. Moreover, we fix a pinning down data for $GL_{n+1}$: let $T$ denote the maximal split torus consisting of diagonal matrices and $B$ denote the Borel subgroup consisting of upper triangular matrices. Now let $Iw$ denote the Iwahori subgroup of $GL_{n+1}(\Z_p)$ which is the preimage of $B(\F_p)$ under the modulo $p$ reduction,  and $Iw^{+}$ denote the \textbf{strict} Iwahori subgroup of $GL_{n+1}(\Z_p)$ which is the preimage of $T(\F_p)$ under the modulo $p$ reduction. What's more, there is a standard Levi subgroup $GL_1 \times GL_{n}$ corresponding to  the decomposition $V\cong <e_0> \bigoplus <e_1,...,e_n>$. Let $M$ denote $GL_{n}$ and through the natural embedding $M\hookrightarrow GL_1 \times GL_n$, we view $M$ as a subgroup of  $G_{\Q_p}\cong GL_{n+1}$.  In this paper, by abuse notations, we will think $M$ as the Levi subgroup $GL_1 \times GL_n$. As we will work with "\textit{reduced}" weights (the extra $GL_1$-part action is trivial), such simplification doesn't matter.

We will mainly work with three level subgroups: $K_p(0)$ (hyperspecial), $Iw$ (Iwahori) and $Iw^{+}$ (strict Iwahori). Let $\Gamma \in \{Iw, Iw^{+}\}$, and let $X_{\Gamma}$ denote the corresponding Shimura variety over $\C_p$ with level $\Gamma$ at $p$. For the hyperspecial subgroup $K_p(0)$, we will omit the level and write $X$ for the Shimura variety. And we use $\mathcal{X}$ etc for their associated adic spaces over $Spa(\C_p,\mathscr{O}_{\C_p})$.

\begin{rem}
In the first part (perfectoid construction), we mainly work with Iwahori level $Iw$ and all arguments work directly for $Iw^+$. The automorphic factor (lemma \ref{transform law}) suggests to use opposite Borel subgroup when establishing $p$-adic Eichler-Shimura map. Therefore in the second part we will work with strict Iwahori level to construct such a comparison map. 
\end{rem}

We recall the moduli information about level structure at $p$:

$\bullet$ For Iwahori level, each point further record a full filtration $Fil_{\bullet}H[p]$ for $H[p]$;

$\bullet$ For strict Iwahori level, each point further record a decomposition of $H[p]$ thus a splitting  filtration $Fil_{\bullet}H[p]$.

Then through the forgetful map, we get natural finite etale maps:

\[X_{Iw}\longrightarrow X_{Iw^+}\longrightarrow X,\]

and similarly for the adic spaces version \[\mathcal{X}_{Iw}\longrightarrow \mathcal{ X}_{Iw^+}\longrightarrow \mathcal{X}.\]

Moreover, we can also discuss infinite level Shimura varieties in the framework of adic spaces. Define the infinite level Shimura variety \[\X_{\infty}:=\varprojlim_{K_{p}}\X_{K_p},\] such limits exits in the category of adic spaces (but not coming from an algebraic variety), and in fact it is a perfectoid space (see \cite{scholze2015torsion} for the proof). The space $\X_{\infty}$ also has certain moduli interpretations. Each $Spa(\C_p,\mathscr{O}_{\C_p})$ point $x$ corresponds to a tuple $(A,\iota,\lambda, \eta^{p})$ (tame PEL data) and an isomorphism (trivialization) $\eta_p:\Z_{p}^{n+1}\cong T_p(H)$ (level at $p$). Moreover, the group $G(\Q_p)$ acts on $\X_{\infty}$ from \textbf{right} and only influences information at $p$. We can describe such right action as follow (in terms of moduli interpretation): Let  $g\in G(\Q_p)$ and $x$ be a $Spa(\C_p,\mathscr{O}_{\C_p})$ point, the action is \[x\mapsto g \cdot x;\] \[(A,\lambda,\iota,\eta^p;\eta_p)\mapsto (A,\lambda,\iota,\eta^p;\eta_p \circ g).\]

Since the seminal work of Scholze in \cite{scholze2015torsion}, the perfectoid (infinite level) Shimura varieties have played increasingly important role in numerous areas of modern number theory, representation theory and arithmetic geometry. One key ingredient is that it has the Hodge-Tate period map $\pi_{HT}$ to flag varieties. This $\pi_{HT}$ is fundamental and relates $p$-adic geometry of flag varieties and perfectoid Shimura varieties.

Now we describe the period map $\pi_{HT}$ in term of $Spa(\C_p,\mathscr{O}_{\C_p})$ points.

Recall our moduli interpretation, each $Spa(\C_p,\mathscr{O}_{\C_p})$ point $x$ for $\X_{\infty}$ corresponds to $(A,\lambda,\iota,\eta^p;\eta_p)$, where $\eta_p: \Z_p^{n+1} \cong T_p(H)$. For the $p$-divisible group $H$, we have the following basic exact sequence in Hodge-Tate period map:

\[0\longrightarrow Lie(H)\longrightarrow T_p(H)\otimes_{\Z_p} \C_p\xrightarrow{HT_{H}} \omega_{H^D}\longrightarrow 0,\]  then the global map $\pi_{HT}$ sends $x$ to $\eta_p^{-1}(Lie(H))$, which is a line of $\C_p^{n+1}=V\otimes_{\Z_p} \C_p$ and corresponds to a point in the projective space $\mathbb{P}(V)$. In this way we define the period map:\[\pi_{HT}:\X_{\infty}\longrightarrow \Fl=\mathbb{P}(V)=\mathbb{P}^{n}.\]

This map $\pi_{HT}$ is $G(\Q_p)$-equivariant. The group $G(\Q_p)$ acts on both sides from \textbf{right}. Previously we have seen  its action on $\X_{\infty}$. Here we briefly describe the action on $\Fl$. The geometry of the flag variety will be examined in more detail in the following section. The group $G(\Q_p)\cong GL_{n+1}(\Q_p)$ has a natural \textbf{left} action on $\Q_p^{n+1}$ through left multiplication. Now we define the following  \textbf{right} action: for any $g \in GL_{n+1}(\Q_p)$ and $v \in \Q_p^{n+1}$, \[g * v:=g^{-1}v.\] This \textbf{right} action $*$ further induces a natural $G(\Q_p)$ \textbf{right} action on the flag variety $\Fl$.

To conclude this section, we outline the fundamental strategy throughout this paper.

The starting point is the following diagram \[\xymatrix{
\X_{\infty} \ar[r]^{\pi_{HT}} \ar[d]_{\pi_{K_p}} & \Fl \\
 \X_{K_p}
}.\]

The basic framework is to utilize  $\pi_{HT}$ to relate the geometry of the flag variety $\Fl$ with the infinite level Shimura variety $\X_{\infty}$, and  the natural projection map $\X_{\infty}\rightarrow \X_{K_p}$ enables us to explore finite level Shimura varieties. In particular we will investigate the following natural maps
\[\pi_{Iw}: \X_{\infty}\rightarrow \X_{Iw}, \pi_{Iw^+}: \X_{\infty}\rightarrow \X_{Iw^+}.\]

As a toy example, we record the following relations (lemma 2.2.4 in \cite{drw}):

\begin{lem} We have 
\[\Ol_{\X_{Iw}}^+ = (\pi_{Iw,*}\Ol_{\X_{\infty}}^+)^{Iw}, \Ol_{\X_{Iw}} = (\pi_{Iw,*}\Ol_{\X_{\infty}})^{Iw}.\]
\end{lem}

\subsection{Flag varieties and representation theory}
\label{flag variety}
In this section we set up basic notations for flag varieties, vector bundles on it and related weights in representation theory.

Recall that previously we have defined the \textbf{right} action $*$ of $G(\Q_p)$ on $\Q_p^{n+1}$ and the flag variety $\Fl=\mathbb{P}^n$. From now on, for simplicity we will also use $\cdot$ to denote this action. And by abuse notation we also write $G$ instead of $G(\Q_p)$ for short. Let $P$ denote the standard parabolic subgroup corresponding to the Levi subgroup $GL_1 \times GL_{n}$. It is exactly the parabolic subgroup $P_{\mu}$ defined by the Hodge cocharacter $[\mu]$ (in Shimura datum), see section 2.1 of \cite{caraianischolze2017} for more details. More explicitly, the parabolic subgroup $P$  consists of matrices in the following form  \[p=\begin{pmatrix}p_{a} & p_{b}\\ 0 & p_{d} \end{pmatrix}. \] Here we are using the \textbf{convention} that  $p_{a}$ is  $1\times 1$ matrix, $p_b$ is a $1 \times n$ matrix and $p_{d}$ is a $n \times n$ matrix.  Now  consider the following map \[G \rightarrow \Fl,\] \[g \mapsto g([e_0]).\] Here $[e_0]$ denotes the line spanned by $e_0$. It induces the following isomorphism \[P\backslash G \cong \Fl.\] Moreover, let $G$ act on left side via right multiplication (which is \textbf{right} action), this isomorphism is also $G$-equivariant.

Now we introduce some notations about \textbf{coordinates} and in particular $w$-ordinary locus for the flag variety.

There is Bruhat stratification on the flag variety. In particular, there is an open strata which is isomorphic to the affine space $\A^{n}$.   More explicitly, let $N_{\mu}^{opp}$ denote the unipotent subgroup for the opposite parabolic subgroup $P_{\mu}^{opp}$, there is an open embedding \[N_{\mu}^{opp}\hookrightarrow P\backslash G \cong \Fl, \] and the image is exactly the open strata. We denote this open strata by $\Fl^{\times}$ (follow the notation  in section 2.3 of \cite{drw}). Through the isomorphism \[\A^{n} \cong N_{\mu}^{opp},\] \[(z_1,...,z_n)\mapsto
\begin{pmatrix}
1   \\
z_1 & 1  \\
... & ... & ...  \\
z_n & ... & ... & 1\\
\end{pmatrix},\] we get coordinates $\textbf{z} \in \mathscr{O}_{\Fl^{\times}}(\Fl^{\times})^n$ for $\Fl^{\times}$. And we will also think $\textbf{z}=(\textbf{z}_1,...,\textbf{z}_n)$ as \textbf{column} vector. This coordinate differs a little from  inhomogeneous coordinate from $V$. For a point $x$ corresponding to $[1:x_1:...:x_n]$, its coordinate $\textbf{z}(x)$ is the following \[\textbf{z}(x)=(-x_1,...,-x_n).\]  In this paper we will always use this coordinate $\textbf{z}$. This convention will simplify our notations.

As we mentioned in the introduction, the overconvergent automorphic forms live in certain subset (neighborhood of ordinary locus) instead   of  the whole Shimura variety $\X_{K_p}$. And in this paper, we will mainly focus certain locus inside the open strata  $\Fl^{\times}$. In next paper of this series, we will discuss higher Coleman theory (see \cite{bp2021higher}) and study the whole $\Fl$.

From now on, we consider the associated adic space for $\Fl$ and $\Fl^{\times}$ and base change to $Spa(\C_p,\mathscr{O}_{\C_p})$. For simplicity we still use the notation $\Fl$, $\Fl^{\times}$ to denote the resulting adic spaces and use $\textbf{z}$ to denote the coordinates.

For each $w \in \Q_{>0}$, consider the following open adic subspace $\Fll_{w}$ inside $\Fll$ \[\Fll_{w}:=\{x \in \Fll: \max\limits_{i}\inf\limits_{a \in p \Z_p}{|\textbf{z}_i(x)-a|}\leq p^{-w}\}.\] It is $w$-ordinary locus over the flag variety. In section \ref{w-locus} we will introduce related $w$-locus over Shimura varieties. Here $w$ measures the "\textit{distance}" to ordinarity. When $w$ grows, the $w$-locus shrinks and becomes closer to ordinary locus.  Here   we briefly mention the following standard fact justifying the name (see \cite{scholze2015torsion} III.1 for more details):

$\bullet$ Let $A$ be an abelian variety over $\C_p$. Then its reduction is ordinary if and only if $Lie(A)$ is a $\Q_p$-rational subspace of $T_p(A)\otimes \C_p$.

The following lemma shows that $\Fll_w$ is stable under the Iwahori subgroup $Iw$ (right) action:

\begin{lem}
The adic space $\Fll_w$ is stable under the right action by $Iw$. More explicitly, the action is described as follow:

\[\Fll_{w}\times Iw \rightarrow \Fll_{w},\] \[\textbf{z},g=\begin{pmatrix}g_a & g_b \\ g_c & g_d \end{pmatrix} \mapsto (\textbf{z}g_b+g_d)^{-1}(\textbf{z}g_a+g_c) .\]

\end{lem}

Its proof is a straightforward computation.

We further introduce a kind of\textit{ automorphic factor} $J$:

\[M_{n \times 1} (\mathscr{O}_{\C_p}) \times Iw \rightarrow M_{n}(\C_p),\] \[J(\overrightarrow{v},g):= \overrightarrow{v}g_{b}+g_d.\]

This  factor is very basic in this paper and also appears in the definition of perfectoid automorphic forms. 

Formally we define the  \textbf{right} action of $Iw$ on $M_{n \times 1}(\mathscr{O}_{\C_p})$ as follow 

\[g \cdot \overrightarrow{v}:=(\overrightarrow{v}g_b+g_d)^{-1}(\overrightarrow{v}g_a+g_c).\]

Then this automorphic factor satisfies \textbf{right 1-cocycle property}:

\begin{lem}
\label{right 1cocycle}
For any $\overrightarrow{v} \in M_{n \times 1}(\mathscr{O}_{\C_p})$ and $g_1$, $g_2\in Iw$, we have the following relation:

\[J(\overrightarrow{v},g_1g_2)=J(\overrightarrow{v},g_1)J(g_1 \cdot \overrightarrow{v}, g_2).\]
\end{lem}

Again we can prove this lemma by direct computations and formally we can rewrite the action $g$ on $\Fll_{w}$ as \[g \cdot \textbf{z}= J(\textbf{z},g)^{-1}(\textbf{z}g_a+g_c).\]

Now we introduce certain vector bundles on the flag variety. They will correspond to automorphic bundles on Shimura varieties. See theorem 2.1.3 in \cite{caraianischolze2017} for more general results.

Recall our $\Fl$ is the projective space $\mathbb{P}^n$. In particular, it has the (universal) tautological line bundle $\mathscr{L}$, and $\mathscr{L}$ embeds into the $n+1$-dimensional trivial vector bundle $\mathscr{V}$. Moreover we have the following exact sequences for vector bundles: \[0\rightarrow \mathscr{L} \rightarrow\mathscr{V} \rightarrow \mathscr{W}^D \rightarrow 0.\] Here   $\mathscr{W}^{D}$ is the $n$-dimensional (quotient) vector bundle.

More explicitly, we can write $\mathscr{W}^{D}$ as the following quotient \[\mathscr{W}^{D}\rightarrow \Fl \cong P \backslash (G \times \A^{n})\rightarrow P \backslash G, \] where $P$ acts on $\A^n$ from \textbf{left}: for any $p=\begin{pmatrix}p_a & p_b \\ 0 & p_d\end{pmatrix}$, it sends a point $\overrightarrow{v}$ (view as column vector) to $p_{d}\overrightarrow{v}$.

What's more, the section of  $\mathscr{W}^{D}$ corresponds to the following algebraic functions on $G$: \[\{\text{algebraic functions }\xi:G \rightarrow A^n: \xi (pg)=p_d \xi(g), \forall p \in P, g \in G\}.\]

For each $1 \leq i \leq n$, we consider the following global section $\textbf{s}_i$ for $\mathscr{W}^D$ \[\textbf{s}_i(g):=\text{the $i$-th column of }g_d.\] Write $\textbf{s}=(\textbf{s}_1,\textbf{s}_2, ..., \textbf{s}_n)$ and view it as \textbf{row} vector. Over the open subspace $\Fll$, there is the following fundamental lemma:

\begin{lem}
\label{transform law}
For any $g=\begin{pmatrix}g_a & g_b \\ g_c & g_d\end{pmatrix} \in Iw$, we have 
\[g^*(\textbf{s})=\textbf{s}(\textbf{z}g_b+g_d)=\textbf{s}J(\textbf{z},g).\]
\end{lem}

\begin{proof}

It is enough to check this over $Spa(\C_p,\mathscr{O}_{\C_p})$ points. Recall our coordinates and set $\delta=\begin{pmatrix}1 & 0 \\ \textbf{z} & \mathbb{I}_n \end{pmatrix}$, then $\textbf{s}(\delta)=\mathbb{I}_n$. What's more, we have \[\delta g=\begin{pmatrix}g_a & g_b \\ \textbf{z}g_a+g_c & \textbf{z}g_b+g_d\end{pmatrix}\] and $(g^{*}(\textbf{s}))(\delta)=\textbf{s}(\delta g)$. Then the relation in the lemma holds.

\end{proof}

This elementary lemma reveals the geometric origin of the automorphic factor $J$ that arises in our construction of perfectoid automorphic forms. See section \ref{compare forms} (comparison with other constructions of overconvergent automorphic forms) for more details.

\begin{rem}
In this paper, as we work under unitary Shimura varieties with signature $(1,n)$ and \textit{reduced weight} (see section \ref{weight space}). The $n$-dimensional vector bundle $\mathscr{W}^{D}$ and its corresponding automorphic bundle (differential forms) are already sufficient. For general cases with signature $(a,b)$, it is necessary to discuss two kinds of vector bundles (dimensional $a$ and $b$). Another one will be the dual bundle for the tautological bundle. So that situation is more complicated but many ideas in this paper still works. In later papers of this series we will work out these details.
\end{rem}

Finally we set up notations for \textbf{inductions} and \textbf{weights}.

Unless otherwise stated, we will always work with  induction from \textbf{right}.

Let $R$ be a commutative ring and $G1$ be a group with a subgroup $G2$. Let $\chi$ denote a character $G1 \rightarrow R^{\times}$ and $Pr$ denote a certain property for $R$-valued functions on $G1$. Consider the following space of functions on $G1$: \[Ind_{G2}^{G1,Pr}(\chi,R):= \left\{f:G1 \rightarrow R
\ \middle| \
\begin{aligned}
& f(g_1 g_2)=\chi(g_2)f(g_1), \quad \forall g_1 \in G1, g2 \in G2,\\
& \text{$f$ has $Pr$}.
\end{aligned}
\right \}.\] It is called \textbf{(right) induction with property} $Pr$. The group $G1$ acts on this space through left translation: for any $g_1 \in G1$ and function $f$ in this induction, $g_1 .f$ is the following element \[\forall h_1 \in G1, (g_1.f)(h_1)=f(g_1^{-1}h_1).\] In this way the group action is from \textbf{left} and the induction space is a group representation.

In this paper, most of times when we use inductions (e.g. define perfectoid forms, analytic distributions), the subgroup $G2$ is clear (e.g. Borel subgroups), for simplicity we also omit $G2$ etc. Moreover, we  define a  \textbf{weight} to be a character for split torus.

 As an example we illustrate such induction ideas  and weights for classical highest weight representation. View the group $G=GL_{n+1}$ as an algebraic group over $Spec(\Q_p)$. We have fixed its Borel subgroup $B$ corresponding to the set of upper triangular matrices. Let $\ka$ be an algebraic dominant weight for the maximal torus $T$ (diagonal torus). More explicitly, \[\ka=(\ka_0,...,\ka_n) \in X^*(T)\cong \Z^{n+1},\] \[\ka_0 \geq \ka_1...\geq \ka_n.\] And consider the twist $\ka^{\vee}=w_{0,G}(-\ka)$, here $w_0$ is the longest element in the Weyl group. More explicitly, we have \[\ka^{\vee}=(-\ka_{n},-\ka_{n-1},...,-\ka_{0}).\] Let $Alg$ denote the property of being algebraic and consider the following induction \[\Vc_{\ka}:=Ind_{B}^{G,Alg}(\ka^{\vee},\Q_p).\] It is a finite dimensional algebraic representation for $G$. And it is indeed irreducible and is called highest weight representation. Each finite dimensional irreducible algebraic representation for $G$ comes in this way. See Humpreys' textbook \cite{humphreys2012linear} chapter XI for more details.

We want to stress that here is a \textbf{sign} issue about weights. The group $G$ acts on $\Vc_{\ka}$ through left translation, and the highest weight (respect to Borel subgroup) for $T$ in this representation is $\ka$ exactly.  This the reason why the definition for $\Vc_{\ka}$ involves $\ka^{\vee}$. To clarify this sign issue, we also call $\ka^{\vee}$ the \textbf{induction weight} and $\ka$ the \textbf{representation weight}. For example, for the standard representation $\Q_p^{n+1}$, the corresponding induction weight is $(0,...0,-1)$ and the representation weight is $(1,0,...,0)$.

We will follow the usual convention about weights in the literature like Andreatta-Iovita-Pilloni in \cite{AIP2015} and Xu Shen in \cite{shen2016}. Roughly speaking, similar to above case, we will define weight $\ka$ forms via induction process involving weights $\ka^{\vee}$.

\subsection{Notations for \texorpdfstring{$w$}{w}-groups and \texorpdfstring{$w$}{w}-ordinary locus}
\label{w-locus}
In this section we further set up some notations for $w$-groups and  introduce $w$-ordinary locus over Shimura varieties.

Previously in section \ref{unitary shimura}, we have already fixed a pin down data $(T,B)$ for $G(\Q_p) $ and introduce the  subgroup $M=GL_{n}$ (plays the role of Levi subgroup). Moreover, let $N$ denote the unipotent subgroup for $B$ and $B^{opp}$ denote the opposite Borel subgroup with the unipotent part $N^{opp}$. The Iwahori subgroup corresponding to $B^{opp}$ is denoted by $Iw^{opp}$, it is the subgroup of $GL_{n+1}(\Z_p)$ that is the preimage of $B^{opp}(\F_p)$ under modulo $p$. Similarly we set up these notations for $M$. We fix a pin down data $(T_{M},B_{M})$, where $T_{M}$ is the maximal torus defined by diagonal matrices and $B_{M}$ is the Borel subgroup consisting of upper triangular matrices. Let $N_{M}$ be the unipotent part of $B_{M}$ and similarly we introduce the opposite counterpart $(B_M^{opp},N^{opp}_{M})$. Let  $Iw_{M}$ (resp. $Iw_{M}^{opp}$) denote the Iwahori subgroup corresponding to $B_M$ (resp $B_M^{opp}$). Obviously these notations are compatible: \[T_{M}=T \cap M, B_{M}=B \cap M, Iw_{M}=Iw \cap M,...\]

To simplify notations, we further set up

\[G_{0}:=GL_{n+1}(\Z_p), N_0:=N(\Z_p), B_0:=B(\Z_p), T_0:=T(\Z_p)\] and 
\[M_0:=GL_n(\Z_p), N_{M,0}:=N_M(\Z_p), B_{M,0}:=B_M(\Z_p), T_{M,0}:=T_M(\Z_p). \] Similarly we define opposite analogue $N^{opp}_0$, $N_{M,0}^{opp}$.

Moreover, for positive integer $s$, we define

\[T_{M,s}:=\ker(T_M(\Z_p)\rightarrow T_{M}(\Z_p/p^{s})),\] \[B_{M,s}:=\ker(B_M(\Z_p)\rightarrow B_{M}(\Z_p/p^{s})),\] \[N_{M,s}:=\ker(N_M(\Z_p)\rightarrow N_{M}(\Z_p/p^{s}))\] and similarly for $N_{M,s}^{opp}$, $B_{M,s}^{opp}$.

In particular, we can write the Iwahori decomposition for $Iw_M$ as follow \[Iw_M=N_{M}^{opp}(p\Z_p)T_{M}(\Z_p)N_M(\Z_p)=N_{M,1}^{opp}T_{M,0}N_{M,0}.\]

We further set up certain "$\Ol_{\C_p}$-\textit{analogue}" of these kernels. For any $w \in \Q_{>0}$ we define

\[T_{M,w,\Ol}:=\ker(T_M(\Op)\rightarrow T_{M}(\Op/p^{s})),\] \[B_{M,w,\Ol}:=\ker(B_M(\Op)\rightarrow B_{M}(\Op/p^{s})),\] \[N_{M,w,\Ol}:=\ker(N_M(\Op)\rightarrow N_{M}(\Op/p^{s}))\] and similarly for $N_{M,w,\Ol}^{opp}$, $B_{M,w,\Ol}^{opp}$. These groups are large and \textbf{not} $p$-adic groups in the usual sense.

We introduce \textbf{$w$-groups}, which are about  \textbf{"$w$-neighborhood"} of previous $p$-adic groups.

For any $w \in \Q_{>1}$ and $s \in \Z_{>0}$, we define 

\[T^{(w)}_{M,s}:=\{\gamma=(\gamma_{i,j}) \in T_{M}(\Op): \exists \gamma^{'} \in T_{M,s} \text{ such that } \forall  (i,j),|\gamma_{i,j}^{'}-\gamma_{i,j} | \leq p^{-w}  \},\] \[B^{(w)}_{M,s}:=\{\gamma=(\gamma_{i,j}) \in T_{M}(\Op): \exists \gamma^{'} \in B_{M,s} \text{ such that } \forall  (i,j),|\gamma_{i,j}^{'}-\gamma_{i,j} | \leq p^{-w}  \},\] \[N^{(w)}_{M,s}:=\{\gamma=(\gamma_{i,j}) \in T_{M}(\Op): \exists \gamma^{'} \in N_{M,s} \text{ such that } \forall  (i,j),|\gamma_{i,j}^{'}-\gamma_{i,j} | \leq p^{-w}  \},\] and similarly define $N^{opp,(w)}_{M,s}$ and $B^{opp,(w)}_{M,s}$.

In the same way we define $w$-neighborhood for $Iw_{M}$:

\[Iw^{(w)}_{M}:=\{\gamma=(\gamma_{i,j}) \in GL_n(\Op): \exists \gamma^{'} \in Iw_M \text{ such that } \forall  (i,j),|\gamma_{i,j}^{'}-\gamma_{i,j} | \leq p^{-w}  \}.\] For any positive integer $m$, we let $Iw_{M}(\Z/p^m)$ denote the subgroup of $GL_{n}(\Z/p^m)$ which is the image of reduction modulo $p^m$ for $Iw_M$. If $m$ is further the largest integer which is equal or smaller than $w$, then the image $Iw_{M}\rightarrow GL_n(\Ol_{\C_p}/p^w)$ is also $Iw_{M}(\Z/p^m)$. We have the following Cartesian diagram \[
\xymatrix{
Iw^{(w)}_{M} \ar[r] \ar@{^{(}->}[d] & Iw_M(\Z/p^m) \ar@{^{(}->}[d] \\
GL_n(\Ol_{\C_p}) \ar[r] & GL_n(\Ol_{\C_p}/p^w)
}.\] 

We have the following analogue of "\textit{large}" Iwahori decomposition:

\[Iw^{(w)}_M=N_{M,1}^{opp,(w)}T_{M,0}^{(w)}N_{M,0}^{(w)}.\]

On the other hand, we have the following relations about $w$-groups:

\[T^{(w)}_{M,s}=T_{M,s}T_{M,w,\Ol}, B^{(w)}_{M,s}=B_{M,s}B_{M,w,\Ol}, N^{(w)}_{M,s}=N_{M,s}N_{M,w,\Ol}.\] 

And similar relations hold for $N^{opp,(w)}_{M,s}$ and $B_{M,s}^{opp,(w)}$.

For later use (in particular the comparison of two constructions in section \ref{classical form}), we further introduce related group objects in adic spaces.

\begin{defn}
Let $\mathbf{B}(0,1)=Spa(\C_p \langle X \rangle, \Ol_{\C_p}\langle X\rangle)$ denote the closed unit disk and then $\mathbf{B}(0,p^{-w})=p^w \mathbf{B}(0,1)$ is the closed disk with radius $p^{-w}$.

(1) Define \[\mathcal{N}_{M,w}:=\begin{pmatrix} 1 &  \mathbf{B}(0,p^{-w}) & \cdots & \cdots & \mathbf{B}(0,p^{-w})\\ & 1 &  \mathbf{B}(0,p^{-w}) & \cdots &  \mathbf{B}(0,p^{-w}) \\  &   & \cdots & \cdots & \cdots \\   &   &  &  & 1 \end{pmatrix} .\] As adic space it is isomorphic to   $\mathbf{B}(0,p^{-w})^{\frac{n(n-1)}{2}}$.

(2) Define \[\mathcal{T}_{M,w}:=\begin{pmatrix} 1+\mathbf{B}(0,p^{-w}) \\ &  1+\mathbf{B}(0,p^{-w}) \\ & &  \cdots \\ & & & 1+\mathbf{B}(0,p^{-w})) \end{pmatrix}.\] As adic space it is isomorphic to  $\mathbf{B}(0,p^{-w})^{n}$.

(2) Define \[\Io ^{(w)}_{M}:=(a_{i,j}+\mathbf{B}(0,p^{-w})), (a_{i,j})\in Iw_{M}.\] As adic space it is isomorphic to finite copies of $\mathbf{B}(0,p^{-w})^{n^2}$.
\end{defn}

The $Spa(\C_p,\Ol_{\C_p})$-points for them coincides with the groups $T_{M,w,\Ol}$, $N_{M,w,\Ol}$ and $Iw^{(w)}_{M}$, which justifies the notations.

Now we turn to introduce \textbf{$w$-ordinary locus} for Shimura varieties.

Let $w \in \Q_{>0}$, in previous section we have defined $w$-ordinary locus $\Fl_{w}^{\times}$ on the flag variety $\Fl$. Through pullback via the period map $\pi_{HT}$, we get the analogue over Shimura varieties:

\begin{defn}

Define \[X_{\infty,w}:=\pi_{HT}^{-1}(\Fl_w^{\times}),\] it is an open adic subspace for $\X_{\infty}$. 

Moreover, we define \[\X_{Iw,w}:=\pi_{Iw}(\X_{\infty,w}), \X_{Iw^+,w}:=\pi_{Iw^+}(\X_{\infty,w}).\] They are open adic subspaces for $\X_{Iw}$ and $\X_{Iw^+}$ respectively. 

They are called \textbf{$w$-ordinary locus }(over infinite level, Iwahori level and strict Iwahori level).

\end{defn}

\begin{rem}

The locus $(\X_{\infty,w})$ is stable under the natural \textbf{right} action by $Iw$. Moreover, we also have \[\X_{\infty,w}= \pi_{Iw}^{-1}(\X_{Iw,w})=\pi_{Iw^+}^{-1}(\X_{Iw^+,w}).\]

\end{rem}

Recall the coordinate $\mathbf{z}=(\mathbf{z}_i)$ on $\Fl^{\times}$, define $\z_{i}:=\pi_{HT}^{*}\mathbf{z}_i$ and $\z:=\pi_{HT}^*\mathbf{z}$, in this way we get coordinates $\z$ (view it as \textbf{column} vector) for $\X_{\infty,w}$. 

Finally we discuss the promised relation between the vector bundle $\mathscr{W}^D$ over $\Fl$ and automorphic bundles over Shimura varieties.

Let $\mathcal{H}^{univ}$ (resp. $\mathcal{H}^{D,univ}$) denote the universal $p$-divisible groups (resp. Cartier dual) over $\X$ (hyperspecial level at $p$).  Denote the  the natural map \[\pi_{H^D}:\mathcal{H}^{D,univ}\rightarrow\X.\]  Consider the universal Hodge bundle \[\underline{\omega}:=(\pi_{H^D})_*\Omega^1_{\mathcal{H}^{D,univ}/\X},\] which is an $n$-dimensional vector bundle over $\X$. Similarly we get the line bundle \[\underline{L}:=(\pi_{H})_*Lie(\mathcal{H}^{univ}/\X).\]

Consider the natural map $\pi_{X}: \X_{\infty}\rightarrow \X$ and define \[\underline{\omega}_{\infty}:=\pi_X^*(\underline{\omega}).\] We have the following observation (a special case of theorem 2.1.3 in \cite{caraianischolze2017}):

\begin{lem}
There is a natural isomorphism \[\pi_{HT}^* \mathscr{W}^D \cong \underline{\omega}_{\infty}.\]
\end{lem}

\begin{proof}
Let $\underline{L}_{\infty}=\pi_X^*(\underline{L})$ and it is a line bundle over $\X_{\infty}$. The exact sequence about Hodge-Tate period map produces an exact sequence for vector bundles over $\X_{\infty}$:\[0 \rightarrow \underline{L}_{\infty}\rightarrow \Ol^{n+1} \rightarrow \underline{\omega}_{\infty} \rightarrow 0. \]

On the other hand, over the flag variety $\Fl$ there is also an exact sequence for vector bundles: \[0\rightarrow \mathscr{L} \rightarrow\mathscr{V} \rightarrow \mathscr{W}^D \rightarrow 0.\] 

From the construction, $\underline{L}_{\infty}$ is exactly the pullback (via $\pi_{HT}$) of the universal line bundle $\mathscr{L}$. Then the first exact sequence is the pullback of the second exact sequence. In particular we obtain this lemma.
\end{proof}

Use pullback under   natural maps $\X_{Iw} \rightarrow \X$ and $\X_{Iw^+} \rightarrow \X$, we further define the  automorphic bundle $\underline{\omega}_{Iw}$ and $\underline{\omega}_{Iw^+}$ respectively.

Moreover, set $\s_i:=\pi_{HT}^* \mathbf{s}_i$ and $\s:=(\s_1,...,\s_n)=\pi_{HT}^*\mathbf{s}$ (view it as \textbf{row} vector). These are sections for $\underline{\omega}_{\infty}$ and we have the following transformation law (due to lemma \ref{transform law}):
\[\gamma^*(\s)=\s (\z \gamma_b+\gamma_d)=\s J(\z,\gamma)\] for any $\gamma=\begin{pmatrix} \gamma_a & \gamma_b \\ \gamma_c & \gamma_d \end{pmatrix} \in Iw$.

\begin{rem}

As we will work with "reduced" weight (see next section), it is enough to discuss this $n$-dimensional automorphic bundle. In general case, we will also consider another automorphic bundle (dual bundle for $\underline{L}$) coming from $\mathcal{H}$.

\end{rem}

 \subsection{Weight space and perfectoid automorphic forms}
 \label{weight space}

 Following \cite{chj} and \cite{drw}, we  introduce small weights, affinoid weights and further  construct perfectoid automorphic forms.
 
 Recall the maximal torus $T_{M}$ for $M\cong GL_{n}$, it is an $n$-dimensional torus. Notice that the torus $T$ for $GL_{n+1}$ has the following decomposition $T=GL_{1}\times T_{M}$. A "\textit{reduced}" weight for $T$ is a weight which is trivial on $GL_1$, and it is equivalent to be a weight for $T_{M}$. In this paper, we will always work with "reduced" weight. In particular, the weight space will be $n$-dimensional instead of $n+1$-dimensional. As central weight twist won't influence too much, such assumption is mild. And this simplified convention is also widely used. For example, in the literature about modular forms or eigencurve (like \cite{chj}), the weight is usually a single element $k$ instead of a two-tuple $(k_0,k_1)$. Moreover, it will simplify many notations in this paper.

\begin{defn}

The \textbf{weight space} is \[\mathcal{W}:=Spa(\Z_p[[T_{M}(\Z_p)]],\Z_p[[T_{M}(\Z_p)]])^{rig}=Spa(\Z_p[[(\Z_p^{\times})^{n}]],\Z_p[[(\Z_p^{\times})^{n}]])^{rig},\] here "$rig$" means taking the generic fiber.

\end{defn}

The weight space parametrizes continuous weights, for example, $\mathcal{W}(\C_p)$ is exactly the set of continuous group maps $T_M(\Z_{p})\rightarrow \C_p^{\times}$. More generally, for any complete  $\Z_p$-algebra $R$, each continuous  weight $\ka:T_{M}(\Z_p)\rightarrow R^{\times}$ can be written as $\ka=(\ka_1,...,\ka_n)$, where $\ka_i:\Z_p^{\times}\rightarrow R^{\times}$ is a continuous character. From this description, we can see that the weight space $\W$ is indeed a finite union of $n$-dimensional open ball. This fact can also be deduced from the following (non-canonical) isomorphism (depending on choice $\Z_p^{\times}\cong \mu_{p-1}\times (1+p\Z_p)$): \[\Z_p[[T_{M}(\Z_p)]]\cong \Z_{p}[(\mu_{p-1})^{n}]\otimes \Z_p[[T_1,...,T_n]]. \]

 Now we define small weights and affinoid weights (see section 3.1 in \cite{drw} for more details):
 
 \begin{defn}
 
 (1) A \textbf{small $\Z_p$-algebra} is a $p$-torsion free reduced ring which is also a finite $\Z_p[[T_1,...,T_d]]$-algebra for some non-negative integer $d \in \Z_{\geq 0}$. In particular, such an algebra is equipped with a canonical adic profinite topology and is complete with respect to $p$-adic topology.
 
 (2) A \textbf{small weight} is a pair $(R_{\U},\ka_{\U})$ where $R_{\U}$ is a small $\Z_p$-algebra and $\ka_{\U}:T_{M}(\Z_p)\rightarrow R_{\U}^{\times}$ is a continuous character such that $\ka_{\U}((1+p)\mathbb{I}_{n})-1$ is topological nilpotent in $\R_{\U}$ with respect to $p$-adic topology. Then it induces a natural map \[Spa(R_{\U},R_{\U})^{rig}\rightarrow \W. \] By abuse notation we also call $\U:=Spa(R_{\U},R_{\U})$ a small weight and  write $R_{\U}^{+}:=R_{\U}$.
 
 (3) An \textbf{affinoid weight} is a pair $(R_{\U},\ka_{\U})$ where $R_{\U}$ is a reduced Tate algebra which is topologically finite type over $\Q_p$ and $\ka_{\U}:T_{M}(\Z_p)\rightarrow R_{\U}^{\times}$ is a continuous weight. Then it induces a natural map \[Spa(R_{\U},R_{\U}^{\circ})^{rig}\rightarrow \W. \] By abuse notation we may also call $\U:=Spa(R_{\U},R_{\U}^{\circ})$ an affinoid weight and also write $R_{\U}^{+}:=R_{\U}^{\circ}$.

 \end{defn}
 
 From now on, by a \textbf{weight} we always mean a small weight or an affinoid weight. 
 
 We further introduce the following \textbf{convention}:
 
 For any integer $m$, we also view $m$ as a weight by identifying it with the character \[T_{M}(\Z_p)\rightarrow \Z_p^{\times},\] \[(t_1,..,t_n)\mapsto \prod_{i}t_i^{n}.\] What's more, for any weight $\ka=(\ka_1,...,\ka_n)$, we write $\ka+m$ for the weight defined by $(\ka_1+m,...,\ka_n+m)$.

 Now we define the conception of "\textit{mixed completed tensor}" as in section 3.1 of \cite{drw}. It is slightly different from the convention in \cite{chj}, see \cite{drw} for more details.
 
 \begin{defn}
 Let $R$ be a small $\Z_p$-algebra.
 
 (1) For any $\Z_p$-module $Q$, we define \[Q\widehat{\otimes}^{'} R:=\varprojlim_{j\in Inx} (Q \otimes_{\Z_p} R/I_j )\] here $(I_j:j \in Inx)$ runs over a cofinal system of neighborhood of 0 consisting of $\Z_p$-submodules of $R$. Moreover, if $Q$ is a $\Z_p$-algebra, then $Q\widehat{\otimes}^{'} R$ is also a $\Z_p$-algebra.
 
 (2) For any $\Q_p$-Banach module $Q$ with an open bounded $\Z_p$-submodule $Q_0$. We define the \textbf{mixed completed tensor} \[Q\widehat{\otimes}R:=(Q_0\widehat{\otimes}^{'} R)[\frac{1}{p}].\] This is independent of choice of $Q_0$.
 
 \end{defn}
 
 Combine with weights we  give the following definitions:
 
 \begin{defn}
 Let $(R_{\U},\ka_{\U})$ be a weight.
 
 (1) For any $\Z_p$-module $Q$, the symbol $Q\widehat{\otimes}^{'} R_{\U}^{+}$ will either represent $Q\widehat{\otimes}^{'} R_{\U}$ in the case of small weights, or represent the $p$-adically completed tensor over $\Z_p$ in the case of affinoid weights.
 
 (2) For any $\Q_p$ Banach module $Q$, the symbol $Q\widehat{\otimes} R_{\U}$ will either represent the mixed completed tensor in the case of small weights, or represent the $p$-adically completed tensor over $\Q_p$ in the case of affinoid weights.
 
 \end{defn}
 
 Now we introduce \textbf{$r$-analytic} functions, which is basic for overconvergent automorphic forms and overconvergent cohomology.

 \begin{defn}
 
 Let $r \in \Q_{>0}$ and $m \in \Z_{>0}$. Let $Q$ be a uniform $\C_p$-Banach algebra and $Q^{\circ}$ be the unit ball.
 
 (1) A function $f: \Z_p^{m} \rightarrow Q$ (resp. a function $f: (\Z_p^{\times})^m \rightarrow Q$) is called \textbf{$r$-analytic} if for each $a=(a_1,...,a_m)\in \Z_p^m$ (resp. $a \in (\Z_p^{\times})^m$), there is a power series $f_{a} \in Q[[T_1,...,T_m]]$ which converges on the $m$-dimensional closed unit ball $\mathbf{B}^m(0,p^{-r})\subset \C_p^m$ such that \[f(x_1+a_1,...,x_m+a_m)=f_a(x_1,...,x_m)\] for all $x_i \in p^{\ulcorner r \urcorner}\Z_p$. Here $\ulcorner r \urcorner$ means the smallest integer that is greater or equal to $r$.
 
 (2) Let $C^{r-an}(\Z_p^{m},Q)$ (resp $C^{r-an}((\Z_p^{\times})^m,Q))$ denote the space of $r$-analytic functions.
 
 (3) Let $C^{r-an}(\Z_p^{m},Q^{\circ})$ (resp $C^{r-an}((\Z_p^{\times})^m,Q^{\circ}))$  denote the subset of   $C^{r-an}(\Z_p^{m},Q)$ (resp $C^{r-an}((\Z_p^{\times})^m,Q))$ consisting of $Q^{\circ}$-valued functions.
 
 \end{defn}
 
 Further we define the $r$-analytic weights:
 
 \begin{defn}
 
 (1) A weight $(R_{\U},\ka_{\U})$ is \textbf{$r$-analytic} if it is $r$-analytic as a function \[ \ka_{\U}:(\Z_p^{\times})^n\rightarrow R_{\U}^{\times}\subset \C_p\widehat{\otimes} R_{\U}\] through the isomorphism \[T_M(\Z_p)\cong (\Z_p^{\times})^n.\]
 
 (2) For a weight $(R_{\U},\ka_{\U})$, we use $r_{\U}$ to denote the smallest positive integer $r$ such that the weight is $r$-analytic.
 
 \end{defn}
 
 \begin{rem}
 
 (1) There is a standard fact that each continuous weight $\Z_p^{\times}\rightarrow R_{\U}^{\times}$ is $r$-analytic for sufficiently large $r$. In addition, if it is $r$-analytic, then it extends to a larger character \[\Z_p^{\times}(1+p^{r+1}\mathscr{O}_{\C_p})\rightarrow (\Ol_{\C_p} \widehat{\otimes} R_{\U}^+)^{\times} .\] See proposition 2.6 of \cite{chj} for more details.
 
 (2) Moreover, for $\ka_{\U}=(\ka_{1},...,\ka_{n})$, it is $r$-analytic if and only if each $\ka_i$ is $r$-analytic. In particular, for any $w \in \Q_{>1}$ with $w >1+r_{\U}$, the weight $\ka_{\U}$ extends to a character \[\ka_{\U}: T_{M,0}^{(w)}\rightarrow (\Ol_{\C_p} \widehat{\otimes} R_{\U}^+)^{\times}.\]
 
 \end{rem}
 
  Recall the twist $\ka_{\U}^{\vee}=w_{0,M}(\ka_{\U})=(-\ka_{U,n},\cdots,-\ka_{U,1})$ (section \ref{flag variety}),  now we introduce \textbf{$r$-analytic induction}:
 
 \begin{defn}
 Let $Q$ be a uniform $\C_p$-Banach algebra.
 
 (1) A function $f:N_{M,1}^{opp}\rightarrow Q$ is called \textbf{$r$-analytic} if under the isomorphism \[N_{M,1}^{opp}=\begin{pmatrix}
  1 \\
  p\Z_p & 1 \\
  ... & ... & ... \\
  p\Z_p & ... & p\Z_p & 1
  \end{pmatrix}\cong \Z_p^{\frac{n(n-1)}{2}},\] the function is $r$-analytic. Let $C^{r-an}(N_{M,1}^{opp},Q)$ denote the space of such functions.
 
 (2) Let $(R_{\U},\ka_{\U})$ be an $r$-analytic weight. Extend $\ka_{\U}$ to a group map $\ka_{\U}:B_{M,0}\rightarrow R_{\U}^{\times}$ by setting $\ka_{\U}|_{N_{M,0}}=1$. Define the following space (\textbf{$r$-analytic induction}):
 
 \[C_{\ka_{\U}}^{r-an}(Iw_M,Q):=\left\{f:Iw_M \rightarrow Q
\ \middle| \
\begin{aligned}
& f(\gamma \beta)=\ka_{\U}^{\vee}(\beta)f(\gamma), \quad \forall \gamma \in Iw_M, \beta \in B_{M,0},\\
& \text{$f|_{N_{M,1}^{opp}}$ is  $r$-analytic}.
\end{aligned}
\right \}.\]
 
 (3) Let $C_{\ka_{\U}}^{r-an}(Iw_M,Q^{\circ})$ denote the subset of $C_{\ka_{\U}}^{r-an}(Iw_M,Q)$ consisting of $Q^{\circ}$-valued functions.
 
 \end{defn}
 
 In term of notations about inductions in  section \ref{flag variety}, we may also write  \[C_{\ka_{\U}}^{r-an}(Iw_M,Q)=Ind_{B_{M,0}}^{Iw_M,r-an}(\ka_{\U}^{\vee},Q).\] We will use the left simplified symbol from now on.
 
 \begin{rem}
 
 We have the natural isomorphism  (through restriction) between $\C_p$-Banach algebras: \[C_{\ka_{\U}}^{r-an}(Iw_M,Q) \cong C^{r-an}(N_{M,1}^{opp},Q).\] The right side doesn't involve the weight $\ka_{\U}$. But different $\ka_{U}$ will produce different representations of $Iw_M$ on this space. This is analogous to the fact that $\Q_p^{\times}$ can act on $\Q_p$ through different characters.
 
 \end{rem}
 
 Let $\ka_{\U}$ be a weight and let $w \in \Q_{>1}$ with $w>r_{\U}+1$. The character $\ka_{\U}$ extends to a character on $T_{M,0}^{(w)}$, As $B^{(w)}_{M,0}=T_{M,0}^{(w)}N_{M,0}^{(w)}$, $\ka_{\U}$ further extends to a character on $B^{(w)}_{M,0}$ via trivial extension (set $\ka_{\U}|_{N_{M,0}^{(w)}}=1$).
 
 Moreover, each function $f \in C_{\ka_{\U}}^{r-an}(Iw_M,Q)$ naturally extends to a function \[f:Iw_{M}^{(w)}\rightarrow Q,\] \[f(\gamma \beta)=\ka_{\U}(\beta)f(\gamma), \quad \forall \gamma \in Iw_M^{(w)}, \beta \in B_{M,0}^{(w)} .\] This fact follows from the following decomposition directly: \[Iw^{(w)}_M=N_{M,1}^{opp,(w)}T_{M,0}^{(w)}N_{M,0}^{(w)}.\] In this way we get the natural identification \[C_{\ka_{\U}}^{r-an}(Iw_M,Q)=\ cong C_{\ka_{\U}}^{r-an}(Iw_M^{(w)},Q),\] where the later one is $r$-analytic induction for $Iw^{(w)}_M$ defined in the same way.
 
 Similarly we have such extension for elements in $C_{\ka_{\U}}^{r-an}(Iw_M,Q^{\circ})$.
 
 In particular, the space of $r$-analytic induction has a \textbf{left} action by $Iw^{(w)}_M$:
 
 \begin{defn}
 
 For $w \in \Q_{>1}$ with $w>1+r_{\U}$, there is a natural \textbf{left} action of $Iw_M^{(w)}$ on $C_{\ka_{\U}}^{r-an}(Iw_M,Q)$ via left translation: 
 
 \[(\gamma_1 \cdot f)(\gamma_2)=f(\gamma_1^{-1}\gamma_2)\]
 
 for any $\gamma_1, \gamma_2 \in Iw_M^{(w)}$ and $f \in C_{\ka_{\U}}^{r-an}(Iw_M,Q)$. We use $\rho_{\ka_{\U}}$ to denote this action. 
 
 Similarly we define the left action $\rho_{\ka_{\U}}$ of $Iw_M^{(w)}$ on $C_{\ka_{\U}}^{r-an}(Iw_M,Q^{\circ})$ 
 
 \end{defn}
 
Recall the natural maps $\pi_{Iw}:\X_{\infty,w}\rightarrow \X_{Iw,w}$, $\pi_{Iw^+}:\X_{\infty,w}\rightarrow \X_{Iw^+,w}$. Now we are ready to define the sheaf of \textbf{perfectoid automorphic forms}.
 
 \begin{defn}
 
 Let $(R_{\U},\ka_{\U})$ be a weight and let $w \in \Q_{>1}$ with $w>1+r_{\U}$.
 
 (1) Let $\Ol_{\X_{\infty,w}}\widehat{\otimes} R_{\U}$ be the sheaf on $\X_{\infty,w}$ given by \[\Y\mapsto \Ol_{\X_{\infty,w}}(\Y)\widehat{\otimes} R_{\U}\] for each affinoid open subspace $\Y\subset \X_{\infty,w}$. This is a sheaf of uniform $\C_p$-Banach algebra.
 
 Similarly, let be the sheaf on $\Ol_{\X_{\infty,w}^{+}}\widehat{\otimes} R_{\U}^{+}$ be the sheaf on $\X_{\infty,w}$ given by \[\Y\mapsto \Ol_{\X_{\infty,w}}(\Y)^{+}\widehat{\otimes} R_{\U}^{+}\] for each affinoid open subspace $\Y\subset \X_{\infty,w}$.
 
 (2) For any $r \in \Q_{>1}$ with $r>1+r_{\U}$, let $\mathscr{C}_{\ka_{\U}}^{r-an}(Iw_M, \Ol_{\X_{\infty,w}}\widehat{\otimes} R_{\U})$ denote the sheaf on $\X_{\infty,w}$ given by \[\Y\mapsto C_{\ka_{\U}}^{r-an}(Iw_M, \Ol_{\X_{\infty,w}}(\Y)\widehat{\otimes} R_{\U}) \] for each affinoid open subspace $\Y\subset \X_{\infty,w}$. It is also a sheaf of uniform  $\C_p$-Banach algebra. Similarly we define the sheaf $\mathscr{C}_{\ka_{\U}}^{r-an}(Iw_M, \Ol_{\X_{\infty,w}}^{+}\widehat{\otimes} R_{\U}^{+})$.

(3) The \textbf{sheaf of $w$-overconvergent perfectoid automorphic forms of Iwahori level and weight $\ka_{\U}$} is a subsheaf $\underline{\omega}_{w}^{\ka_{\U}}$ of $\pi_{Iw,*}\mathscr{C}_{\ka_{\U}}^{r-an}(Iw_M, \Ol_{\X_{\infty,w}}\widehat{\otimes} R_{\U})$ defined as follow: For each affinoid open subspace $\V\subset \X_{Iw,w}$ with $\V_{\infty}=\pi_{Iw}^{-1}(\V)$, we define \[\underline{\omega}_{w}^{\ka_{\U}}(\V):=\left\{f \in C_{\ka_{\U}}^{w-an}(Iw_M, \Ol_{\X_{\infty,w}}(\V_{\infty})\widehat{\otimes} R_{\U}):\begin{aligned}
&\gamma^*f=\rho_{\ka_{\U}}(\z\gamma_b+\gamma_d)^{-1}f,\\
& \forall \gamma=\begin{pmatrix}\gamma_a & \gamma_b \\ \gamma_c & \gamma_d  \end{pmatrix}\in Iw.
\end{aligned}\right\}.\] Here $\gamma^{*}f$ means the left action of $\gamma$ on $\mathscr{O}_{\X_{\infty,w}}$ induced by the right $Iw$-action on $\X_{\infty,w}$.

Similarly, the \textbf{sheaf of integral $w$-overconvergent perfectoid automorphic forms of Iwahori level and weight $\ka_{\U}$} is a subsheaf $\underline{\omega}_{w}^{\ka_{\U},+}$ of $\pi_{Iw,*}\mathscr{C}_{\ka_{\U}}^{w-an}(Iw_M, \Ol_{\X_{\infty,w}}^+ \widehat{\otimes} R_{\U}^+)$ defined as follow: For each affinoid open subspace $\V\subset  \X_{Iw,w}$ with $\V_{\infty}=\pi_{Iw}^{-1}(\V)$, we define 
\[\underline{\omega}_{w}^{\ka_{\U},+}(\V):=\left\{f \in C_{\ka_{\U}}^{w-an}(Iw_M, \Ol_{\X_{\infty,w}}^+(\V_{\infty})\widehat{\otimes} R_{\U}^+):\begin{aligned}
&\gamma^*f=\rho_{\ka_{\U}}(\z\gamma_b+\gamma_d)^{-1}f,\\
& \forall \gamma=\begin{pmatrix}\gamma_a & \gamma_b \\ \gamma_c & \gamma_d  \end{pmatrix}\in Iw.
\end{aligned}\right\}.\]
 \end{defn}
 
 (4) The \textbf{space of  $w$-overconvergent perfectoid automorphic forms of Iwahori level and weight $\ka_{\U}$} is defined to be \[M_{Iw,w}^{\ka_{\U}}:=H^0(\X_{Iw,w}, \underline{\omega}_{w}^{\ka_{\U}}).\] Similarly the  \textbf{space of integral $w$-overconvergent automorphic forms of Iwahori level and weight $\ka_{\U}$} is defined to be \[M_{Iw,w}^{\ka_{\U},+}:=H^0(\X_{Iw,w}, \underline{\omega}_{w}^{\ka_{\U},+}).\]
 
 (5) The \textbf{space of  perfectoid automorphic forms of Iwahori level and weight $\ka_{\U}$} is defined to be the following limit over $w$:
 \[M^{\ka_{\U}}_{Iw}:=\lim_{ w \rightarrow \infty} M^{\ka_{\U}}_{Iw,w}.\]
 Similarly  the \textbf{space of  perfectoid automorphic forms of Iwahori level and weight $\ka_{\U}$} is: \[M^{\ka_{\U},+}_{Iw}:=\lim_{ w \rightarrow \infty} M^{\ka_{\U},+}_{Iw,w}.\]
 
 (6) In the same way, we define the  \textbf{sheaf of $w$-overconvergent perfectoid automorphic forms of}  \textbf{strict Iwahori level} a\textbf{nd weight} $\ka_{\U}$ as $\underline{\omega}_{w}^{\ka_{\U}}$ (use we use same notation) over $\X_{Iw^+,w}$, \textbf{the  space of  $w$-convergent perfectoid automorphic forms of} \textbf{strict Iwahori level} \textbf{and weight} $\ka_{\U}$ as $M^{\ka_{\U}}_{Iw,w}$,\textbf{ the  space of  perfectoid automorphic forms of}  \textbf{strict Iwahori level}\textbf{ and weight} $\ka_{\U}$ as  $M^{\ka_{\U}}_{Iw^+,w}$ and their \textbf{integral }analogues $\underline{\omega}_{w}^{\ka_{\U},+}$, $M^{\ka_{\U},+}_{Iw^+,w}$, $M^{\ka_{\U},+}_{Iw^+}$.
 
 \begin{rem}
 
 Indeed in the first half of this paper (establish theory of perfectoid automorphic forms), we will mainly work with Iwahori level. Only in the theory  about overconvergent Eichler-Shimura map we will work with strict Iwahori level.
 
 \end{rem}
 
 \begin{rem}
 
 In the previous construction of overconvergent automorphic forms, there are two basic numbers $r_1$ and $r_2$ measuring "distance" to ordinary locus and analytic radius. For example see \cite{AIP2015}, \cite{shen2016} and \cite{brasca2016}. Here for simplicity we use a single number $w$ representing these two quantities. 
 
 \end{rem}

  The perfectoid automorphic forms is indeed a kind of $Iw$-invariant (resp. $Iw^+$-invariant) under a twisted action of $Iw$ (resp. $Iw^+$). We introduce the following \textbf{twisted left action} of $Iw$ on $\mathscr{C}_{\ka_{\U}}^{w-an}(Iw_M, \Ol_{\X_{\infty,w}}\widehat{\otimes} R_{\U})$ by \[\gamma . f:= \rho_{\ka_{\U}}(\z \gamma_b+\gamma_d) \gamma^* f.\] We make the following remark, which is one important advantage of perfectoid methods and is useful in the construction of overconvergent Eichler-Shimura map (see section \ref{p-adic eichler shimura} for more details).
 
 \begin{rem}
 \label{twist action invariant}
 
 The subsheaf $\omega_{w}^{\ka_{\U}}$ is exactly the $Iw$-invariant (resp. $Iw^+$-invariant) of sheaf $\pi_{Iw,*}\mathscr{C}_{\ka_{\U}}^{w-an}(Iw_M, \Ol_{\X_{\infty,w}}\widehat{\otimes} R_{\U})$ (resp. $\pi_{Iw^+,*}\mathscr{C}_{\ka_{\U}}^{w-an}(Iw_M, \Ol_{\X_{\infty,w}}\widehat{\otimes} R_{\U})$) under the twisted action.

 \end{rem}
 
\begin{rem}
As we mentioned in the introduction, the perfectoid automorphic forms is a kind of $p$-adic analogue of complex automorphic forms. The latter is expressing automorphic forms as certain functions on Hermitian symmetric domain satisfying transformation law, see chapter I of \cite{shimura2000book}. Also see proposition \ref{classical form} about classical forms.
\end{rem}

 \subsection{Hecke operators}
 \label{hecke operator}
 In this section we define Hecke operators acting on the space of perfectoid automorphic forms. For simplicity we will work with Iwahori level $Iw$. The same method applies to strict Iwahori level $Iw^+$ directly.
 
 Let $(R_{\U},\ka_{\U})$ be a weight and $w>1+r_{\U}$.
 
 \textbf{Hecke operators outside $p$.} First we define tame Hecke operators (for good primes) through Hecke correspondence.
 
 Recall the global level subgroup $K=K^{p}\times K_p$ and $K^{p}=K^{S_0}\times K_{S_0}$. Let $l$ be a good prime ($l \notin S_0 \cup \{p\}$). The local unramified Hecke algebra for $G(\Q_l)$ (with hyperspecial subgroup $K_l$) is commutative. For any $\gamma \in G(\Q_l)$, we will define a Hecke operator $T_{\gamma}$ corresponding to the double coset $[K_l \gamma K_l]$.  
 
 The Shimura variety $X_{Iw \cap \gamma Iw \gamma^{-1} }$ have two finite etale maps to $X_{Iw}$, one is the natural map $pr_1$ (through level subgroup inclusion) and another one $pr_2$ is induced by translation via $\gamma$. In other words, we have the following diagram of here is a diagram of Hecke correspondence:
 
 \[
\xymatrix{
& X_{Iw \cap \gamma Iw \gamma^{-1} } \ar[dl]_{pr_1} \ar[dr]^{pr_2} & \\
X_{Iw} & & X_{Iw}
}
\]

For simplicity let $X_{Iw,\gamma}$ denote  $X_{Iw \cap \gamma Iw \gamma^{-1} }$ and define $\X_{Iw,\gamma,w}:=pr_1^{-1}(\X_{Iw,w})$. Similarly we still have the Hecke correspondence:  
 
\[
\xymatrix{
& \X_{Iw,\gamma,w} \ar[dl]_{pr_1} \ar[dr]^{pr_2} & \\
\X_{Iw,w} & & \X_{Iw,w}
}
\] 

To define the Hecke operator $T_{\gamma}$, we first produce a natural isomorphism \[\psi_{\gamma}: pr_2^*\underline{\omega}_{w}^{\ka_{\U}} \cong pr_1^* \underline{\omega}_{w}^{\ka_{\U}}.\]

Pullback the above diagram of Hecke correspondence to infinite level, we get 

\[
\xymatrix{
& \X_{\infty,\gamma,w} \ar[dl]_{pr_{1,\infty}} \ar[dr]^{pr_{2,\infty}} & \\
\X_{\infty,w} & & \X_{\infty,w}
}
\] 

and consequently a natural $Iw$-equivariant isomorphism \[\psi_{\gamma,\infty}: pr_{2,\infty}^* \Ol_{\X_{\infty,w}} \cong pr_{1,\infty}^* \Ol_{\X_{\infty,w}}.\]

It further induces an isomorphism \[\psi_{\gamma,\infty}: C_{\ka_{\U}}^{w-an}(Iw_M,pr_{2,\infty}^* \Ol_{\X_{\infty,w}}\widehat{\otimes} R_{\U}) \cong C_{\ka_{\U}}^{w-an}(Iw_M, pr_{1,\infty}^* \Ol_{\X_{\infty,w}}\widehat{\otimes} R_{\U} ).\]

Recall the coordinate $\z$ via pullback $\pi_{HT}:\X_{\infty,w}\rightarrow \Fl^{\times}_w$. Let $\z^{'}=pr_{1,\infty}^*\z$ and $\z^{''}=pr_{1,\infty}^*\z$. As $\pi_{HT}$ is purely about information at $p$. we have $\z^{'}=\z^{''}$. Then the above isomorphism $\psi_{\gamma,\infty}$ is further equivariant under twisted $Iw$-action. Because the perfectoid automorphic sheaf is indeed such invariant under twisted $Iw$-action (see remark \ref{twist action invariant}), we get the desired isomorphism:

\[\psi_{\gamma}: pr_2^*\underline{\omega}_{w}^{\ka_{\U}} \cong pr_1^* \underline{\omega}_{w}^{\ka_{\U}}.\]

Finally we can define the Hecke operator

\[
T_{\gamma}: \xymatrix{
H^0(\X_{Iw,w}, \underline{\omega}_{w}^{\ka_{\U}}) \ar[r]^(0.5){pr_2^*} & H^0(\X_{Iw,\gamma,w}, pr_2^*\underline{\omega}_{w}^{\ka_{\U}}) \ar[ld]_{\psi_{\gamma}} & \\
 H^0(\X_{Iw,\gamma,w}, pr_1^*\underline{\omega}_{w}^{\ka_{\U}}) \ar[r]^(0.5){Tr_{pr_1}} & H^0(\X_{Iw,w}, \underline{\omega}_{w}^{\ka_{\U}}).
}
\]
 
 \textbf{Hecke operators at $p$.} Now we turn to define Hecke operators at $p$. This is more subtle than tame case.
  
 For $0 \leq i \leq n-1$, we consider the matrices $u_{p,i} \in G(\Q_p)$ defined by  
 
 \[u_{p,i}:=\begin{pmatrix}p \\ & p\mathbb{I}_i \\ & & \mathbb{I}_{n-i} \end{pmatrix}\]
 
 We further write \[u_{p,i}=\begin{pmatrix}u_{p,i,+} \\ & u_{p,i,-} \end{pmatrix}\] where $u_{p,i,+}$ is the $1 \times 1$-matrix $p$ and $u_{p,i,-}$ is the $n \times n$-matrix.

 \begin{rem}
 
 For general unitary Shimura variety (e.g. see \cite{brasca2016}) with signature $(a,b)$, we will use similar convention. And $u_{p,i,+}$ will be an $a \times a$-matrix and $u_{p,i,-}$ will be a $b \times b$-matrix. We will discuss such generalizations in later papers of this series.
 \end{rem}
 
 Through direct computations, both $\Fl_{w}^{\times}$ and $\X_{\infty,w}$ are stable under $u_{p,i}$-action. Moreover, $u_{p,0}$ sends $\X_{\infty,w}$ into $\X_{\infty,w+1}$.
 
 Recall the twisted action of $Iw$ on $C_{\ka_{\U}}^{w-an}(Iw_M,\Ol_{\X_{\infty,w}}\widehat{\otimes} R_{\U})$ given by  
 
 \[\gamma . f:= \rho_{\ka_{\U}}(\z \gamma_b+\gamma_d) \gamma^* f.\]
 
 Now we define the Hecke operators $U_{p,i}$.
 
 \begin{defn}
 
 (1) We define $u_{p,i}$-action on  $ C_{\ka_{\U}}^{w-an}(Iw_M, \Ol_{\X_{\infty,w}}\widehat{\otimes} R_{\U})$  by \[(u_{p,i}.f)(\gamma):=u_{p,i}^*f(u_{p,i,-}^{-1}\gamma_0u_{p,i,-}\beta),\] where $f \in C_{\ka_{\U}}^{r-an}(Iw_M, \Ol_{\X_{\infty,w}}\widehat{\otimes} R_{\U})$, $\gamma=\gamma_0 \beta \in Iw_M$ with $\gamma_0 \in N^{opp}_{M,1}$ and $\beta \in B_{M,0}$.
 
 (2) For $f \in C_{\ka_{\U}}^{w-an}(Iw_M, \Ol_{\X_{\infty,w}}\widehat{\otimes} R_{\U})$ satisfies \[\gamma.f=f\] for any $\gamma \in Iw$. Pick a decomposition \[Iw u_{p,i} Iw= \bigsqcup_{j} \delta_{i,j}u_{p,i}Iw\] with $\delta_{i,j} \in Iw$. Define \[U_{p,i}(f):=\sum_{j}\delta_{i,j}.(u_{p,i}.f) \in  C_{\ka_{\U}}^{r-an}(Iw_M, \Ol_{\X_{\infty,w}}\widehat{\otimes} R_{\U}). \]

 \end{defn}

 The following proposition shows that $U_{p,i}$ is well defined. Its proof also gives  a better interpretation for   $u_{p,i}$-action: formally it is   \textbf{same} as the twisted  action defined by $Iw$. 
 
 \begin{prop}
 \label{define up}
  The operator $U_{p,i}$ is well defined. In other words, it is independent of choice the representatives $\delta_{i,j}$. 

 \end{prop}
 
 \begin{proof}
 
 First we give a \textbf{unified} interpretation of $u_{p,i}$-action and $Iw$-action.
 
 Consider the exact sequence \[1\rightarrow T_{M}(\Z_p)\rightarrow T_{M}(\Q_p)\rightarrow X_*(T_M)\rightarrow 1\] where $X_*(T_M) \cong \Z^n$ is the group of algebraic cocharacters. Consider the following map which splits this exact sequence \[ X_*(T_M) \rightarrow T_M(\Q_p),\] \[\lambda \mapsto \lambda(p).\] We obtain the following decomposition \[T_{M}(\Q_p) \cong T_M(\Z_p) \times X_*(T_M).\] Through this decomposition, any $\ka$ for $T_{M}(\Z_p)$ extends trivially to $T_{M}(\Q_p)$.  Because $B_M(\Q_p)=T_M(\Q_p)N_M(\Q_p)$, the original weight $\ka_{\U}$ further extends trivially to $B_M(\Q_p)$.

 Let $\widetilde{Iw_M}$ denote the product $Iw_M B_M(\Q_p)$. The natural inclusions \[N_{M,1} \hookrightarrow Iw_M \hookrightarrow \widetilde{Iw_M}\] induces natural identifications \[N_{M,1}^{opp}\cong Iw_{M}/B_{M}(\Z_p) \cong \widetilde{Iw_M}/B_M(\Q_p).\] Then for the same weight $\ka_{\U}$, the induction from $B_{M}(\Z_p)$ to $Iw_M$ is the same as the induction from $B_M(\Q_p)$ to $\widetilde{Iw_M}$. In other words, we have \[C_{\ka_{\U}}^{r-an}(Iw_M, \Ol_{\X_{\infty,w}}\widehat{\otimes} R_{\U})=C_{\ka_{\U}}^{r-an}(\widetilde{Iw_M}, \Ol_{\X_{\infty,w}}\widehat{\otimes} R_{\U}).\]
 
What's more, $\widetilde{Iw_M}$ is stable under left multiplication of $u_{p,i}^{-1}$. 
 
 Now for any $f \in  C_{\ka_{\U}}^{r-an}(\widetilde{Iw_M}, \Ol_{\X_{\infty,w}}\widehat{\otimes} R_{\U})$, we can rewrite the $u_{p,i}$-action as follow:\[(u_{p,i}.f)(\gamma):=u_{p,i}^*(f)(u_{p,i}^{-1}\gamma).\]
 
 Further notice that the automorphic factor (in section \ref{flag variety}) formally extends \[J(\z,u_{p,i})=\z u_{p,i,b}+u_{p,i,d}=u_{p,i,-}.\] Then formally \[(u_{p,i}.f)(\gamma)=u_{p,i}^*(f)(J(\z,u_{p,i})^{-1}\gamma)=(\rho_{\ka_{\U}}(J(\z,u_{p,i}))u_{p,i}^*(f))(\gamma),\] which is exactly the \textbf{same} manner as twisted action by $Iw$. What's more, the right 1-cocycle property (lemma \ref{right 1cocycle}) of factor $J$ also extends.
 
 Now to show that $U_{p,i}$ is well defined, it is enough to check the following statement: 
 
 $\bullet$ Suppose $\delta u_{p,i}=u_{p,i}\gamma$ with $\delta$, $\gamma \in Iw$, then  for any $f \in C_{\ka_{\U}}^{r-an}(\widetilde{Iw_M}, \Ol_{\X_{\infty,w}}\widehat{\otimes} R_{\U})$, we have $\delta.(u_{p,i}.f)=u_{p,i}.f$
 
 Just observe that $\delta.(u_{p,i}.f)=(\delta u_{p,i}).f=(u_{p,i}\gamma).f=u_{p,i}.(\gamma.f)=u_{p,i}.f$, we're done.

 \end{proof}
 
Consequently we get the following lemma:
 
 \begin{lem}
Suppose $f \in C_{\ka_{\U}}^{w-an}(Iw_M, \Ol_{\X_{\infty,w}}\widehat{\otimes} R_{\U})$ which is invariant under twisted $Iw$-action. Then the resulting element $U_{p,i}(f)$ is also invariant under twisted $Iw$-action.
 
 \end{lem}
 
 \begin{proof}
 
 Still use the same decomposition  in the above definition for $U_{p,i}$ and pick $\gamma \in Iw$, we have \[\gamma.U_{p,i}(f)=\sum_{j}\gamma.(\delta_{i,j}.u_{p,i}.f)=\sum_j (\gamma \delta_{i,j}).(u_{p,i}.f).\] The last sum is also $U_{p,i}(f)$ as $\{\gamma \delta_{i,j}\}$ is another set of representatives. 
 
 \end{proof}
 
  In particular, we construct $U_{p,i}$-operators acting on $M_{Iw,w}^{\ka_{\U}}=H^(\X_{Iw,w},\underline{\omega}_{w}^{\ka_{\U}})$.
  
  In summary we give the following definition.
  \begin{defn}

 The \textbf{tame Hecke algebra }(outside $p$) is defined to be \[\mathbb{T}^{S_0}:=\Z_p[T_{\gamma}: \gamma \in G(\Q_l), l \notin S_0 \cup \{p\}]\]

 and the \textbf{total Hecke algebra} is \[\mathbb{T}:=\mathbb{T}^{S_0}\otimes \Z_p[U_{p,i}:0 \leq i  \leq n-1].\]

 \end{defn}

 Finally we define $U_p:=\prod_{i}U_{p,i}$ and this is a compact operator (such compactness is fundamental for overconvergent automorphic forms, e.g. see \cite{AIP2015} and \cite{shen2016}):
 
 \begin{prop}
 The operator $U_p$ is a compact operator on $M^{\ka_{\U}}_{Iw,w}$.
 
 \end{prop}
 
 \begin{proof}
 
 By straightforward computations, $U_p$ factors as \[H^0(\X_{Iw,w},\underline{\omega}_{w}^{\ka_{\U}})\rightarrow H^0(\X_{Iw,w+1},\underline{\omega}_{w}^{\ka_{\U}})\rightarrow H^0(\X_{Iw,w+1},\underline{\omega}_{w-1}^{\ka_{\U}}) \rightarrow H^0(\X_{Iw,w},\underline{\omega}_{w}^{\ka_{\U}}) ,\]
 
 The first map (natural restriction) is a compact operator and the last map is also compact (see \cite{hansen2017universal} section 2.2). Therefore $U_p$ is a compact operator.
 
 \end{proof}

 \begin{rem}

In work of   Andreatta-Iovita-Pilloni \cite{AIP2015} and Xu Shen \cite{shen2016}, their Hecke operators  at $p$ come with  certain normalization factors  (i.e. they are multiplied by suitable powers of $p$). Our operators are not normalized.  This is the only difference of two constructions. In this paper we willy mainly work with generic fiber (over $\Q_p$ or $\C_p$), such normalization factor is irrelevant.  Its role becomes essential at integral setting (e.g. \cite{pilloniintegralhecke} and \cite{boxerpillonimodularcurve}). I plan to investigate this subtle issue about integrality in the future.

 \end{rem}
 
\subsection{Classical automorphic forms}
\label{section classical form}
In this section we interpret classical automorphic forms in a similar manner of perfectoid automorphic forms via infinite level Shimura varieties. This also provides a toy example for the equivalence between perfectoid automorphic forms and overconvergent automorphic forms via Andreatta-Iovita-Pilloni methods (see section \ref{compare forms}).

Let $\ka=(\ka_1,\cdots,\ka_n) \in \Z^{n}$ be a dominant weight, i.e. $\ka_1 \geq \ka_2 \cdots \geq \ka_n$. Consider the algebraic (\textbf{right}) $GL_n$-torsor over  $\X_{Iw}$ \[\pi: \mathcal{M}=Isom_{\X_{Iw}}(\Ol ^{n}, \underline{\omega}_{Iw})\rightarrow \X_{Iw}.\] Let $\pi_{*}\Ol_{\mathcal{M}}[\ka^{\vee}]$ denote the subsheaf  where $B_{M}$  acts through the character $\ka^{\vee}$. 

\begin{defn}  
(1) The sheaf  of \textbf{classical automorphic forms of weight $\ka$ with Iwahori level} is defined to be \[\underline{\omega}_{Iw}^{\ka}:=\pi_{*}\Ol_{\mathcal{M}}[\ka^{\vee}].\]    The space of \textbf{classical automorphic forms of weight $\ka$ with Iwahori level} is \[M_{Iw}^{\ka,cl}:=H^0(\X_{Iw}, \underline{\omega}_{Iw}^{\ka} ).\]

(2) The sheaf of \textbf{integral classical automorphic forms of   $\ka$ with Iwahori level} is \[\underline{\omega}_{Iw}^{\ka,+}:=\pi_{*}\Ol_{\mathcal{M}}^+[\ka^{\vee}].\]  The space of \textbf{integral classical automorphic forms of weight $\ka$ with Iwahori level} is \[M_{Iw}^{\ka,cl,+}:=H^0(\X_{Iw}, \underline{\omega}_{Iw}^{\ka,+} ).\]

\end{defn}

From the viewpoint of representation theory, the automorphic sheaf $\underline{\omega}_{Iw}^{\ka}$ corresponds to  the highest weight representation $\Vc_{\ka}$ (where $\ka$ is the  \textbf{representation weight}) of $GL_{n}$. See section 2.2 of \cite{shen2016} and chapter 8 of \cite{hida2004book} for more details.

To relate such concept with infinite level Shimura varieties, we further introduce the following notations.

\begin{defn}

(1) Let $P(GL_n,\A^1)$ denote the $\Q_p$-vector space of (polynomial) maps $GL_n\rightarrow \A^1$ between algebraic varieties over $\Q_p$.

(2) For each uniform $\C_p$-Banach algebra $Q$, define \[P(GL_n,Q):=P(GL_n,\A^1)\widehat{\otimes}_{\Q_p}Q\] and $P_{\ka}(GL_n,Q)$ denote the subspace consisting of maps $f: GL_n \rightarrow Q$ with \[f(\gamma \beta)=\ka^{\vee}(\beta)f(\gamma), \forall \gamma \in GL_n, \beta \in B_M.\]

(3) There is a natural \textbf{left} action of $M=GL_n$ on $P_{\ka}(GL_n,Q)$ given by \[(\gamma_1 . f)(\gamma_2)=f(\gamma_1^{-1}\gamma_2)\] and we denote this representation by $\rho_{\ka}$.
\end{defn}

Now we have the following proposition expressing classical automorphic forms via infinite level Shimura varieties. As we mentioned in the introduction, the strategy is indeed similar to complex setting. The observation is that through the $p$-adic uniformization map \[\X_{\infty,w} \rightarrow \X_{Iw,w},\] we can trivialize the $GL_n$-torsor $\mathcal{M}$ after pulling back to infinite level. This method will also be used in establishing equivalence between two constructions of overconvergent automorphic forms (see section \ref{compare forms}).

\begin{prop}
\label{classical form}
For any affinoid open $\V \subset \X_{Iw,w}$ with preimage $\V_{\infty}$ on $\X_{\infty,w}$,  we have the natural identification \[\underline{\omega}_{Iw}^{\ka}(\V)=\left\{f \in P_{\ka}(GL_n, \Ol_{\X_{\infty,w}}(\V_{\infty})) :\begin{aligned}
&\gamma^*f=\rho_{\ka }(\z\gamma_b+\gamma_d)^{-1}f,\\
& \forall \gamma=\begin{pmatrix}\gamma_a & \gamma_b \\ \gamma_c & \gamma_d  \end{pmatrix}\in Iw.
\end{aligned}\right\}.\] This further implies a natural inclusion over $\X_{Iw,w}:$ \[ \underline{\omega}_{Iw}^{\ka} \hookrightarrow \underline{\omega}_{w}^{\ka}.\]

\end{prop}

\begin{proof}

Consider the following pullback diagram

\[
\xymatrix{
\mathcal{M}_{\infty} \ar[r] \ar[d]_{\pi_{\infty}} & \mathcal{M} \ar[d]_{\pi} \\
\X_{\infty,w} \ar[r] & \X_{Iw,w}
}
\]

Recall that we have constructed sections $\s=(\s_1,\cdots,\s_n)$ (through the flag variety $\Fl$) for the $n$-dimensional vector bundle $\underline{\omega}_{\infty}$. Moreover, at each $Spa(\C_p,\Ol_{\C_p})$ point of $\X_{\infty,w}$, these sections generate the fiber ($n$-dimensional vector space). Therefore the $GL_n$-torsor  $\mathcal{M}_{\infty}$ is trivialized by this section $\s=(\s_1,...,\s_n)$. In other words, we have the following commutative diagram
\[
\xymatrix{
\X_{\infty,w}\times GL_n \ar[rr]^{\delta}_{\cong} \ar[rd] & & \mathcal{M}_{\infty} \ar[ld] \\
& \X_{\infty,w} &
}
\]
 where $GL_n$ acts on $\X_{\infty,w}\times GL_n$ via right multiplication on the second factor and $\delta$ is a $GL_n$-equivariant isomorphism.
 
In particular, we have the following isomorphism \[\pi_{\infty,*} \Ol_{\mathcal{M}_{\infty}} (\V_{\infty})\cong P(GL_n,\Ol_{\X_{\infty,w}}(\V_\infty)).\]

Under the natural map $\pi_{Iw}: \X_{\infty,w}\rightarrow \X_{Iw,w}$, the sheaf $\pi_{*}\Ol_{\mathcal{M}}$ is exactly the $Iw$-invariant of $\pi_{Iw,*}(\pi_{\infty,*} \Ol_{\mathcal{M}_{\infty}} )$.  From the definition of classical automorphic sheaf we get the equality \[\underline{\omega}_{Iw}^{\ka}(\V)= P_{\ka}(GL_n,\Ol_{\X_{\infty,w}}(\V_{\infty}))^{Iw}.\]

The remaining task is to figure out this $Iw$-action on $\X_{\infty,w}\times GL_n$. The isomorphism $\delta$ is over infinite level and the action of $Iw$ is \textbf{twisted} (also influences the second factor $GL_n$). Recall the transform rule (due to lemma \ref{transform law}) \[\gamma^*(\s)=\s (\z \gamma_b+\gamma_d)\] for any $\gamma=\begin{pmatrix}\gamma_a & \gamma_b \\ \gamma_c & \gamma_d  \end{pmatrix}\in Iw.$ This implies that the action of $Iw$ on $P_{\ka}(GL_n,\Ol_{\X_{\infty,w}}(\V_{\infty}))$ is given by \[\gamma \diamond f = \rho_{\ka}(\z\gamma_b+\gamma_d) \gamma^*f,\] which is exactly the same twisted action in defining perfectoid automorphic forms. The first statement holds.

The second statement follows from the natural inclusion \[P_{\ka}(GL_n,\Ol_{\X_{\infty,w}}(\V_{\infty})) \hookrightarrow \mathscr{C}_{\ka}^{w-an}(Iw_M, \Ol_{\X_{\infty,w}}(\V_{\infty})).\]

\end{proof}

Moreover, through the study of irreducible components of $X_{Iw}$, $\X_{Iw}$ and $\X_{Iw,w}$, apply the   method in the proof of lemma 3.4.4 in \cite{drw}, we further deduce the following injectivity:

\begin{prop}
The following composition of Hecke-equivariant maps \[M^{\ka,cl}_{Iw}=H^0(\X_{Iw},\underline{\omega}_{Iw}^{\ka}) \xrightarrow{Res} H^0(\X_{Iw,w},\underline{\omega}_{Iw}^{\ka}) \hookrightarrow M_{Iw,w}^{\ka}\] is injective.
\end{prop}

\section{Comparison with overconvergent automorphic forms}
\label{section3}

In previous section we have constructed perfectoid automorphic forms together with Hecke operators acting on it. The main goal of this section is compare it with previous construction of overconvergent automorphic forms. 

In \cite{AIP2015} Andreatta-Iovita-Pilloni established a novel method to construct   overconvergent automorphic forms for Siegel Shimura varieties. Their result has many important applications and their method has also been generalized to other Shimura varieties. In \cite{shen2016}, Xu Shen did such a generalization for compact unitary Shimura varieties with signature $(1,n)\times (0,n+1)\times...\times(0,n+1)$, which is closest to our setting. In \cite{brasca2016}, Riccadro Brasca further generalized such methods to any  PEL type unitary  Shimura varieties. In this paper we will compare our construction with Xu Shen's results. In later paper of this series, we will generalize to other unitary Shimura varieties and compare with Brasca's results.

At first glance, the two constructions appear quite different, as they involve distinct sheaves defined on different loci. Nevertheless, they yield equivalent results. The key insight is to \textit{reinterpret the theory of canonical subgroups through the framework of the Hodge-Tate period map} $\pi_{HT}$. Using (pseudo-)canonical subgroups, we first compare the two kinds of loci-both of which measure the "distance" to ordinarity-and show that they define equivalent systems of neighborhoods of the canonical ordinary locus (section \ref{two locus}). We then compare the two constructions of overconvergent automorphic sheaves and deduce their equivalence (section \ref{compare forms}).

\subsection{Construction via Andreatta-Iovita-Pilloni's methods}
In this section we briefly describe Xu Shen's construction for overconvergent automorphic forms over unitary Shimura varieties. See his work \cite{shen2016} for more details. Also see Andreatta-Iovita-Pilloni's \cite{AIP2015} for their original methods in Siegel setting. We will further introduce certain \textit{generalized Igusa torsor}, which is helpful in comparison of two constructions (see section \ref{compare forms}).

Let $m$ be a positive integer and  $v \in \Q \cap [0,\frac{1}{2}]$ such that $v< \frac{1}{2p^{m-1}}$. Recall the moduli interpretation of the Shimura variety $\X$, each point $x$ will determine a one-dimensional $p$-divisible group $H_x$. Let $\widetilde{Ha}$ be a lift of Hasse invariant, we introduce the following adic open subspace \[\X(v):=\{x \in \X: |\widetilde{Ha}(H_{x})|\geq p^{-v}\}.\]  This locus is indeed independent of choice of such lift. The locus $\X(0)$ is the ordinary locus, parametrizes ordinary $p$-divisible groups $H_x$ (equivalently ordinary abelian varieties $A_x$). For positive $v$, the locus $\X(v)$ is strict neighborhood of the ordinary locus. Moreover, because $v < \frac{1}{2p^{m-1}}$, each $p$-divisible group $H$ has the \textbf{canonical subgroup of level $m$} (denoted by $C_m$) inside $H[p^m]$,  see theorem 3.1 in \cite{shen2016} and theorem 6 of \cite{fargues2011canonical} for basic results of such theory. In particular, $C_m(\C_p)\cong \Z/p^m$. Let $H^D$ be the Cartier dual of $H$ and $C_m^{\perp}$ denote the annihilator of $C_m$ under the natural pairing \[H[p^m]\times H^D[p^m]\rightarrow \mu_{p^m},\] then $\widetilde{Ha}(H)=\widetilde{Ha}(H^D)$ and $C_m^{\perp}$ is the canonical subgroup of level $m$ for $H^D$.

Let $\mathcal{H}$ denote the universal $p$-divisible group over $\X(v)$, $\mathcal{C}_m$ be its canonical subgroup of level $m$ and $\mathcal{C}_m^{\perp}$ be the corresponding canonical subgroup of level $m$ for $\mathcal{H}^D$. We further introduce the following space \[\X_1(p^m)(v):=Isom_{\X(v)}((\Z/p^m)^n,\mathcal{C}_m^{\perp,D}),\] it is a \textbf{right} $M(\Z/p^m)=GL_n(\Z/p^m)$-torsor over $\X(v)$ parametrizing  frames for $\mathcal{C}_m^{\perp,D}$. Define \[\X(p^m)(v):=\X_1(p^m)(v)/B_M(Z/p^m).\] Through moduli interpretations (see proposition 3.3 of \cite{shen2016}), $\X(p)(v)$ coincides with the locus $\X_{Iw}(v)$ inside $\X_{Iw}$, which parametrizes $p$-divisible groups $H$ with $|\widetilde{Ha}(H)|\geq p^{-v}$ with filtration $Fil_{\bullet}(H[p])$ such that the first piece is the canonical subgroup $C_1$. In particular, these loci $\X(p)(v)$ are strict neighborhoods of the canonical ordinary locus (multiplicative ordinary locus) inside $\X_{Iw}$. Let $\mathfrak{X}(v)$, $\mathfrak{X}_1(p^m)(v)$ and $\mathfrak{X}(p^m)(v)$ denote their corresponding formal models.

Let $w \in \Q \cap (0,m-v\frac{p^m}{p-1}]$, we further introduce properties of canonical subgroups and the concept of \textbf{$w$-compatibility}, which is one basic ingredient in Andreatta-Iovita-Pilloni methods (see section 4.5 of \cite{AIP2015}).

Let $R$ be an admissible $\Ol_{\C_p}$-algebra (see section 4.1 of \cite{AIP2015}) with $S$ be $Spec(R)$ and $S^{rig}$ be the rigid analytic space. Let $R_w$ denote the reduction $R\otimes_{\Ol_{\C_p}}\Ol_{\C_p}/p^w$. Let $H/S$ be a (relative) one dimensional $p$-divisible group with constant height $n+1$. Suppose for any point $x \in S^{rig}$, we have $|\widetilde{Ha}(H_x)| \geq p^{-v}$. Then $H$ has a finite flat subgroup scheme $C_m \subset H[p^m]$ which interpolates canonical subgroup of level $m$ $C_{m,x}$ for $H_x$. Shrink $R$ if necessary to ensure that $C_m^{\perp,D} (R) \cong (\Z/p^m)^n$  and $\omega_{H^D}$ is free over $S$. Then proposition 4.3.1 of \cite{AIP2015} (also see section 3.1 of \cite{shen2016}) produces the following sheaf $\Fg^{-}$:

$\bullet$ There is a free subsheaf $\Fg^{-}$ of $\omega_{H^D}$ which is equipped with a (truncated Hodge-Tate) map \[HT_{w}^{-}:(C_m^{\perp,D})(R)\rightarrow \Fg^{-}\otimes_{R}R_w\] produced from the Hodge-Tate period map \[HT_{C_m^{\perp,D}}: C_m^{\perp,D}(R)\rightarrow \omega_{C_m^{\perp}}.  \] It further induces an isomorphism \[HT_{w}\otimes Id: (C_m^{\perp,D})(R)\otimes_{\Z} R_w \rightarrow \Fg^{-}\otimes_{R}R_w.\]

Take an isomorphism (frame) $\psi^{-}: (\Z/p^m)^n \cong C_m^{\perp,D}(R) $, it  produces a basis \[\{\psi^{-}(e_1),\cdots,\psi^{-}(e_n)\}\] (here we are using standard basis of $(\Z/p^m)^n$) for $C_m^{\perp,D}(R)$ together with a filtration \[Fil_{\bullet}^{\psi^{-}}:0 \subset \langle \psi^{-}(e_1)\rangle \subset\langle\psi^{-}(e_1),\psi(e_2)\rangle \subset \cdots \subset \langle \psi^{-}(e_1),\cdots,\psi^{-}(e_n)\rangle.\] Set $x_i=\psi^{-}(e_i)$ and $\overline{x_i}$ denote the basis of each graded pieces.

On the other hand, let $Fil_{\bullet}\Fg ^{-}$ be a filtration for $\Fg^{-}$ (free module): \[Fil_0 \Fg^{-}=0 \subset Fil_1 \Fg^{-1} \subset \cdots \subset Fil_n \Fg^{-}=\Fg^{-} .\] Let $v_i$ ($1 \leq i \leq n$) be a basis for the graded piece $Fil_{i}\Fg^{-}/Fil_{i-1} \Fg^{-1}$. Let $\{\delta_i, 1 \leq i \leq n\}$   be a basis for $\Fg^{-}$. 

We say that the filtration $Fil_{\bullet}\Fg ^{-}$ is \textbf{$w$-compatible} with $\psi^{-}$ if \[Fil_{\bullet}\Fg^{-}\otimes R_w = HT_w(Fil^{\psi^{-}}_{\bullet})\otimes R_w.\] We say the tuple $(Fil_{\bullet}\Fg ^{-}, \{v_i\})$ is \textbf{$w$-compatible} if \[Fil_{\bullet}\Fg^{-}\otimes R_w = HT_w(Fil^{\psi^{-}}_{\bullet})\otimes R_w\] and \[v_i \text{ mod } (p^w \Fg^{-}+ Fil_{i-1} \Fg^{-})=HT_{w}(\overline{x_i}). \] We say the basis $\{\delta_i\}$ is \textbf{$w$-compatible} with $\{x_i\}$ if \[\delta_i \text{ mod  } p^w \Fg^{-}=HT_w(x_i). \]

Now recall the moduli space $\mathfrak{X}_1(p^m)(v)$ parametrizes frames of $C_{m}^{\perp,D}$ over $\mathfrak{X}(v)$. Under the concept of $w$-compatibility, there are formal schemes $\mathfrak{I}\mathfrak{W}_w$ parametrizing $w$-compatible filtration $Fil_{\bullet}\Fg^{-}$ over $\mathfrak{X}_1(p^m)(v)$ and $\mathfrak{I}\mathfrak{W}^{+}_w$ parametrizing $w$-compatible tuples $(Fil_{\bullet}\Fg^{-},\{v_i\})$. Moreover, there is another formal scheme $\mathfrak{I}\mathfrak{G}\mathfrak{B}_w$ parametrizing $w$-compatible basis $\{\delta_i\}$. And we have the following natural (forget) maps:

\[ \mathfrak{I}\mathfrak{G}\mathfrak{B}_w \rightarrow \mathfrak{I}\mathfrak{W}^{+}_w \rightarrow \mathfrak{I}\mathfrak{W}_w \rightarrow \mathfrak{X}_1(p^m)(v). \]

Now we turn to generic fibers. Let $\mathcal{I}\mathcal{W}^{+}_w$, $\iw_w$ and $\Igb_w$ denote their corresponding adic spaces over $Spa(\C_p,\Ol_{\C_p})$. We have the following natural map:

\[\pi^{AIP}: \iw^{+}_w \rightarrow \iw_w \rightarrow \X_1(p^m)(v) \rightarrow \X_1(p)(v)\rightarrow \X(p)(v). \] Moreover, we also have the following natural map \[\pi^{AIP,B}: \Igb_w \rightarrow \iw^{+}_w \xrightarrow{\pi^{AIP}} \X(p)(v). \]

We remark that in the literature (like \cite{shen2016} and \cite{AIP2015}) people only use $\pi^{AIP}$. The auxiliary space (\textit{generalized Igusa torsor}) $\Igb_w$ has several advantages and makes the comparison in section \ref{compare forms} more clearly. For example, we have the following straightforward  lemma:

\begin{lem}

(1) Under the natural map $\Igb_w \rightarrow \iw^+_w$, $\Igb_w$ is a \textbf{right} $\mathcal{N}_{M,w}$-torsor over $\iw^+_w$. 

(2) Under the map $\pi^{AIP,B}$, the space  $\Igb_w$ is a \textbf{right} $\Io^{(w)}_{M}$-torsor over $\X(p)(v)$.

\end{lem}

\begin{rem}
Although the space $\iw^+_w$ has a natural group action via $B_M(\Z_p)$,  it is \textbf{not} a group torsor over $\X(p)(v)$ when $n>1$. This is one advantage of $\Igb_w$.
\end{rem}

 Finally we can give the following definitions of overconvergent automorphic forms over $\X_{Iw}(v)= \X(p)(v)$:

\begin{defn}

Let $(R_{\U},\ka_{\U})$ be a $w$-analytic weight.

(1) The \textbf{sheaf of $w$-analytic $v$-overconvergent automorphic forms of weight $\ka_{\U}$ with Iwahori level}  is \[\underline{\omega}_{w,v}^{\ka_{\U},AIP}:=\pi_{*}^{AIP}\Ol_{\iw^+_{w,v}}[\ka_{\U}^{\vee}],\] which means the subsheaf of $\pi_*^{AIP}(\Ol_{\iw^+_{w,v}}\otimes R_{\U})$ where $B(\Z_p)$ acts through the character $\ka_{\U}^{\vee}$.

(2) The space of \textbf{$w$-analytic $v$-overconvergent automorphic forms of weight $\ka_{\U}$ with Iwahori level} is \[M^{\ka_{\U},AIP}_{Iw,w,v}:=H^0(\X_{Iw}(v), \waip).\]

(3) The space of \textbf{locally analytic overconvergent automorphic forms of weight $\ka_{\U}$ with Iwahori level} is
\[M^{\ka_{\U},AIP}_{Iw}:=\lim_{v \rightarrow0, w \rightarrow \infty} M^{\ka_{\U},AIP}_{Iw,w,v} .\]

\end{defn}

Use the formal models, we can also define the integral version (see section 3.2 of \cite{shen2016}). 

Finally we mention the following lemma relating the above definition with $\Igb_w$:

\begin{lem}
The sheaf $\underline{\omega}_{w,v}^{\ka_{\U},AIP}$   is also equivalent to the sheaf $\pi_{*}^{AIP,B}\Ol_{\Igb_w}[\ka_{\U}^{\vee}]$. 

\end{lem}

\begin{proof}
Just observation that $\Igb_w$ is a \textbf{right} $\mathcal{N}_{M,w}$-torsor over $\iw^+_w$, thus  $\Ol_{\iw^+_w}$ is exactly the $\mathcal{N}_{M,w}$-invariant of the pushforward of $\Ol_{\Igb_w}$. We're done.

\end{proof}

\subsection{Comparison of two kinds of locus }
\label{two locus}

In this section we show that $\{\X_{Iw}(v)\}$ and $\{\X_w\}$ are equivalent system of neighborhoods for canonical (multiplicative) ordinary locus inside $\X_{Iw}$.

Both locus measures the \textit{distance} to ordinarity. To further compare them, one useful auxiliary tool is \textit{pseudo-canonical subgroups}. See section 2.3 of \cite{chj} and section 3.6 of \cite{drw} for more details. 

Let $w \in \Q_{>0}$. Let $H$ be a one dimensional $p$-divisible group with height $n+1$ over $\Ol_{\C_p}$. There is a complex (not exact) \[0 \rightarrow Lie(H) \rightarrow T_p(H)\otimes \Ol_{\C_p} \xrightarrow{HT_H} \omega_{H^D}\rightarrow 0\] and denote the image of $HT_H$ by $Im_{H}$, which is a rank $n$ free module. We further introduce:

\begin{defn}

(1) A trivialization $\alpha: \Z_p^{n+1} \rightarrow T_p(H)$ is  \textbf{$w$-ordinary} if \[HT_H(\alpha(e_0)) \in p^w Im_{H}.\]

(2) We call $H$ is \textbf{$w$-ordinary} if it has a $w$-ordinary trivialization.

(3) If $H$ is $w$-ordinary, let $m$ be a positive integer such that $m<w+1$, then the kernel of \[H[p^m](\C_p)\rightarrow Im_{H}/p^{min(n,w)}Im_{H}\] further defines a finite flag group $C_{ps,m}\subset H[p^m]$ whose generic fiber is isomorphic to $\Z/p^m$. It is called \textbf{pseudo-canonical subgroup of level} $m$ for $H$. For $m=1$, we also call it \textbf{pseudo-canonical subgroup}.
\end{defn}

 If $\alpha$ is a $w$-ordinary trivialization for $H$, then $C_{ps,n}[\C_p]$ is generated by $\alpha(e_0)$. See lemma 3.6.5 of \cite{drw} for the proof. Use the same method  in lemma 3.6.7 of \cite{drw} (or lemma 2.11 in \cite{chj}), we deduce the following properties (analogue to canonical subgroup):
 
 $\bullet$ Let $m_1 \leq m$ be a positive integer and $w \in \Q_{>0}$ with $w > m$. Let $H$ be a $w$-ordinary one dimensional $p$-divisible group over $\Ol_{\C_p}$. Then $H/C_{ps,m_1}$ is $(w-m_1)$-ordinary, and for any positive integer $m_2$ with $m_1 <m_2 \leq m$, we have $C^{'}_{ps,m_2-m_1}=C_{ps,m_2}/C_{ps,m_1}$, where  $C^{'}_{ps,m_2-m_1}$ is the pseudo-canonical subgroup of $H/C_{ps,m_1}$ with level $m_2-m_1$.
 
 Use degree theory for finite flat group schemes and Hodge height (truncated valuation via Hasse lift $\widetilde{Ha}$) of $p$-divisible groups, we can show that  pseudo-canonical subgroup are exactly canonical subgroup:
 
 \begin{lem}
 Let $H$ be a $w$-ordinary one dimensional $p$-divisible group over $\Ol_{\C_p}$. Suppose $\frac{p}{2p-2}+m-1 <w \leq m$, then $C_{ps,m}$  coincides with the canonical subgroup of level $m$, $C_m$. 
 
 \end{lem}
 
 \begin{proof}

 We first deduce the case for $m=1$, relating pseudo-canonical subgroup and canonical subgroup. The method is the same as lemma 2.14 of \cite{chj} and lemma 3.6.9 of \cite{drw}. We can give a bound for the degree \[deg(C_{ps,1})\geq \frac{p-1}{p}w > 1- \frac{1}{2},\] then proposition 3.1.2 of \cite{AIP2015} shows that $C_{ps,1}$ is the canonical subgroup $C_1$ and the Hodge height $Hdg(H)< \frac{1}{2}$.

  Then use the above properties of pseudo-canonical subgroup we  can  do induction for general $m$. See proposition 3.6.11 of \cite{drw} for more details. 
 
 \end{proof}
 
 A simple corollary is the  following (half) comparison:
 
 \begin{prop}
 Let $m$ be a positive integer and $w \in \Q_{>0}$ such that $\frac{p}{2p-2}<w \leq n$. Then there exists $v \in \Q \cap [0,\frac{1}{2p^{n-1}})$ and the natural inclusion $\X_{w} \subset \X_{Iw}(v)$.
 \end{prop}

For such $w$, we also have the following nice moduli interpretations: $\X_{Iw,w}$ parametrizes $w$-ordinary $p$-divisible groups $H$ together  with a filtration $Fil_{\bullet}H[p]$ such that $Fil_{1}H[p]=C_1$ (canonical subgroup). $\X_{\infty,w}$ further records a trivialization for $T_p(H)$ that is compatible with $Fil_{\bullet}H[p]$.
 
If $(H,\alpha)$ corresponds to such a point in $\X_{\infty,w}$ and $\alpha$ is a $w$-ordinary trivialization. Then $\alpha(e_0)$ will produce the level $m$ canonical subgroup $C_m$. Suppose $(H,\alpha^{'})$ is another point which maps to the same point $(H,Fil_{\bullet}H[p])$ in $\X_{Iw}$, then $\alpha^{'}$ differs from $\alpha$ via an element in $Iw$ and we still have the isomorphism: \[(\Z_p/p^m)^{n}\hookrightarrow^{\alpha^{'}} H[p^m](\C_p)   \twoheadrightarrow H[p^m](\C_p)/C_m(\C_p) \cong C_m^{\perp,D}(\C_p).\] In other words, each point $(H,\alpha^{'})$ inside $\X_{\infty,w}$ also records a frame (trivialization) for $C_m^{\perp,D}$. This observation will be  used in next section \ref{compare forms}.

We only remains to show another direction of the comparison between two locus.

For $w \in \Q_{>1}$, suppose $w \in (n-1,n]$ for large enough integer $n$. Pick $v \in \Q_{>0} \cap [0, \frac{1}{2p^{n-1}}) $ such that $w \in (n-1+\frac{v}{p-1},n-v \frac{p^n-1}{p-1}]$, use the proposition 3.2.1 and 3.2.2 of \cite{AIP2015} (at level of $Spa(\C_p,\Ol_{\C_p})$-points), we deduce that $\X_{Iw}(v) \subset \X_{Iw,w}$.

In conclusion, we have established the following equivalence:

\begin{thm}

Two kinds of locus are equivalent:

(1) For $w \in \Q_{>0}$, there exists a $v \in \Q \cap (0,\frac{1}{2})$ such that $\X_{Iw}(v)  \subset \X_{Iw,w}$;

(2) For  $v \in \Q \cap (0,\frac{1}{2})$, there exists a $w \in \Q_{>0}$ such that $\X_{Iw,w}  \subset \X_{Iw}(v)$.

\end{thm}

\subsection{Comparison of two kinds of sheaves}
\label{compare forms}
In this section we will deduce the equivalence between two constructions for overconvergent  $p$-adic automorphic forms.

Regard the result in previous section,  For $w \in \Q_{>1}$, suppose $n$ is a positive integer large enough such that we can  pick up $v \in \Q_{>0} \cap [0, \frac{1}{2p^{n-1}}) $ with $w \in (n-1+\frac{v}{p-1},n-v \frac{p^n-1}{p-1}]$. Then   $\X_{Iw}(v) \subset \X_{Iw,w}$.

Now we can prove the first main result of this paper:

\begin{thm}
Over $\X_{Iw}(v)$, there is a canonical isomorphism \[\Psi: \waip \cong \underline{\omega}_{w}^{\ka_{\U}}. \] In particular, perfectoid automorphic forms is equivalent to locally analytic overconvergent automorphic forms.
\end{thm}

\begin{proof}

We use the same method as in proposition \ref{classical form}.
 
 Through the natural map $\X_{\infty,w} \rightarrow \X_{Iw,w}$, let $\X_{\infty}(v)$ denote the pullback of $\X_{Iw}(v)$.

Consider the pullback diagram
\[
\xymatrix{
\Igb_{\infty,w} \ar[r] \ar[d]_{\pi_{\infty}^{AIP,B}} & \Igb_w \ar[d]_{\pi^{AIP,B}} \\
\X_{\infty}(v) \ar[r] & \X_{Iw}(v)
}
\]

Recall that we have constructed sections $\s=(\s_1,\cdots,\s_n)$  for the $n$-dimensional vector bundle $\underline{\omega}_{\infty}$. We first claim that $\s$ will further provide a section for $\Igb_{\infty,w}$ over $\X_{\infty}(v)$.

It is sufficient to verify this at level of $Spa(\C_p,\Ol_{\C_p})$-points. For such a point $x$, let $H$ denote the corresponding $p$-divisible group together with the trivialization \[\eta_p: \Z_p^{n+1} \cong T_p(H).\] From the discussion of $\X_{Iw,w}$ in previous section \ref{two locus}, we have seen that   $\eta_p$ further induces an isomorphism (frame) \[(\Z/p^m)^{n}\cong (H[p^m]/C_m)(\C_p)\cong C_{m}^{\perp,D}(\C_p).\] Then $\{HT(\eta_p(e_1)),\cdots,HT(\eta_p(e_n))\}$ is a $w$-admissible basis for the sheaf $\Fg^{-}$ inside $\omega_{H^D}$. In particular, the section $\s=(\s_1,\cdots,\s_n)$ for $\underline{\omega}_{\infty}$ provides a trivialization of $\Igb_{w,\infty}$ as \textbf{right} $\Io^{(w)}_M$-torsor:

\[
\xymatrix{
\X_{\infty,w}\times \Io^{(w)}_M \ar[rr]^{\delta}_{\cong} \ar[rd] & & \Igb_{w,\infty} \ar[ld] \\
& \X_{\infty,w} &
}.
\]

More explicitly, for any affinoid open $\V \subset \X_{Iw,w}$ with preimage $\V_{\infty}$ on $\X_{\infty,w}$, the section $\waip(\V)$ will be the $Iw$-invariant of the following induction \[ \mathscr{C}_{\ka_{\U}}^{w-an}(Iw^{(w)}_M, \Ol(\V_{\infty}) \widehat{\otimes} R_{\U}).\]

As the isomorphism $\delta$ is over infinite level, the action of $Iw$ on $\X_{\infty,w}\times \Io^{(w)}_M$ is \textbf{twisted}. Recall the transform rule (lemma \ref{transform law}) \[\gamma^*(\s)=\s (\z \gamma_b+\gamma_d),\] for any $\gamma=\begin{pmatrix}\gamma_a & \gamma_b \\ \gamma_c & \gamma_d  \end{pmatrix}\in Iw.$ Therefore the $Iw$-action is   given by \[\gamma \diamond f = \rho_{\ka_{\U}}(\z\gamma_b+\gamma_d) \gamma^*f,\] which is exactly the same twisted action in defining perfectoid automorphic forms.

We're done.

\end{proof}

\begin{rem}

The same proof works for comparison over integral versions. Moreover, two constructions of Hecke operators are also equivalent (ignore the normalization factor at $p$).

\end{rem}

\begin{rem}
Our proof works for more general unitary Shimura varieties with signature $(a,b)$ (although more complicated). I will write more details in later paper of this series.
\end{rem}

Through pullback via $\X_{Iw^+,w}\rightarrow \X_{Iw,w}$, we can get overconvergent automorphic form at strict Iwahori level and the same argument shows that it is just perfectoid automorphic form at strict Iwahori level.  Moreover, the lemma 3.3.10 in \cite{drw} also works in our setting so that we can apply the generalized projection formula (see appendix A) in \cite{drw}.

\section{Overconvergent cohomology}
\label{section4}

In this section, we shift our focus to the cohomological side and construct the overconvergent cohomology, which serves as the $p$-adic interpolation of the classical etale cohomology of the Shimura variety. This construction, in turn, provides an alternative realization of $p$-adic automorphic forms. Our approach follows the general framework established in the literature (cf. \cite{hansen2017universal} and \cite{jn2019eigenvariety}). Notably, the construction on this side necessitates the use of the \textbf{opposite} Borel subgroup to  further construct $p$-adic Eichler-Shimura comparison map in later section. 

\subsection{Analytic distributions}
\label{distribution}

Let's first recall some standard constructions in  \cite{hansen2017universal} and \cite{jn2019eigenvariety}.

The starting point is also analytic induction. Previously we did such thing about Levi subgroup $M=GL(n)$. Now we apply such method to $G$ itself and use \textbf{opposite} Borel subgroup (lower triangular matrices). The reason of this opposite choice is to construct the  overconvergent Eichler-Shimura map later (e.g. see lemma \ref{highest weight construction}).

Recall that we will consider "reduced" weight for $T$. Let $\ka_{\U}$ denote a weight for $T_{M}(\Z_p)$ with \[\ka_{\U}=(\ka_{\U,n},\cdots,\ka_{U,1}).\] Let the $n\times n$-matrix $w_{0}=\begin{pmatrix}& & 1 \\ & 1 \\ & \cdots \\ 1 \end{pmatrix}$ be a lift of the longest element of the Weyl group for $M$. It is an involution, i.e. $w_0^2=\mathbb{I}_n$.  We define \[\widetilde{w_0}(\ka_{\U}):=(0,\ka_{\U,n},\cdots,\ka_{\U,1}),\] which trivially extend $(\ka_{\U,n},\cdots,\ka_{U,1})=w_0(\ka_{\U})$. Consider the opposite Borel subgroup $B^{opp}$ (lower triangular matrices) for $G=GL(n+1)$ and we  further  trivially extend $\widetilde{w_0}(\ka_{\U})$ to a character for $B^{opp}(\Z_p)$. Let $Iw^{opp}$ denote the Iwahori subgroup corresponding to $B^{opp}$.

\begin{rem}

 For a weight $\ka_{\U}$, the define for weight $\ka_{\U}$ analytic induction for $Iw^{opp}$ will involve $\widetilde{w_0}(\ka_{\U})$ while the definition for such analytic induction via $\tt$ will use $w_0(\ka_{\U})$. Such twist by $w_0$ is due to the convention of using \textbf{opposite} Borel subgroup.    Such conventions  are indeed compatible the usual convention in the literature (via Borel subgroup $B$) There is also a hidden "\textit{switch}" process.  For example, in the case of  classical (algebraic irreducible) representations, see section \ref{classical es} for more details.

\end{rem}

 Let $r$ be a positive integer and $(R_{\U}, \kappa_{\U})$ be an $r$-analytic weight. Fix an isomorphism
\[N(p\Z_p)\cong \Z_p ^{\frac{n(n+1)}{2}},\] \[[a_{i,j}] \mapsto (\frac{a_{i,j}}{p})_{i<j}.\]

\begin{defn}
A function $f:N(p\Z_p)\rightarrow R_{\U}^{+}$ is $r$-analytic if the composition \[ \Z_p ^{\frac{n(n+1)}{2}} \cong N(p\Z_p) \xrightarrow{f} R_{\U}^+ \rightarrow \C_p \widehat{\otimes} R_{\U}\] is $r$-analytic.
\end{defn}

Let $A^{r,\circ}(N(p\Z_p),R_{\U})$ denote the set of such $r$-analytic functions on $N(p\Z_p)$ and define \[A^{r}(N(p\Z_p),R_{\U})=A^{r,\circ}(N(p\Z_p),R_{\U})[\frac{1}{p}].\]

Now we define the following $r$-analytic induction:

\[
A_{\kappa_{\mathcal{U}}}^{r,\circ}(Iw^{opp}, R_{\mathcal{U}}) = \left\{
    f \colon Iw^{opp} \to R_{\mathcal{U}}^{+}
    \ \middle| \
    \begin{aligned}
        & f(\gamma b) = \widetilde{w_0}(\kappa_{\mathcal{U}})(b) f(\gamma), \quad \forall \gamma \in Iw^{opp}, b \in B^{opp}(\Z_p), \\
        & f|_{N(p\Z_p)} \text{is $r$-analytic}.
    \end{aligned}
\right\}
\] and \[A_{\kappa_{\U}}^{r}(Iw^{opp},R_{\U})=A_{\kappa_{\U}}^{r,\circ}(Iw^{opp},R_{\U})[\frac{1}{p}].\]

Through the restriction we get a natural identification \[A_{\kappa_{\mathcal{U}}}^{r,\circ}(Iw^{opp}, R_{\mathcal{U}}) \cong A^{r,\circ}(N(p\Z_p),R_{\U}), f\mapsto f|_{N(p\Z_p)}. \]

Taking continuous dual, we get the corresponding spaces of $r$-analytic distributions \[ D_{\kappa_{\mathcal{U}}}^{r,\circ}(Iw^{opp}, R_{\mathcal{U}}) =Hom_{R_{\U}^{+}}^{cts}(A_{\kappa_{\mathcal{U}}}^{r,\circ}(Iw^{opp}, R_{\mathcal{U}}), R_{\U}^{+})\] and \[D_{\kappa_{\mathcal{U}}}^{r,\circ}(Iw^{opp}, R_{\mathcal{U}})=D_{\kappa_{\mathcal{U}}}^{r,\circ}(Iw^{opp}, R_{\mathcal{U}})[\frac{1}{p}].\]

As our goal is to establish $p$-adic Eichler-Shimura theory, we  provide another construction of such distribution involving the Levi subgroup $M=GL(n)$.

Let $B_{M}^{opp}$ denote the \textbf{opposite} Borel subgroup (lower triangular matrices) for $M$ and $Iw_M^{opp}$ denote the corresponding Iwahori subgroup. Similarly we define $N_M^{opp}$.

Consider the following set \[\tt=\{(t_b,t_d)\in M_{1,n}(\Z_p)\times Iw^{opp}_{M}\}\] and its subset \[\tt_0=\{(t_b,t_d)\in M_{1,n}(\Z_p)\times N_M(pZ_p)\}. \]

This $\tt$ ("\textit{twisted Iwahori}") will play the role of $Iw^{opp}$.

Define the following subgroup $\widetilde{N^{opp}}$ of $Iw^{opp}$: \[\widetilde{N^{opp}}=\{\begin{pmatrix} t_a & 0 \\ t_c & \mathbb{I}_n \end{pmatrix}|t_a \in \Z_p^*, t_c \in M_n(\Z_p)\}.\]

The natural projection $Iw^{opp}\longrightarrow \tt$, $t=\begin{pmatrix} t_a & t_b \\ t_c & t_d \end{pmatrix} \mapsto (t_b,t_d)$ induces an isomorphism \[Iw^{opp}/\widetilde{N^{opp}} \cong \tt.\]

Then we can lift any element in $\tt$ to $Iw^{opp}$. In particular the automorphic factor \[M_{n,1}(\mathscr{O}_{\C_p})\times \tt \longrightarrow M_n(\C_p)\] \[J(\overrightarrow{z},t)=\overrightarrow{z}t_b+t_d\] is well defined and is compatible with previous construction (see section \ref{flag variety}). In other words, for an element $t \in \tt$, we can pick up any lift $\widetilde{t}$ and get $J(\overrightarrow{z},t)=J(\overrightarrow{z},\widetilde{t})$. Moreover the right 1-cocyle property for this automorphic factor $J$ still holds.

Through the above identification we equip $\tt$ with two natural actions:

$(i)$ There is a \textbf{left} action by $Iw^{opp}$ through left multiplication: \[Iw^{opp}\times \tt \longrightarrow \tt\]  \[\begin{pmatrix} \gamma_a & \gamma_b \\ \gamma_c & \gamma_d \end{pmatrix}, (t_c,t_d)\mapsto  (\gamma_a t_c+\gamma_b t_d, \gamma_c t_c+\gamma_d t_d).\]

In particular, as $Iw^{+}$ is a subgroup of $Iw^{opp}$, it also acts on $\tt$ in this way.

$(ii)$ The group $\widetilde{N^{opp}}$ is normal in $B^{opp}(\Z_p)$ with $\widetilde{N^{opp}}\backslash B^{opp}(\Z_p)\cong B_M^{opp}(\Z_p)$. In particular there is a \textbf{right} action by $B_M^{opp}(\Z_p)$ via right multiplication: \[\tt \times B_M^{opp}(\Z_p) \longrightarrow \tt  \] \[(t_b, t_d), \gamma \mapsto (t_b \gamma, t_d \gamma).\]

 The natural inclusion $\tt_0 \hookrightarrow \tt$ induces an isomorphism $\tt_0 \cong \tt/B_M^{opp}(\Z_p)$, combine with $N(p\Z_p) \cong Iw^{opp}/B^{opp}(\Z_p)$ we get the natural identification $\tt_0 \cong N(p\Z_p)$.

Similar to the standard construction via $Iw^{opp}$, we can also define certain analytic inductions and distributions via $\tt$.

\begin{defn}

A function $f: \tt_0 \rightarrow R_{\U}^{+}$ is $r$-analytic is it is $r$-analytic as a function on $N(p\Z_p)$ through the identification $\tt_0 \cong N(p\Z_p)$.

\end{defn}

Similarly denote the set of such $r$-analytic functions by $A^{r,\circ}(\tt_0,R_{\U})$ and we define \[A^{r}(\tt_0,R_{\U})=A^{r,\circ}(\tt_0,R_{\U})[\frac{1}{p}].\]

Subsequently we define the $r$-analytic induction:

\[
A_{\kappa_{\mathcal{U}}}^{r,\circ}(\tt, R_{\mathcal{U}}) = \left\{
    f \colon \tt \to R_{\mathcal{U}}^{+}
    \ \middle| \
    \begin{aligned}
        & f(\gamma b) = w_0(\kappa_{\mathcal{U}})(b) f(\gamma), \quad \forall \gamma \in \tt, b \in B_{M}^{opp}(\Z_p), \\
        & f|_{\tt_0} \text{is $r$-analytic}.
    \end{aligned}
\right\}
\] and \[A_{\kappa_{\U}}^{r}(\tt,R_{\U})=A_{\kappa_{\U}}^{r,\circ}(\tt,R_{\U})[\frac{1}{p}].\]

Taking continuous dual, we get the $r$-analytic distributions: \[ D_{\kappa_{\mathcal{U}}}^{r,\circ}(\tt, R_{\mathcal{U}}) =Hom_{R_{\U}^{+}}^{cts}(A_{\kappa_{\mathcal{U}}}^{r,\circ}(\tt, R_{\mathcal{U}}), R_{\U}^{+})\] and \[D_{\kappa_{\mathcal{U}}}^{r}(\tt, R_{\mathcal{U}})=D_{\kappa_{\mathcal{U}}}^{r,\circ}(\tt, R_{\mathcal{U}})[\frac{1}{p}].\]

We have the following natural identifications:

\[A_{\kappa_{\mathcal{U}}}^{r,\circ}(\tt, R_{\mathcal{U}}) \cong A_{\kappa_{\mathcal{U}}}^{r,\circ}(Iw^{opp}, R_{\mathcal{U}})   \]  and \[D_{\kappa_{\mathcal{U}}}^{r,\circ}(\tt, R_{\mathcal{U}}) \cong D_{\kappa_{\mathcal{U}}}^{r,\circ}(Iw^{opp}, R_{\mathcal{U}})   .\] The other constructions are also identified. In conclusion, through $\tt$ we give an \textbf{equivalent} description of analytic functions and distributions on $Iw^{opp}$.

The previous \textbf{right} action of $B_{M}^{opp}(\Z_p)$ and \textbf{left} action of $Iw^{opp}$ on $\tt$ then correspondingly induces a right action of of $B_{M}^{opp}(\Z_p)$ and \textbf{left} action of $Iw^{opp}$ on $D_{\kappa_{\mathcal{U}}}^{r,\circ}(\tt, R_{\mathcal{U}})$.

For any $r_1 \geq r_2$, the natural injection $A_{\kappa_{\mathcal{U}}}^{r_1,\circ}(\tt, R_{\mathcal{U}}) \hookrightarrow A_{\kappa_{\mathcal{U}}}^{r_2,\circ}(\tt, R_{\mathcal{U}})  $ induces a natural \textit{injection} (see \cite{hansen2017universal} section 2.2) $D_{\kappa_{\mathcal{U}}}^{r_2,\circ}(\tt, R_{\mathcal{U}}) \hookrightarrow D_{\kappa_{\mathcal{U}}}^{r_1,\circ}(\tt, R_{\mathcal{U}})  $ , and we denote \[D^{\dag}_{\kappa_{\U}}(\tt,R_{\mathcal{U}})=\varprojlim_{r} D_{\kappa_{\mathcal{U}}}^{r}(\tt, R_{\mathcal{U}}).\]

Suppose $(R_{\U},\kappa_{\U})$ is a small weight and take large enough $r>1+r_{\U}$, following the proposition 3.1 of \cite{chj} and section 4.1 of \cite{drw}, $D_{\kappa_{\mathcal{U}}}^{r_2,\circ}(\tt, R_{\mathcal{U}})$ has a decreasing filtrations $Fil^{j}$, which are also stable under two actions.

Define \[D_{\kappa_{\mathcal{U}},j}^{r,\circ}(\tt, R_{\mathcal{U}})=D_{\kappa_{\mathcal{U}}}^{r,\circ}(\tt, R_{\mathcal{U}})/Fil^{j}\] and we have \[D_{\kappa_{\mathcal{U}}}^{r,\circ}(\tt, R_{\mathcal{U}})= \varprojlim_{j} D_{\kappa_{\mathcal{U}},j}^{r,\circ}(\tt, R_{\mathcal{U}}).\] In particular $D_{\kappa_{\mathcal{U}}}^{r,\circ}(\tt, R_{\mathcal{U}})$ is a profinite flat $\Z_p$-module (see \cite{chj} definition 6.1).

\subsection{Overconvergent cohomology}

For simplicity, we will \textbf{restrict} to $Iw^{+}$ from now on.

The construction of this section works directly at level of $Iw^{opp}$. But soon we will construct the $p$-adic Eichler-Shimura map, which works naturally at level $Iw^{+}$ (as $Iw^{+}$ is the intersection $Iw \cap Iw^{opp}$).

Fix a small weight $(R{\U},\kappa_{\U})$ and consider the etale site (over adic space) $\mathcal{X}_{Iw^{+},et}$. The perfectoid Shimura variety $\mathcal{X}_{\infty}$ is a pro-etale Galois cover of $\mathcal{X}_{Iw^{+}}$ with Galois group $Iw^{+}$. For each positive integer $j$, let $\mathscr{D}_{\kappa_{\U},j}^{r,\circ}$ denote the locally constant sheaf on $\mathcal{X}_{Iw^{+}}$ corresponding to \[\pi_{1}^{et}(\mathcal{X}_{Iw^+})\rightarrow Iw^{+}\rightarrow Aut(D_{\kappa_{\mathcal{U}},j}^{r,\circ}(\tt, R_{\mathcal{U}})). \] We get an inverse system of locally constant sheaves $(\mathscr{D}_{\kappa_{\U},j}^{r,\circ})$. Thus we can consider the resulting etale cohomology groups \[H^*_{et}(\X_{Iw^+},\mathscr{D}_{\kappa_{\U}}^{r,\circ}):=\varprojlim_{j}H^*_{et}(\X_{Iw^+},\mathscr{D}_{\kappa_{\U},j}^{r,\circ})\] \[H^*_{et}(\X_{Iw^+},\mathscr{D}_{\kappa_{\U}}^{r}):=H^*_{et}(\X,\mathscr{D}_{\kappa_{\U}}^{r,\circ})[\frac{1}{p}].\]

On the other hand, we can also consider the Betti cohomology of the Shimura variety  over  $\C$ (as manifold). For the Adelic level subgroup $K=K^{p} \times Iw^{+}$, let $K^{p}$ act on $D_{\kappa_{\mathcal{U}},j}^{r,\circ}(\tt, R_{\mathcal{U}})$ trivially and $Iw^{+}$ acts on it through the previous left multiplication. Then $D_{\kappa_{\mathcal{U}},j}^{r,\circ}(\tt, R_{\mathcal{U}})$ corresponds to a local system on $X_{Iw^+}(\C)$ and we can consider the resulting Betti cohomology \[H^*(X_{Iw^+}(\C), D_{\kappa_{\mathcal{U}},j}^{r,\circ}(\tt, R_{\mathcal{U}})).\]

Moreover we have the following  important comparison proposition.

\begin{prop}

There are natural isomorphisms \[H^*_{et}(\X_{Iw^+},\mathscr{D}_{\kappa_{\U}}^{r}) \cong H^*(X_{Iw^+}(\C), D_{\kappa_{\mathcal{U}},j}^{r,\circ}(\tt, R_{\mathcal{U}})).\]

\end{prop}

See \cite{drw} proposition 4.2.2 for the proof. Although they work with Siegel cases, their argument directly works for general Shimura varieties.

Now we define Hecke actions on the overconvergent cohomology $H^*_{et}(\X_{Iw^+},\mathscr{D}_{\kappa_{\U}}^{r})$. Indeed this follows from the general construction in \cite{hansen2017universal}. See that paper for more details.

\textbf{Hecke operators outside} $p$. Recall the tame level subgroup $K^{p}=K^{S}\times K_{S}$ ($S$ is a finite set of "\textit{bad primes}"), for any prime $l \neq p$ outside $S$, for any $\gamma \in G(\Q_l)$, consider the decomposition of the corresponding double coset \[K_l \gamma K_l=\coprod \delta_{j} \gamma K_l.\] Let $G(\Q_l)$ trivially act on $D_{\kappa_{\mathcal{U}},j}^{r,\circ}(\tt, R_{\mathcal{U}})$, combine with the natural action of $G(\Q_l)$ on Shimura variety, we get the Hecke operator \[T_{\gamma}: H^*_{et}(\X_{Iw^+},\mathscr{D}_{\kappa_{\U}}^{r})\rightarrow H^*_{et}(\X_{Iw^+},\mathscr{D}_{\kappa_{\U}}^{r}) \]\[[\mu] \mapsto \sum_{j}\delta_j\gamma.([\mu]).\]

\textbf{Hecke operators at} $p$.  For any $0\leq i \leq n-1$, recall the matrices $u_{p,i}=\begin{pmatrix} u_{p,i,+} & \\  & u_{p,i,-} \end{pmatrix}$ (in section \ref{hecke operator}), where $u_{p,i,+}=p$ and $u_{p,i,-}=diag(p,..,p,1,..,1)$ ($i$-terms $p$). In particular, $u_{p,0,-}=\mathbb{I}_{n}$.

Consider a $u_{p,i}$ action on $\tt$ as follow: For any $(t_b,t_d) \in \tt$, suppose $(t_b,t_d)=(t_{b,0},t_{d,0})\beta$, where $t_{d,0} \in N_M(p\Z_p)$ and $\beta \in B_{M}^{opp}(\Z_p)$, then put \[u_{p,i}.(t_b,t_d)=(u_{p,i,+}t_b u_{p,i,-}^{-1}, u_{p,i,-}t_du_{p,i,-}^{-1})\beta.\] This further induces $u_{p,i}$-action on $D_{\kappa_{\mathcal{U}}}^{r}(\tt, R_{\mathcal{U}})$.

Similar to section \ref{hecke operator}, we define the Hecke operators at $p$ as follow:

Take a double coset decomposition \[Iw^{+}u_{p,i}Iw^{+}=\coprod_{j}\delta_{i,j}u_{p,i}Iw^{+},\] combine the natural action of $G(\Q_p)$ on the Shimura variety with the actions of $Iw^{+}$ and $u_{p,i}$ on $D_{\kappa_{\mathcal{U}}}^{r}(\tt, R_{\mathcal{U}})$, we get the Hecke operator \[U_{p,i}: H^*_{et}(\X_{Iw^+},\mathscr{D}_{\kappa_{\U}}^{r})\rightarrow H^*_{et}(\X_{Iw^+},\mathscr{D}_{\kappa_{\U}}^{r}) \]\[[\mu] \mapsto \sum_{j}\delta_{i,j}.(u_{p,i}.([\mu])).\]

\begin{rem}
\label{up action on symbol}
 Use the samw trick  in proposition \ref{define up} again, we can get a better viewpoint about this action.

Similarly set \[\widetilde{Iw^{opp}}=Iw^{opp}B^{opp}(\Q_p)=N(p\Z_p)B^{opp}(\Q_p).\] Then this set is stable under left multiplication by $u_{p,i}$. And for any weight $\ka$ of $T(\Z_p)$, similarly get (extension) the weight for  $T(\Q_p)$, $\widetilde{\ka}$. Moreover, two inductions are canonically identified \[Ind_{B^{opp}(\Z_p)}^{Iw^{opp}}(\ka,-)=Ind_{B^{opp}(\Q_p)}^{\widetilde{Iw^{opp}}}(\ka,-).\]

Then we can give equivalent construction of analytic functions and distributions on $Iw^{opp}$ via $\widetilde{Iw^{opp}}$. In terms of $Iw^{opp}$, the $u_{p,i}$-action is the same as \textbf{left multiplication}.  In particular,  it acts in the \textbf{ same}  manner as $Iw^{+}$. Then the Hecke operator $U_{p,i}$ is also \textit{well defined}.

\end{rem}

Finally we define the total operator $U_p$:

\begin{defn}
Define the operator $U_p=\prod_{i}U_{p,i}$.

\end{defn}

\section{The overconvergent Eichler-Shimura map}
\label{section5}

In this section, we will establish a kind of overconvergent Eichler-Shimura map relating these two construction of $p$-adic automorphic forms.

\subsection{The pro-etale cohomology groups}

As we mentioned in the introduction, one key ingredient is to use pro-etale site established in the seminal work \cite{scholzepro}. First we will construct a "\textit{big}"  sheaf $\mathscr{OD}^{r}_{\ka_{\U}}$ on the pro-etale site $\X_{Iw^+,proet}$ which computes the overconvergent cohomology in the previous section.

Let \[v:\X_{Iw^+,proet}\longrightarrow \X_{Iw^+,et}\] be the natural projection. Consider the following completed pullback: \[\mathscr{OD}_{\ka_{\U}}^{r}:=(\varprojlim_{j}v^{-1}\mathscr{D}_{\ka_{\U},j}^{r,\circ}\otimes \mathscr{O}_{\X_{Iw^+,proet}}^{+})[\frac{1}{p}].\]

For simplicity, we introduce the following notations:

\begin{defn}
For a small weight $(R_{\U},\ka_{\U})$ and $r \geq r_{\U}+1$, we set \[\mathfrak{OC}^{r,\circ}_{\ka_{\U}}:= \varprojlim_{j}H^{n}_{et}(\X_{Iw^+},\mathscr{D}_{\kappa_{\U},j}^{r,\circ}),\] \[\mathfrak{OC}^{r}_{\ka_{\U}}:=\mathfrak{OC}^{r,\circ}_{\ka_{\U}}[\frac{1}{p}]=H^{n}_{et}(\X_{Iw^+},\mathscr{D}_{\kappa_{\U}}^{r}),\] \[\mathfrak{OC}^{r,\circ}_{\ka_{\U},\mathcal{O}_{\C_p}}:= \varprojlim_{j}(H^{n}_{et}(\X_{Iw^+},\mathscr{D}_{\kappa_{\U},j}^{r,\circ})\otimes \mathcal{O}_{\C_p}),\] \[\mathfrak{OC}^{r}_{\ka_{\U},\C_p}:=\mathfrak{OC}^{r}_{\ka_{\U},\mathcal{O}_{\C_p}}[\frac{1}{p}].\]
\end{defn}

The following comparison proposition shows that the "\textit{big}" sheaf exactly computes such overconvergent cohomology.

\begin{prop}
There is a natural isomorphism: \[H^n_{proet}(\X_{Iw^+}, \mathscr{OD}_{\ka_{\U}}^{r}) \cong \mathfrak{OC}^{r}_{\ka_{\U},\C_p}\]
\end{prop}

We refer to proposition 5.1.2 in \cite{drw} for the proof.

Previously we have seen that there is  a Hecke action on the right side, therefore the left side also has such Hecke actions. For tame (prime to $p$) Hecke operators corresponding to double coset $[K\gamma K]$, such action can also be given via the following Hecke correspondence (the same as perfectoid forms case): \[\xymatrix{
  & \X_{Iw^+\cap \gamma Iw^+ \gamma^{-1}} \ar[dl]_{pr_1} \ar[dr]^{pr_2} \\
   \X_{Iw^+} & & \X_{Iw^+}
}.\]

\subsection{The overconvergent Eichler-Shimura map}
\label{p-adic eichler shimura}

In this section we will first construct a suitable comparison map between two "big" sheaves over the pro-etale site $\X_{Iw^+,w,proet}$, then taking $Iw^{+}$-invariant will produce the desired overconvergent Eichler-Shimura map.

Let $(R_{\U},\ka_{\U})$ be a small weight and $r \geq r_{\U}+1$. Recall that the perfectoid automorphic forms is indeed a kind of  $Iw^{+}$-invariant (under twisted action, see remark \ref{twist action invariant}). Similarly we have such analogue for overconvergent cohomology sheaf:

\begin{lem}
\label{OD invariant}
Let $\V \rightarrow \X_{Iw^+}$ be an affinoid perfectoid pro-etale over $X_{Iw^+}$ and set $\V_{\infty}:=\V\times_{\X_{Iw^+}}\X_{\infty}$. Then we have the following natural identification \[\mathscr{OD}_{\kappa_{\U}}^{r}(\V)\cong (D_{\kappa_{\mathcal{U}}}^{r,\circ}(\tt, R_{\mathcal{U}})\widehat{\otimes} \widehat{\mathscr{O}}_{\X_{Iw^+,proet}}(\V_{\infty}))^{Iw^+}.\]
\end{lem}

We refer to \cite{drw} lemma 5.2.1 for the proof.

Now consider the completed pullback to the pro-etale site for overconvergent automorphic sheaf: \[\widehat{\underline{\omega}}_{w}^{\ka_{\U},+}:=\varprojlim_{j}(\underline{\omega}_{w,et}^{\ka_{\U},+}\bigotimes_{\mathscr{O}^{+}_{\X_{Iw^+,w,et}}}\mathscr{O}^{+}_{\X_{Iw^+,w,proet}}/p^{j})\] and \[\widehat{\underline{\omega}}_{w}^{\ka_{\U}}:=\widehat{\underline{\omega}}_{w}^{\ka_{\U},+}[\frac{1}{p}].\]

\begin{prop}
\label{proposition proetale etale}
There is a canonical $Gal_{\Q_p}$-equivariant map
\[H^{n}_{proet}(\X_{Iw^+,w},\widehat{\underline{\omega}}_{w}^{\ka_{\U}})\rightarrow H^0(\X_{Iw^+,w}, \underline{\omega}_{w}^{\ka_{\U}+n+1})(-n).\]

\end{prop}

\begin{proof}
Consider the natural projection $v:\X_{Iw^+,w,proet}\rightarrow \X_{Iw^+,w,et}$, apply the corollary A.3.14 (a kind of projection formula) in \cite{drw}, we get a canonical isomorphism:

 \[\underline{\omega}_{w,et}^{\ka_{\U}} \otimes_{\mathscr{O}_{\X_{Iw^+,w,et}}} R^{i}v_* \widehat{\mathscr{O}}_{\X_{Iw^+,w,proet}} \cong R^{i}v_{*} \widehat{\underline{\omega}}_{w}^{\ka_{\U}} .\]

 On the other hand, we have the following standard computation in $p$-adic Hodge theory (e.g. see remark 6.20 in \cite{scholzepro} or  proposition A.2.3 in \cite{drw}) \[R^{i}v_* \widehat{\mathscr{O}}_{\X_{Iw^+,w,proet}} \cong \Omega^{i}_{\X_{Iw^+,w,et}}(-i).\]

 In particular, if $i=n$, we get  \[R^{n}v_* \widehat{\mathscr{O}}_{\X_{Iw^+,w,proet}} \cong \Omega^{n}_{\X_{Iw^+,w,et}}(-n).\]

According to the Kodaira-Spencer isomorphism for $\Omega^{1}_{\X_{Iw^+}}$ and take its determinant bundle $\Omega^{n}_{\X_{Iw^+}}$, we get the following isomorphism (after a suitable central weight shift): \[\Omega^{n}_{\X_{Iw^+}} \cong \underline{\omega}_{{Iw^+}}^{n+1} .\]

  Recall that over $\X_{Iw^+,w}$, there is a natural inclusion from the sheaf of  classical automorphic forms into the sheaf of overconvergent automorphic forms (proposition \ref{classical form}): \[\underline{\omega}_{Iw^+}^{n+1} \hookrightarrow \underline{\omega}_{w}^{n+1}.\]

Combine them together and apply the spectral sequence relating etale cohomology and pro-etale cohomology, we obtain the desired natural map: \[H^{n}_{proet}(\X_{Iw^+,w},\widehat{\underline{\omega}}_{w}^{\ka_{\U}})\rightarrow H^0(\X_{Iw^+,w}, \underline{\omega}_{w}^{\ka_{\U}+n+1})(-n).\]

\end{proof}

Now we are ready to construct the  comparison map. Consider the following function $e_{hwt,\ka_{\U}}$ on $Iw_M^{(w)}$:\[e_{hwt,\ka_{\U}}: X=(X_{i,j})\mapsto \frac{\ka_{\U,1}(X_{1,1})}{\ka_{\U,2}(X_{1,1})} \times \frac{\ka_{\U,2}(\det(X_{i,j})_{1 \leq i,j \leq 2})}{\ka_{\U,3}(\det(X_{i,j})_{1 \leq i,j \leq 2})} \times \dots \times \ka_{\U,n}(\det(X)).\] It is a certain analytic analogue of highest weight vector in algebraic irreducible representations (see section \ref{classical es}).

 We have the following construction:

\begin{lem}
\label{highest weight construction}
for any $Z \in M_{n,1}(\mathscr{O}_{\C_p})$ and $\mu \in D^{r}_{\ka_{\U}}(\tt,R_{\U})$, we associate the following function $f_{\mu,Z} \in C_{\ka_{\U}}^{r-an}(Iw_{M},\C_p \widehat{\otimes} R_{\U})$ to them: \[f_{\mu,Z}: \gamma \mapsto \int_{\tt}e_{hwt,\ka_{\U}}(w_0 \gamma^{-1}J(Z,ti) w_0)d\mu.\]

\end{lem}

\begin{proof}

This lemma follows from the following straightforward observations:

(1) For any $Z \in M_{n,1}(\mathscr{O}_{\C_p})$ and $\gamma \in Iw_{M}$, the function \[ti=(t_b,t_d)\mapsto e_{hwt,\ka_{\U}}(w_0 \gamma^{-1}(Zt_b+t_d)w_0) \] lies in $A^{r}_{\ka_{\U}}(\tt,R_{\U})$.

(2) For any $\gamma \in Iw_{M}$ and $b \in B_{M}(\Z_p)$, we have \[f_{\mu,Z}(\gamma b)=\ka_{\U}^{\vee}(b)f_{\mu,Z}(\gamma).\]

\end{proof}

 Now we can construct the desired map $\eta_{\ka_{\U}}:\mathscr{OD}^{r}_{\ka_{\U}}\rightarrow \widehat{\underline{\omega}}_{w}^{\ka_{\U}}$ between sheaves on the pro-etale site $\X_{Iw^+,proet}$. It is enough to construct functorial maps $\mathscr{OD}^{r}_{\ka_{\U}}(\V)\rightarrow \widehat{\underline{\omega}}_{w}^{\ka_{\U}}(\V)$ for any affinoid perfectoid $\V$ in $\X_{Iw^+,w,proet}$.

 Recall the coordinate $\mathfrak{z}$ on $X_{\infty,w}$ and consider the following map \[D_{\ka_{\U}}^{r,\circ}(\tt,R_{\U})\widehat{\otimes}\widehat{\mathscr{O}}_{\X_{Iw^+,proet}}(\V_{\infty})\rightarrow C_{\ka_{\U}}^{w-an}(Iw_M,\widehat{\mathscr{O}}_{\X_{Iw^+,proet}}(\V_{\infty})\widehat{\otimes} R_{\U}),\] \[\mu \otimes \delta \mapsto \delta f_{\mu,\z}.\] It is sufficient to show that this map is $Iw^+$-equivariant (the right side is twisted action) and taking $Iw^+$-invariants will produce the desired map $\mathscr{OD}^{r}_{\ka_{\U}}(\V)\rightarrow \widehat{\underline{\omega}}_{w}^{\ka_{\U}}(\V)$ (due to lemma \ref{OD invariant} and remark \ref{twist action invariant}).

 Such \textbf{equivariant property} is a routine check:

 For any $\alpha \in Iw^+$ and $\gamma \in Iw_M$, we have \[\alpha^*(\delta)f_{\alpha\cdot \mu,\z}(\gamma)=\alpha^*(\delta)\int_{\tt}e_{hwt,\ka_{\U}}(w_0\gamma^{-1}J(\z,ti)w_0)d(\alpha\cdot\mu)\]\[=\alpha^*(\delta)\int_{\tt}e_{hwt,\ka_{\U}}(w_0 \gamma^{-1}J(\z,\alpha ti)w_0)d\mu\] \[=\alpha^*(\delta)\int_{\tt}e_{hwt,\ka_{\U}}(w_0 \gamma^{-1}J(\z,\alpha) J(\alpha \cdot \z,ti)w_0)d\mu=\alpha\cdot (\delta f_{\mu,\z}).\] Here the third equality is due to right 1-cocyle property (lemma \ref{right 1cocycle}) \[J(\z,\alpha ti)=J(\z,\alpha)J(\alpha \cdot \z,ti).\]

 What's more, the $u_{p,i}$-action on analytic distribution $D_{\ka_{\U}}^{r,\circ}(\tt,R_{\U})$ is equivalent to left multiplication by $u_{p,i}$, exactly the \textbf{same} as $Iw^+$-action, see the remark \ref{up action on symbol}. Therefore the same argument shows that this map is also $u_{p,i}$-equivariant.

 Combine everything together, we get the following composition map \[\mathfrak{OC}^{r}_{\ka_{\U},\C_p}  \cong H^n_{proet}(\X_{Iw^+}, \mathscr{OD}_{\ka_{\U}}^{r})  \xrightarrow{Res} H^n_{proet}(\X_{Iw^+,w}, \mathscr{OD}_{\ka_{\U}}^{r}) \] \[\xrightarrow{\eta_{\ka_{\U}}}H^{n}_{proet}(\X_{Iw^+,w},\widehat{\underline{\omega}}_{w}^{\ka_{\U}})\rightarrow H^0(\X_{Iw^+,w}, \underline{\omega}_{w}^{\ka_{\U}+n+1})(-n)=M^{\ka_{\U}+n+1}_{Iw^+,w}(-n).\]

 Moreover, this map is $Gal_{\Q_p}$-equivariant. For tame Hecke operators (outside $p$), it acts on both sides via the same correspondence, therefore this map is  equivariant. For Hecke operators at $p$, from the definition of $U_{p,i}$, and the equivariant property of $\eta_{\ka_{\U}}$ respect to $Iw^+$ and $u_{p,i}$ discussed above, this map is also equivariant.

 In summary, we obtain the following result:

 \begin{thm}
 \label{overconvergent es}
 There exits a Hecke and $Gal_{\Q_p}$-equivariant  \textbf{overconvergent Eichler-Shimura map} \[ES_{\ka_{\U}}: \mathfrak{OC}^{r}_{\ka_{\U},\C_p} \rightarrow  M^{\ka_{\U}+n+1}_{Iw^+,w}(-n).\]
 \end{thm}
 
 \begin{rem}
 Our methods works for more general unitary Shimura varieties with signature $(a,b)$. In that setting, the construction of lemma \ref{highest weight construction} is more subtle.  I will write these details in later papers of this series.
 \end{rem}

\subsection{Relation with the classical Eichler-Shimura map}
\label{classical es}
In this section we will deduce compatibility between overconvergent Eichler-Shimura map and classical Eichler-Shimura map. In particular, if the weight $\ka$ is classical, the image of $p$-adic Eichler-Shimura map is contained in the image of classical Eichler-Shimura map.

A classical weight is an algebraic dominant weight:   \[\ka=(\ka_1,\cdots,\ka_n)\in X^*(T)=\Z^{n},  \ka_1...\geq \ka_n.\] As we work with "reduced" weight, we will  view $(0,\ka_1,\cdots,\ka_n)$ as the corresponding weight for $T$. To produce an irreducible algebraic representation for $GL_{n+1}$, the necessary and sufficient condition is that \[\ka_1 \geq \cdots \ka_n \geq 0= \ka_0.\] And the weight $(\ka,0)$ will be the corresponding \textbf{representation weight}.

\begin{rem}

Here is a subtle issue about weights. In the definition of "analytic  distribution" (indeed algebraic version)  for $GL_{n+1}$, we will use $ (0,w_0(\ka))$ while the resulting $GL_{n+1}$-representation  has the representation weight $(\ka,0)$. Implicitly there is a  "\textbf{switch}" process. This hidden fact is due to basic behaviour of Hodge-Tate period map in the viewpoint of  representation theory for $GL_{n+1}$ and $GL_n$. This issue doesn't arise in the  Siegel case treated in  \cite{drw}, since $Sp_{2n}$ is  self-dual,  whereas $GL_N$ is not. This leads a crucial difference between our construction and theirs.  I will make this subtle fact more explicitly in the case of $GU(a,b)$ (later papers of this series).

\end{rem}

Set $\widetilde{\ka}=(\ka,0)$. Recall the classical highest weight  representation $\Vc_{\widetilde{\ka}}$ (see section \ref{flag variety}). In that section, we're using traditional convention, $V_{\widetilde{\ka}}=Ind_{B}^{G,alg}(\widetilde{\ka}^{\vee},\Q_p)$. More concretely, it is \[\Vc_{\widetilde{\ka}}=\left\{f:GL_{n+1} \rightarrow \A^1
\ \middle| \
\begin{aligned}
& f(g b)=\widetilde{\ka}^{\vee}(b)f(g), \quad \forall g \in GL_{n+1}, b \in B,\\
& \text{$f$ is an algebraic map over  $\Q_p$}.
\end{aligned}
\right \}.\] It is a finite dimensional $\Q_p$-vector space with left $G(\Q_p)$-action \[(g_2 \cdot f)(g_1)=f(g_2^{-1}g_1) \] for $g_1$, $g_2 \in G(\Q_p)$.

During our construction for overconvergent cohomology, we are using other conventions (opposite Borel subgroup $B^{opp}$ and distribution). Now we rewrite $\Vc_{\widetilde{\ka}}$ in this convention, as a kind of dual space. Recall the convention $\widetilde{w_0}(\ka)=(0,\ka_n,\cdots,\ka_1)$, then the dual space of $\Vc_{\widetilde{\ka}}$ is exactly \[Ind_{B}^G((\ka,0),\Q_p)\cong Ind_{B^{opp}}^G(\widetilde{w_0}(\ka),\Q_p)\].  In particular, our convention in weights for analytic distribution is \textbf{compatible} with usual notations in classical representation theory.

What's more, during the construction of overconvergent Eichler-Shimura map, we have constructed $e_{hwt,\ka}$. When the weight $\ka$ is classical, it naturally extends to a function on $M_{n}(\C_p)$: For any matrix $X=(X_{i,j})\in M_{n}(\C_p)$, we define \[e_{hwt,\ka}(X):=X_{1,1}^{\ka_1-\ka_2}\times \det((X_{i,j})_{1\leq i,j \leq 2})^{\ka_2-\ka_3}\times...\times \det(X)^{\ka_n}.\] And it  is  the \textbf{highest weight vector} (respect to the \textbf{opposite Borel subgroup} $B_M$) in the representation $Ind_{B_M}^{M}(\ka,\Q_p)$ (its representation weight is $\ka^{\vee}$) of $GL_{n}$. We will also apply $e_{hwt,\ka}$ to study classical Eichler-Shimura map.

\begin{rem}

We stressed that in this section when we say highest weight vector, it is always respect to the \textbf{opposite} Borel subgroup. It is the unique eigenspace of the opposite Borel subgroup.   
\end{rem}

First we introduce the classical etale cohomology defined by these highest weight representations. The left action $G(\Q_p)$ on $\Vc_{\ka}$ induces etale $\Q_p$-local system on $X_{Iw^+}$ which we still denote by the same symbol. Moreover, we have the following natural isomorphisms \[H_{et}^*(\X_{Iw^+},\Vc_{\ka})\cong H_{et}^*(X,\Vc_{\ka})\cong H^*(X_{Iw^+}(\C),\Vc_{\ka}).  \]  And similarly there are Hecke operators (tame $T_{\gamma}$) and $U_{p,i}$ acting on them.

Then similar to the discussion of overconvergent cohomology, we will introduce pro-etale sheaves to study these classical cohomology. We define $\Ol \Vb_{\ka} $ on $\X_{Iw^+,proet}$ \[\Ol\Vb_{\ka}:=v^{-1}\Vc_{\ka}\otimes \widehat{\Ol}_{\X_{Iw,proet}}.\] Similarly we have the following natural isomorphism \[H_{et}^*(\X_{Iw^+},\Vc_{\ka})\otimes \C_p \cong H^*_{proet}(\X_{Iw^+},\Ol \Vb_{\ka}).\] What's more, let $\V \rightarrow \X_{Iw^+}$ be an affinoid perfectoid object in $\X_{Iw^+,proet}$ and set \[\V_{\infty}:=\V_{\infty}\times_{\X_{Iw^+}}\X_{\infty},\] we have similar  relation \[ \Ol\Vb_{\ka}(\V)=(\Vc_{\ka}\otimes \widehat{\Ol}_{\X_{Iw,proet}}(\V_{\infty}))^{Iw^+}.\]

On the other hand we similarly introduce the $p$-adically completed automorphic sheaf on $\X_{Iw^+,proet}$ defined by \[\underline{\omega}_{Iw^+}^{\ka}:=\varprojlim_{m}(\underline{\omega}_{Iw^+}^{\ka,+}\otimes _{\mathscr{O}_{\X_{Iw^+}}} \mathscr{O}_{\X_{Iw^+,proet}/p^{m}})[\frac{1}{p}].\] And we also have the analogue (proposition \ref{classical form}):

\[\widehat{\underline{\omega}}_{Iw^+}^{\ka}(\V)=\left\{f \in P_{\ka}(GL_n, \widehat{\Ol}_{\X_{\infty,w}}(\V_{\infty})) :\begin{aligned}
&\gamma^*f=\rho_{\ka}(\z\gamma_b+\gamma_d)f,\\
& \forall \gamma=\begin{pmatrix}\gamma_a & \gamma_b \\ \gamma_c & \gamma_d  \end{pmatrix}\in Iw^+.
\end{aligned}\right\}.\]

\textbf{The classical Eichler-Shimura map} We first show that Hodge-Tate period map has nice properties in representation theory.

Let $V_{std,n}$ denote the standard representation of $GL_n$, which has \textbf{induction weight} $(0,\cdots,0,-1)$ (\textbf{representation weight} is $(1,0,\cdots,0)$). It has standard basis $\{\widetilde{e_1},\cdots,\widetilde{e_n}\}$.  And use  $\{e_0,...,e_n\}$  to denote the standard basis for $V_{std,n+1}$. Consider the $GL_{n}$-equivariant map \[\delta:V_{std,n+1}\rightarrow V_{std,n},\]\[e_i \mapsto \widetilde{e_i} (i>0); e_0 \mapsto 0. \] This map sends the highest weight vector (for opposite Borel subgroup $B^{opp}$) $e_n$ to highest weight vector (for opposite Borel subgroup $B_M^{opp}$) $\widetilde{e_n}$. Moreover, for any $1 \leq k \leq n$, the induced map $\wedge^{k} V_{std,n+1} \rightarrow \wedge^{k}V_{std,n}$ has the property \[e_{n}\wedge e_{n-1} \cdots \wedge e_{n-k+1} \mapsto \widetilde{e_{n}}\wedge \widetilde{e_{n-1}} \cdots \wedge \widetilde{e_{n-k+1}}.\] For any classical weight $\ka=(\ka_1,\cdots,\ka_n)$, its corresponding algebraic irreducible representation $\Vc_{\ka}$ ($(\ka,0)$ is \textbf{representation weight}) is a direct summand of the $GL_{n+1}$-representation \[V_{std}^{\ka}:=(Sym^{\ka_{1}-\ka_{2}}V_{std,n+1})\otimes (Sym^{\ka_{2}-\ka_{3}}(\wedge^2 V_{std,n+1}))\otimes \cdots \otimes (Sym^{\ka_{n}}(\wedge^n V_{std,n+1})).\] The weight $(\ka_{1},...,\ka_{n})$ also produces an algebraic irreducible representation $\Vc_{\ka,n}$ (with \textbf{representation weight} $\ka$) for $GL_{n}$. Similarly $\Vc_{\ka,n}$  is a direct summand of $V_{std,n}^{\ka}$ (analogue of $V_{std}^{\ka}$). The map $\delta$ induces a map $V_{std}^{\ka}\rightarrow V_{std,n}^{\ka}$, which further produces a $GL_n$-equivariant surjection  $\Vc_{\ka}\twoheadrightarrow \Vc_{\ka,n}$, and the highest weight vector (for $B^{opp}$) is sent to the highest weight vector (for $B_M^{opp}$).

This discussion fits into the picture of Hodge-Tate period map. Recall the map \[\Ol_{\X_{\infty}}^{n+1}\rightarrow \underline{\omega}_{{\infty}},\] it exactly corresponds to the map $\delta: V_{std,n+1}\rightarrow V_{std,n}$ in representation theory. And similar to $\Vc_{\ka}$, the $GL_{n+1}$ representation $V_{std}^{\ka}$ also produces etale local system and related sheaf in pro-etale topology. The $GL_n$-representation $V_{std,n}^{\ka}$ correspondences to automorphic bundle \[\underline{\omega}_{Iw^+,std}^{\ka}:=(Sym^{\ka_{1}-\ka_{2}}\underline{\omega}_{Iw^+})\otimes (Sym^{\ka_{2}-\ka_{3}}(\wedge^2 \underline{\omega}_{Iw^+}))\otimes \cdots \otimes (Sym^{\ka_{n}}(\wedge^n \underline{\omega}_{Iw^+})).\] It has a natural surjection to $\underline{\omega}_{Iw^+}^{\ka}$. Put everything together, we obtain the comparison map \[\eta_{\ka}^{alg}: \Ol \Vb_{\ka}\rightarrow \underline{\omega}_{Iw^+}^{\ka}.\] The map implies the classical Eichler-Shimura map of weight $\ka$ \[
ES_{\ka}^{alg}: \xymatrix{
H^n(\X_{Iw^+}, \Vc_{\ka}) \otimes_{\Q_p}\C_p  \ar[r]^(0.5){\cong} & H^{n}_{proet}(\X_{Iw^+}, \Ol \Vb_{\ka}) \ar[ld]_{\eta_{\ka}^{alg}} & \\
 H^n_{proet}(\X_{Iw^+}, \underline{\omega}_{Iw^+}^{\ka} ) \ar[r] & H^0(\X_{Iw^+}, \underline{\omega}_{w}^{\ka+n+1})(-n).
}
\]This map coincides with the classical construction in theorem VI 6.2 of \cite{faltingschai}. It is a Hecke and $Gal_{\Q_p}$   equivariant surjection.
  
What's more, when restrict to $\X_{Iw^+,w}$, we can make this map more explicitly in a way similar to overconvergent Eichler-Shimura map.

Consider a matrix (column vector) $Z \in M_{n,1}(\C_p)$, $g=\begin{pmatrix} g_a & g_b \\ g_c & g_d \end{pmatrix} \in GL_{n+1}(\Q_p)$ and $\gamma \in GL_{n}(\Q_p)$, recall $J(Z,g)=Z g_b+g_d$.  Similar to overconvergent Eichler-Shimura map, we obtain the following construction:

For  any $\mu \in \Vc_{\ka}\otimes \C_p =(Ind_{B^{opp}}^G(\widetilde{w_0}(\ka),\C_p))^{\vee} $, we can  associate the following element $f_{\mu,Z} \in P_{\ka}(GL_n,\C_p)$ to them: \[f_{\mu,Z}^{alg}: \gamma \mapsto \int_{G(\Q_p)}e_{hwt,\ka}(w_0 \gamma^{-1}J(Z,g)w_0)d\mu.\] This is well defined  because $e_{hwt,\ka}$ extends naturally to $M_{n}(\C_p)$ and we have the same observations in lemma \ref{highest weight construction}.

In conclusion, we obtain the following proposition:

\begin{prop}

(1) Let $V$ be an affinoid perfectoid object in $\X_{Iw^+,w,proet}$ and $\V_{\infty}$ is the base change to $\X_{\infty,w,proet}$. There is a well-defined $Iw^+$-equivariant map
\[\widetilde{\eta}_{\ka}^{alg}: \Vc_{\ka} \otimes \mathscr{O}_{\X_{\infty,w,proet}}(\V_{\infty}) \rightarrow P_{\ka}(GL_n, \mathscr{O}_{\X_{\infty,w,proet}}(\V_{\infty}))\] defined by \[\mu \otimes \varphi \mapsto \varphi f_{\mu,\z}^{alg},\] \[f_{\mu,\z}^{alg}(\gamma)=\int_{G(\Q_p)}e_{hwt,\ka}(w_0 \gamma^{-1}J(\z,g)w_0)d\mu.\]

(2) The map $\eta_{\ka}^{alg}$ is obtained from $\widetilde{\eta}_{\ka}^{alg}$ via taking $Iw^+$-invariants.
\end{prop}

\begin{proof}

For the first statement, we only remains to check the $Iw^+$-equivariant property, which follows from the same computations in the construction of overconvergent Eichler-Shimura map.

For the second statement, we observe that over the point $\z=0$, the construction of $\widetilde{\eta}_{\ka}^{alg}$ maps the highest weight vector for $B^{opp}(\Q_p)$ to highest weight vector for $B_{M}^{opp}(\Q_p)$, which is the same property of $\eta_{\ka}^{alg}$. Further apply the $Iw^+$-equivariant property of $\widetilde{\eta}_{\ka}^{alg}$ we finish the proof.

\end{proof}

Finally we deduce  the following compatibility:

\begin{thm}
\label{thm p adic and classical es}
Let $\ka$ be a classical weight, then the overconvergent Eichler-Shimura map $ES_{\ka}$ factors through the classical Eichler-Shimura map. In particular, the image of $ES_{\ka}$ is contained in the finite dimensional $\C_p$-vector  space of classical forms $M^{\ka+n+1}_{Iw^+}(-n)$. 
\end{thm}

\begin{proof}

There is a natural inclusion \[Ind_{B^{opp}(\Q_p)}^{G(\Q_p),alg}(\widetilde{w_0}(\ka),\Q_p) \hookrightarrow  A_{\ka}^{r}(\tt,\Q_p).\] Taking dual we obtain \[D_{\ka}^{r}(\tt,\Q_p)\twoheadrightarrow  \Vc_{\ka},\] which further induces a map between sheaves \[\Ol \mathscr{D}^{r}_{\ka} \rightarrow \Ol \mathscr{V}_{\ka}.\] Moreover, we have the following commutative diagram

\[
\xymatrix{
\Ol \mathscr{D}_{\ka}^{r} \ar[r]^{\eta_{\ka}} \ar[d] & \widehat{\underline{\omega}}_{w}^{\ka}   \\
\Ol \mathscr{V}_{\ka} \ar[r]^{\widetilde{\eta}_{\ka}^{alg}} &  \widehat{\underline{\omega}}_{Iw^+}^{\ka}  \ar[u]
}
\]

As both overconvergent Eichler-Shimura map and classical Eichler-Shimura map can be carried out by the same process, and recall the injective map \[H^0(\X_{Iw^+}, \underline{\omega}_{Iw^+}^{\ka}) \hookrightarrow H^0(\X_{Iw^+,w}, \underline{\omega}_{Iw^+}^{\ka} ),\]put everything together, we get the desired commutative diagram:

\[
\xymatrix{
H^n_{proet}(\X_{Iw^+},\Ol \mathscr{D}_{\ka}^{r}) \ar[r]^{ES_{\ka}} \ar[d] & H^0(\X_{Iw^+,w},  \underline{\omega}_{w}^{\ka})   \\
H^n_{proet}(\X_{Iw^+},\Ol \mathscr{V}_{\ka}) \ar[r]^{ {ES}_{\ka}^{alg}} &  H^0(\X_{Iw^+}, \underline{\omega}_{Iw^+}^{\ka})  \ar[u]
}
\]
\end{proof}

\begin{rem}

A simple corollary is that the kernel of $ES_{\ka}$ is   infinite dimensional. Its deeper structure, which is further related with   higher Coleman theory,  will be explored in a subsequent paper in this series.
\end{rem}

\section{The Eichler-Shimura map over the eigenvariety}
\label{section6}
Previously we have constructed overconvergent Eichler-Shimura map. Indeed it is functorial respect to weights and in this section we will glue it over the eigenvariety.

\subsection{Construction of the eigenvariety}
In this section we  quickly recall the construction of the eigenvariety from the overconvergent cohomology, following the standard methods in \cite{hansen2017universal} and \cite{chj}.

\begin{defn}
(1) Let $(R_{\U},\ka_{\U})$ be a \textbf{small weight}, we call it is \textbf{open} if the natural map  \[\U^{rig}=Spa(R_{\U},R_{\U})\rightarrow \W \]is an open immersion.

(2) Let $(R_{\U},\ka_{\U})$ be an \textbf{affinoid weight}, we call it is \textbf{open} if the natural map  \[\U^{rig}=Spa(R_{\U},R_{\U}^{\circ})\rightarrow \W \]is an open immersion.
\end{defn}

  Let $(R_{\U},\ka_{\U})$ be an open weight with an integer $r>1+r_{\U}$. we can consider the \textit{Borel-Serre chain complex } $C_{\bullet}(Iw^+,A_{\ka_{\U}}^{r}(\tt,R_{\U}))$ (resp. \textit{Borel-Serre cochain complex} $C^{\bullet}(Iw^+,D_{\ka_{\U}}^{r}(\tt,R_{\U}))$) which is constructed from the Borel-Serre compactification of the locally symmetric space $X_{Iw^+}(\C)$ and computes  Betti \textit{homology} $H_*(X_{Iw^+}(\C),A_{\ka_{\U}}^{r}(\tt,R_{\U}))$ (resp. Betti \textit{cohomology} $H^*(X_{Iw^+}(\C),D_{\ka_{\U}}^{r}(\tt,R_{\U}))$). See \cite{hansen2017universal} section 2 for more details. We define \[C^{\ka_{\U},r}_{tot}:=\bigoplus_{i}C_i(Iw^+,A_{\ka_{\U}}^{r}(\tt,R_{\U})),\] \[C^{tot}_{\ka_{\U},r}:=\bigoplus_{i}C^i(Iw^+,D_{\ka_{\U}}^{r}(\tt,R_{\U})).\]

Due to \cite{hansen2017universal} section 2.2, $A_{\ka_{\U}}^{r}(\tt,R_{\U})$ is \textit{orthonormalizable} thus $C^{\ka_{\U},r}_{tot}$ is also orthonormalizable $R_{\U}[\frac{1}{p}]$-module. And there are natural Hecke operators on it and the action of $U_p$ (Hecke operator at $p$) is compact. Let $F_{\ka_{\U},r}^{oc}\in R_{\U}[\frac{1}{p}][[T]]$ be the resulting Fredholm determinant of $U_p$ acting on it. Then for any $h \in \Q_{\geq0}$, the existence of a slope-$\leq h$ decomposition of $C^{\ka_{\U},r}_{tot}$ is equivalent to the existence of  a slope-$\leq h$ factorization of $F_{\ka_{\U},r}^{oc}$. Moreover, such slope truncation has certain functorial property. If $C^{\ka_{\U},r}_{tot,\leq h}$ is the slope-$\leq h$ submodule of $C^{\ka_{\U},r}_{tot}$ and let $\U_1=(R_{\U_1},\ka_{\U_1})$ be another open weight with $\U_1^{rig}\subset \U^{rig}$, then there is a canonical isomorphism \[ C^{\ka_{\U},r}_{tot,\leq h} \otimes_{R_{\U}[\frac{1}{p}]} R_{\U_1}[\frac{1}{p}] \cong C^{\ka_{\U_1},r}_{tot,\leq h}.\]

\begin{defn}
Let $\U=(R_{\U},\ka_{\U})$ be an open weight and let $h \in \Q_{\geq 0}$. Then pair $(\U,h)$ is called a \textbf{ slope datum }if $F_{\ka_{\U},r}^{oc}$ admits a slope-$\leq h$ factorization.
\end{defn}

According to \cite{chj} proposition 3.3 and 3.4, we have the following result concerning slope truncation of  Betti homology (cohomology) for the Shimura variety:

\begin{prop}
\label{slope truncation functorial}
Let $(\U,h)$ be a slope datum and $(R_{\U_1},\ka_{\U_1})$ be another open weight with $\U_1^{rig} \subset \U^{rig}$.

(1) There is a canonical isomorphism \[H_*(X_{Iw^+}(\C),A_{\ka_{\U}}^{r}(\tt,R_{\U}))_{\leq h}\otimes_{R_{\U}[\frac{1}{p}]} R_{\U_1}[\frac{1}{p}]\cong H_*(X_{Iw^+}(\C),A_{\ka_{\U_1}}^{r}(\tt,R_{\U_1}))_{\leq h}.\]

(2) The dual side, $C^{tot}_{\ka_{\U},r}$ and $H^*(X_{Iw^+}(\C),D_{\ka_{\U}}^{r}(\tt,R_{\U}))$, also admit slope-$\leq h$ decomposition, and such slope truncation have similar functorial property as above.
\end{prop}

What's more, we can vary the open weight $\U$ and glue the Fredholm determinant  into an entire power series $F_{W}^{oc}\in \mathscr{O}_{\W}(\W)\{\{T\}\}$. It is independent of radius $r$ (see \cite{hansen2017universal} proposition 3.1.1).

On the other hand, we can also consider the space of overconvergent automorphic forms $M_{Iw^+,w}^{\ka_{\U}}=H^0(\X_{Iw^+,w},\underline{\omega}^{\ka_{\U}}_{w})$. It has property \textit{(Pr)} in the sense of \cite{buzzard2007} and also has Hecke actions. The $U_p$ operator is also compact. Then   we obtain the Fredholm determinant $F^{mf}_{\ka_{\U},w}$. Similarly we can glue them together when vary affinoid weights. Further taking limit over $w$ we get the entire power series $F^{mf}_{\W}\in\mathscr{O}_{\W}(\W)\{\{T\}\}\widehat{\otimes}\C_p$.

Now we proceed to construct \textbf{spectral variety}, which is the first step to construct eigenvariety. See \cite{ludwig2024spectral} for more details about the general machine of eigenvariety.

Define \[F_{\W}:=F_{\W}^{oc}F_{\W}^{mf}\in \mathscr{O}_{\W}(\W)\{\{T\}\}\widehat{\otimes}\C_p,\] which is still a Fredholm series. Let $\A_{\W}^1:=\W \times_{Spa(\Q_p,\Z_p)}\A_{\C_p}^1$. The \textbf{spectral variety} $\mathcal{S}$ is the zero locus of $F_{\W}$ inside $\A_{\W}^1$.

\begin{defn}
Let $\U$ be an open weight and $\U_{\C_p}^{rig}$ denote the base change to $Spa(\C_p,\mathscr{O}_{\C_p})$. Define the adic line $\A_{\U}^1:=\U^{rig}\times \A^1_{\C_p}$. For any $h \in \Q_{>0}$, consider the closed ball $\textbf{B}(0,p^{h})$ of radius $p^{h}$ over $\C_p$ and define $\textbf{B}_{\U,h}:=\U^{rig}\times \textbf{B}(0,p^h)$. Set $\mathcal{S}_{\U,h}:=\mathcal{S} \cap \textbf{B}_{\U,h}$. We call the pair $(\U,h)$ is \textbf{slope-adapted} if the natural map $\mathcal{S}\rightarrow \U_{\C_p}^{rig}$ is finite flat.
\end{defn}

Consider the collection \[Cov(\mathcal{S})=\{\mathcal{S}_{\U,h}: (\U,h)\text{ is slope-adapted} \}.\] Let $Cov_{aff}(\mathcal{S})$ be a sub-collection, consisting of $\mathcal{S}_{\U,h}$ with $\U$ being an affinoid weight. By \cite{hansen2017universal} proposition 4.14, $Cov_{aff}(\mathcal{S})$ forms an open cover for $\mathcal{S}$. Use this cover, we define the following coherent sheaves on $\mathcal{S}$:

\begin{defn}
(1) Recall the limit $D^{\dag}_{\ka_{\U}}(\tt,R_{\U})$, the coherent sheaf $\mathscr{H}^{tot}$ on $\mathcal{S}$ is defined as follow:

\[\mathscr{H}^{tot}(\mathcal{S}_{\U,h}):=(\bigoplus_{i}H^i(X_{Iw^+}(\C),D^{\dag}_{\ka_{\U}}(\tt,R_{\U}))^{\leq h})\widehat{\otimes} \C_p.\]

(2) On the other hand, the coherent sheaf $\mathscr{F}^{\dag}_{Iw^+}$ on $\mathcal{S}$ is defined by \[\mathscr{F}^{\dag}_{Iw^+}(\mathcal{S}_{\U,h}):=M_{Iw^+}^{\ka_{\U}+n+1,\leq h}.\]
\end{defn}

The Hecke algebra $\mathbb{T}$ (see section \ref{hecke operator}) acts on both of the coherent sheaves and we get two eigenvarieties:

\begin{defn}
(1) Let $\mathbb{T}^{oc}_{\U,h}$ be the reduced $\mathscr{O}_{\mathcal{S}_{\U,h}}(\mathcal{S}_{\U,h})$-algebra generated by image of $\mathbb{T}\rightarrow End(\mathscr{H}^{tot}(\mathcal{S}_{\U,h}))$, and the resulting coherent sheaf on $\mathcal{S}$ denoted by $\mathscr{S}_{oc}$. Let $\mathbb{T}^{oc,\circ}_{\U,h}$ be the integral closure of $\mathscr{O}_{\mathcal{S}_{\U,h}}(\mathcal{S}_{\U,h})^{\circ}$ inside $\mathbb{T}^{oc}_{\U,h}$, and the resulting coherent sheaf is denoted by $\mathscr{S}_{oc}^{\circ}$.

(2) \textbf{The reduced eigenvariety} $\mathcal{E}^{oc}$ is defined to be the relative adic space $Spa_{\mathcal{S}}(\mathscr{S}_{oc},\mathscr{S}_{oc}^{\circ})$.

\end{defn}

\begin{defn}
(1) Let $\mathbb{T}^{mf}_{\U,h}$ be the reduced $\mathscr{O}_{\mathcal{S}_{\U,h}}(\mathcal{S}_{\U,h})$-algebra generated by image of $\mathbb{T}\rightarrow End(\mathscr{F}^{\dag}_{Iw^+}(\mathcal{S}_{\U,h}))$, and the resulting coherent sheaf on $\mathcal{S}$ denoted by $\mathscr{S}_{mf}$. Let $\mathbb{T}^{oc,\circ}_{\U,h}$ be the integral closure of $\mathscr{O}_{\mathcal{S}_{\U,h}}(\mathcal{S}_{\U,h})^{\circ}$ inside $\mathbb{T}^{mf}_{\U,h}$, and the resulting coherent sheaf is denoted by $\mathscr{S}_{mf}^{\circ}$.

(2) \textbf{The equidimensional eigenvariety }$\mathcal{E}^{mf}$ is defined to be the equidimensional locus inside the relative adic space $Spa_{\mathcal{S}}(\mathscr{S}_{mf},\mathscr{S}_{mf}^{\circ})$.

\end{defn}

\begin{rem}

Previously, through the methods of \cite{AIP2015}, Xu Shen constructed the eigenvariety in \cite{shen2016}, which is equidimensional. Our $\mathcal{E}_{0}^{mf}$ is the base change of his construction to $\C_p$ (with reduced weights, i.e. $n$-dimensional).

\end{rem}

The later eigenvariety is  "smaller" than the previous one:

\begin{prop}
There is a natural closed embedding $\mathcal{E}^{mf}\hookrightarrow \mathcal{E}^{oc}$.
\end{prop}

The proof is the same as proposition 6.2.6 in \cite{drw}. The key ideas is to apply classicality results (small slope implies classicality) and $p$-adic functoriality (theorem 5.1.2 in \cite{hansen2017universal}).

\begin{rem}
In next paper of this series,  we will further consider the higher Coleman theory  (see \cite{bp2021higher}) and put all degree together, then the resulting eigenvariety should be the same as the eigenvariety produced from overconvergent cohomology.
\end{rem}

We identify $\mathcal{E}^{mf}$ with its image inside $\mathcal{E}^{oc}$ and denote it by $\mathcal{E}$. It is the desired \textbf{eigenvariety} and is equipped with the \textbf{weight map} $wt:\mathcal{E} \rightarrow \mathcal{W}$ and the natural map $\pi:\mathcal{E} \rightarrow \mathcal{S}$.

In summary, we have the following commutative diagram: \[\xymatrix{
\mathcal{E} \ar[r]^{\pi} \ar[dr]^{wt} & \mathcal{S} \ar[d]^{wt_{\mathcal{S}}}\\
& \W}.\]

\subsection{Comparison map between sheaves on the eigenvariety}
In this section we glue the $p$-adic Eichler-Shimura map into a comparison map between corresponding coherent sheaves on the eigenvariety $\mathcal{E}$.

Recall the overconvergent cohomology group \[\mathfrak{OC}^{r}_{\ka_{\U},\C_p}=H^n(X_{Iw^+}(\C),D_{\kappa_{\mathcal{U}},j}^{r}(\tt, R_{\mathcal{U}}))\widehat{\otimes}\C_p,\] we define the limit \[\mathfrak{OC}^{\dag}_{\ka_{\U},\C_p}:=\varprojlim_{r}\mathfrak{OC}^{r}_{\ka_{\U},\C_p}.\]

Let $(R_{\U},\ka_{\U})$ be a small open weight and recall the overconvergent Eichler-Shimura map \[ES_{\ka_{\U}}:\mathfrak{OC}^{r}_{\ka_{\U},\C_p} \rightarrow  M^{\ka_{\U}+n+1}_{Iw^+,w}(-n).\] If $(\U,h)$ is slope-adapted, then the Hecke equivariant property of $ES_{\ka_{\U}}$ induces a $\C_p\widehat{\otimes} R_{\U}$-linear map on the slope truncation: \[ES_{\ka_{\U}}^{\leq h}: \mathfrak{OC}^{r,\leq h}_{\ka_{\U},\C_p} \rightarrow  M^{\ka_{\U}+n+1,\leq h}_{Iw^+,w}(-n).\] Because the slope truncation is functorial respect to weight (see proposition \ref{slope truncation functorial}), the comparison map $ES_{\ka_{\U}}^{\leq h}$ is also functorial for weights.

For each $\mathcal{S}_{\U,h}$, let $\mathcal{E}_{\U,h}$ be the preimage under $\pi:\mathcal{E} \rightarrow \mathcal{S}$. On the eigenvariety $\mathcal{E}$, we consider the following two coherent sheaves: \[\mathscr{OC}^{\dag}(\mathcal{E}_{\U,h}):= \mathfrak{OC}^{r,\leq h}_{\ka_{\U},\C_p}\] and \[\mathscr{F}^{\dag}_{Iw^+}(-n)(\mathcal{E}_{\U,h}):=M^{\ka_{\U}+n+1,\leq h}_{Iw^+,w}(-n).\]

Because $ES_{\ka_{\U}}^{\leq h}$ is  functorial for weights and we can glue them. In conclusion we obtain the following result:

\begin{thm}
\label{es over eigenvariety}
There is a comparison map \[\mathcal{ES}:\mathscr{OC}^{\dag} \rightarrow \mathscr{F}^{\dag}_{Iw^+}(-n)\] between coherent sheaves over $\mathcal{E}$ such that for each slope-adapted pair $(\U,h)$, the map $\mathcal{ES}(\mathcal{E}_{\U,h})$ is the overconvergent Eichler-Shimura map $ES_{\ka_{\U}}^{\leq h}$.

\end{thm}

We can say a little more about this map $\mathcal{ES}$. Denote by $\mathscr{I}m$ and $\mathscr{K}er$ the image and the kernel of $\mathcal{ES}$. We get the following exact sequence over $\mathcal{E}$: \[0\rightarrow \mathscr{K}er \rightarrow \mathscr{OC}^{\dag} \rightarrow  \mathscr{I}m \rightarrow 0,\] and the following result: 

\begin{prop}
Let $E^{lof}$ be the locus of $E$ where $\mathscr{K}er$, $\mathscr{OC}^{\dag}$ and $\mathscr{I}m$ are locally free. Then the above exact sequence splits locally over $E^{lof}$.
\end{prop}

See theorem 6.3.2 in \cite{drw} for the proof. The key technique is to apply Sen operators in $p$-adic Galois representations to split this sequence (due to proposition 2.3 in \cite{Kisin2003}).

\begin{rem}
In next paper of this series, I will further study $\mathscr{K}er$. It is related with higher Coleman theory, just like \cite{diao2025overconvergent} in $GSp_4$ case. And such $p$-adic Eichler-Shimura theory has applications to ramification locus of the weight map.
\end{rem}

\section{Further development}
\label{section7}

\subsection{General unitary Shimura varieties}
The results of this paper admit several natural generalizations, which we plan to pursue in future papers of this series.

One direction is to establish a refined $p$-adic Eicher-Shimura comparison theory. As mentioned earlier, the kernel of the overconvergent Eichler-Shimura map $ES_{\ka_U}$ is "large". In the envisioned "full" $p$-adic Eichler-Shimura theory, the overconvergent cohomology should admit a filtration whose graded pieces are related to higher Coleman theory introduced in \cite{bp2021higher}.  In a subsequent paper, we aim to develop this refined theory and generalize the results of \cite{diao2025overconvergent} (on $GSp_4$) to our unitary setting.

Another natural direction is to extend our results to more general unitary Shimura varieties, including those of signature $(a,b)$  and allowing for non-compact situations. The main ideas of this paper carry over, but they require substantial technical modifications. The case of general signature is significantly more subtle and involves heavier notation; in particular, a delicate "\textit{switch}" process (see section \ref{classical es})  must be handled carefully. Regarding compactifications, we can draw on the general results of \cite{drw}. These technical challenges, while non-trivial, can be overcome with additional work. We plan to address these details in  later papers of this series.

\subsection{Other arithmetic applications}

As mentioned in the introduction, our work is motivated by the study of eigenvarieties and $p$-adic $L$-functions.  

One important application is to the ramification locus of the weight map on the eigenvariety-a fundamental question concerning the geometry of eigenvarieties  (see, e.g., \textit{open problem} 1 in section 8.3 of \cite{AIP2015}). Through classical Eichler-Shimura theory, in  \cite{Bellaiche2021} Bellaiche related such ramification locus of eigencurve to $p$-adic adjoint $L$-function for modular forms. More recently, Diao, Rosso and Wu used their $p$-adic Eichler-Shimura theory for $GSp_4$ to obtain certain etale points for the eigenvariety (see corollary 1.2.5 of \cite{diao2025overconvergent}).  In a subsequent  paper, we aim to establish analogous results in our unitary setting.

On the other hand, $p$-adic Eichler-Shimura theory also has significant arithmetic applications to $p$-adic $L$-functions and related topics such as the Bloch-Kato conjectures and Iwasawa theory.

For example, in his ICM 2022 report \cite{loeffler2022icm} (see in particular Section 7), Loeffler outlines an approach to the Bloch-Kato conjectures for $GSp_4$ that involves two types of $p$-adic $L$-functions: one arising from the etale cohomology of Shimura varieties, and the other from their coherent cohomology. A key ingredient is the comparison between these two $p$-adic $L$-functions, which requires a suitable $p$-adic Eichler-Shimura  map and is often beyond the natural scope of classical comparison theory.   Another illustrative example is the work on $p$-adic Artin formalism \cite{Buyuk2025} (see in particular section 3.7), whereas $p$-adic Eichler-Shimura maps  again plays an essential role by interpolating classical comparison maps into a $p$-adic family version.

 These examples suggest that our results may open the door to analogous arithmetic applications for unitary groups, particularly in the context of $p$-adic $L$-functions.  We plan to explore these directions in the future.

\bibliographystyle{plain}
\bibliography{reference}

\end{document}